\documentclass[11pt,reqno]{amsart}
\usepackage{t1enc}
\usepackage[latin1]{inputenc}
\usepackage{amsmath,latexsym,amssymb,graphicx,dsfont,amsthm,amsfonts}
\usepackage{hyperref}
\usepackage{color}
\usepackage{tikz}
\usetikzlibrary{arrows.meta}
\usepackage{yfonts}
\usepackage{mathrsfs}
\usepackage{umoline,comment}

\newtheorem{theorem}{Theorem}[section]
\newtheorem{lemma}[theorem]{Lemma}
\newtheorem{corollary}[theorem]{Corollary}
\newtheorem{proposition}[theorem]{Proposition}

\newtheorem{example}[theorem]{Example}

\newtheorem{remark}[theorem]{Remark}

\newcommand{\E}{\mathbb{E}}
\newcommand{\Prob}{\mathbb{P}}
\renewcommand{\P}{\mathbb{P}}

\newcommand{\calI}{\mathcal{I}}
\newcommand{\calR}{\mathcal{R}}

\newcommand{\N}{\mathbb{N}}
\newcommand{\Z}{\mathbb{Z}}

\newcommand{\W}[1]{\mathbf{W}_{#1}}

\newcommand{\e}[1]{\mathbf{e}_{#1}}

\renewcommand{\t}{\delta}

\newcommand{\TT}{\mathbf{T}}
\newcommand{\Rn}{\widetilde{\mathcal{R}}_{\textrm{norm}}}
\newcommand{\ii}{\mathbf{i}}

\numberwithin{equation}{section}

\tikzset{
    partial ellipse/.style args={#1:#2:#3}{
        insert path={+ (#1:#3) arc (#1:#2:#3)}
    }
}

\begin{document}

\title[Capacity of the Range of Rand Walks on Free Products]{Asymptotic Capacity of the Range of Random Walks on Free Products of Graphs}
\author{Lorenz A. Gilch}

\address{Lorenz A. Gilch: University of Passau, Innstr. 33, 94032 Passau, Germany}

\email{Lorenz.Gilch@uni-passau.de}
\urladdr{http://www.math.tugraz.at/$\sim$gilch/}
\date{\today}

\maketitle

\begin{abstract}
In this article we prove existence of the asymptotic capacity of the range of random walks on free products of graphs. In particular, we will show that the asymptotic capacity of the range is almost surely constant and strictly positive. Furthermore, we provide a central limit theorem for the capacity of the range and show that it varies real-analytically in terms of finitely supported probability measures of constant support.
\end{abstract}

\section{Introduction}

Suppose we are given finite or countable sets $V_1$ and $V_2$ with $|V_i|\geq 2$, $i\in\{1,2\}$, and distinguished vertices $o_1\in V_1$ and $o_2\in V_2$. The free product of $V_1$ and $V_2$ is given by $V:=V_1\ast V_2$, the set of all finite words $x_1\dots x_n$ over the alphabet $\bigl(V_1\setminus\{o_1\}\bigr) \cup \bigl(V_2\setminus\{o_2\}\bigr)$ such that no two consecutive letters $x_j$ and $x_{j+1}$ arise from the same $V_i\setminus\{o_i\}$, $i\in\{1,2\}$. Furthermore, let $P_1$ and $P_2$ be transition matrices on $V_1$ and $V_2$. Consider now a time-homogeneous,  transient Markov chain $(X_n)_{n\in\N_0}$ starting at the empty word $o$ whose transition matrix arises as a convex combination of some versions of $P_1$ and $P_2$ shifted to $V$ (see Section \ref{subsec:rw-on-fp} for the exact definition). For better visualization, we may think of graphs $\mathcal{G}_1$, $\mathcal{G}_2$ and $\mathcal{G}$ with vertex sets $V_1$, $V_2$ and $V$ such that there is an oriented edge connecting the vertices $x$ and $y$ if and only if the corresponding transition probability (w.r.t. $P_1$, $P_2$ and $P$) of going from $x$ to $y$ in one step is strictly positive. 
\par
For $A\subseteq V$, denote by $S_A:=\inf\{m\in\N \mid  X_m\in A\bigr\rbrace \in[1,\infty]$ the  first returning time to $A$. The capacity of the set $A$ is then defined as
$$
\mathrm{Cap}(A) := \sum_{x\in A} \P\bigl[ S_A=\infty \mid X_0=x\bigr].
$$
The number $\mathrm{Cap}(A)$ is a measure for the size and density of the set $A$ and can be regarded as a mathematical analogue of the ability of $A$ to hold electrical charge. In particular, the  denser the set $A$ the heavier it is to never visit $A$ again when starting at an ``inner'' point which has many elements of $A$ in its neighbourhood. In this article we are interested to study the asymptotic behaviour of the capacity of the range at time $n\in\N$ as $n\to\infty$. Recall that  the range of the random walk $(X_n)_{n\in\N_0}$ at time $n$ is given by the set
$$
\mathbf{R}_n:= \{X_0,X_1,\dots,X_n\},
$$
the set of vertices visited up to time $n$. The \textit{asymptotic capacity of the range of the random walk $(X_n)_{n\in\N_0}$ on $V$} is given by
$$
\lim_{n\to\infty} \frac1n \mathrm{Cap}(\mathbf{R}_n),
$$
provided the limit exists. The aim of this article is to show that the asymptotic capacity exists, is almost surely constant and strictly positive. In particular, we link the asymptotic capacity with the rate of escape of the underlying random walk (see (\ref{equ:capacity-formula}) in the proof of Theorem \ref{th:capacity-existence}). Moreover, we will provide a central limit theorem for the capacity of the range and show that it varies real-analytically, when the transition probabilities depend on finitely many  parameters.
\par
The range of random walks has been studied in great variety in the past. Let us outline some main results. Dvoretzky and Erd\"os \cite{dvoretzky-erdos:51} proved a strong law of large numbers for the range of simple random walk on $\Z^d$, $d\geq 2$. This result was generalized to arbitrary random walks on $\Z^d$, $d\geq 1$, by Spitzer \cite{spitzer:76}. Jain and Orey \cite{jain-orey:68} proved a central limit theorem for the range on $\Z^d$. For simple random walk on regular trees with $N+1$ branches, Chen, Yan and Zhou \cite{chen:97} calculated that the asymptotic range is given by $(N-1)/N$. For random walks on finitely generated groups with identity $e$, Guivarc'h \cite{guivarch} provided the nice formula 
$$
\lim_{n\to\infty} \frac{|\mathbf{R}_n|}{n}=1-\P[\exists n\in\N: X_n=e\mid X_0=e].
$$ 
The capacity of the range of random walks has been studied mainly on $\Z^d$ and groups. Jain and Orey \cite{jain-orey:68} proved existence of the asymptotic capacity of the range for random walks on the integer lattice $\Z^d$, $d\geq 3$, where the asymptotic capacity is strictly positive if and only if $d\geq 5$. Lawler \cite{lawler:91} gave estimates for intersection probabilities of random walks which in turn allow to estimate the capacity of the range. More recently, some topics have been studied which are related to the capacity of the range of random walks and have given some momentum to the study of the capacity, e.g., Sznitman \cite{sznitman:10} studied random interlacements, Asselah and Schapira \cite{asselah-schapira:17} investigated the geometry of random walks under localization constraints. 
%Lyons \cite{lyons:92} studied the capacity of the range.
For the capacity of the range of random walks on $\Z^d$, Asselah, Schapira and Sousi \cite{asselah-schapira-sousi:18,asselah-schapira-sousi:19} proved %results (including 
a central limit theorem. Further results are due to Chang \cite{chang:17} and Schapira \cite{schapira:20}. 
\par
Recently, Mrazovi\'c, Sandri\'c and \v{S}ebek \cite{mrazovic-sandric-sebek:21} proved that the asymptotic capacity of the range of symmetric simple random walks on finitely generated groups exists. They applied Kingman's subadditive theorem for proving existence of the asymptotic capacity of the range, and they also provided a central limit theorem for the asymptotic capacity. The goal of this paper is to go beyond group-invariant random walks and to derive analogous statements on existence of the asymptotic capacity of the range. In our case of general free products we have no group operation on $V$, and therefore we can \textit{not} apply Kingman's subadditive ergodic theorem. Thus, existence of the asymptotic capacity of the range is not guaranteed a-priori. Since the asymptotic capacity of the range is an important random walk characteristic number, studying existence for random walks on general free products deserves its own right.
\par
Free products form an important class of (graph) structures whose importance is due to Stallings' Splitting Theorem (see Stallings \cite{stallings:71}): a finitely generated group $\Gamma$ has more than one (geometric) end if and only if $\Gamma$ admits a non-trivial decomposition as a free product by amalgamation (amalgam) or as an HNN extension over a finite subgroup; see, e.g., Lyndon and Schupp \cite{lyndon-schupp} for more information on amalgams and HNN extensions. We recall that free products are amalgams over the trivial subgroup. In this article we consider free products of graphs which form a generalization of free products of groups. In particular, we have no underlying group-invariant random walk and the reasoning of the group case cannot be applied which makes it necessary to follow other approaches. Let us note that free products can -- at least to some extent -- also be used to model random walks on regular languages, First-In First-Out queues or stacks.
\par
Random walks on free products have been studied to a great extent throughout the last decades. Asymptotic behaviour of return probabilities of random walks on free products has been studied, e.g., by Gerl and Woess \cite{gerl-woess}, Woess \cite{woess3}, Sawyer \cite{sawyer78}, Cartwright and Soardi \cite{cartwright-soardi87}, Lalley \cite{lalley93,lalley:04} and Candellero and G. \cite{candellero-gilch}. Explicit formula for the drift and the asymptotic entropy have been computed by 
Mairesse and Math\'eus \cite{mairesse1} for random walks on free products of finite groups, while G. \cite{gilch:07,gilch:11} calculated different formulas for both characteristic numbers for random walks on free products of graphs. The spectral radius of random walks on some classes of free products of graphs has been studied by \mbox{Shi et al. \cite{sidoravicius:18}.}
Finally, the range of random walks on (general) free products has been studied in G. \cite{gilch:22}, where existence of the limit $\lim_{n\to\infty} \mathbf{R}_n/n$ has been proven, including a central limit theorem. This work will serve as a main reference for the current article, but the current article will go far beyond the scope of \cite{gilch:22}.
\par
The aim of this article is to go a step beyond the group case when studying the asymptotic capacity of the range of random walks. 
%under some mild assumptions formulated in Subsection \ref{subsec:free-products}. 
Throughout this paper we make the basic assumption that the Green function $G(z):=\sum_{n\geq 0} \P[X_n=o| X_0=o]z^n$, \mbox{$z\in\mathbb{C}$,} has radius of convergence $\mathscr{R}$ strictly bigger than $1$. This is equivalent to 
$$
\varrho:=\limsup_{n\to\infty} \mathbb{P}[X_n=o| X_0=o]^\frac1n <1.
$$
This assumption ensures transience of the underlying random walk $(X_n)_{n\in\N_0}$ and excludes  degenerate cases.
Our main result guarantees existence of the asymptotic capacity of the range of $(X_n)_{n\in\mathbb{N}_0}$:
\begin{theorem}\label{th:capacity-existence}
Assume that $\varrho<1$. Then there exists a constant \mbox{$\mathfrak{c}\in (0,1]$} such that
$$
\lim_{n\to\infty} \frac{\mathrm{Cap}(\mathbf{R}_n)}{n}=\mathfrak{c} \quad \textrm{almost surely}.
$$
\end{theorem}
We also provide a central limit theorem w.r.t. the capacity of the range:
\begin{theorem}\label{th:clt}
Assume that $\varrho<1$. Furthermore, assume that there are $x_0\in V$ and $\kappa\in\N$ such that $\mathbb{P}[X_\kappa=x_0| X_0=x_0]>0$.Then the asymptotic capacity satisfies the following central limit theorem:
$$
\frac{\mathrm{Cap}(\mathbf{R}_n)-n\cdot \mathfrak{c}}{\sigma \cdot \sqrt{n}} \xrightarrow{d} N(0,1),
$$
where $\sigma^2>0$ is given by (\ref{equ:sigma}).
\end{theorem}
The assumption of existence of $x_0\in V$ and $\kappa\in\N$ with $\mathbb{P}[X_\kappa=x_0| X_0=x_0]>0$ excludes once again degenerate cases, where a central limit theorem might not exist, see \mbox{Remark \ref{rem:degen-cases}} for a counterexample.
\par
We remark that Theorem \ref{th:capacity-existence} and \ref{th:clt} formulate   new results for random walks on free products of groups, since in \cite{mrazovic-sandric-sebek:21} only symmetric simple random walk on finitely generated groups were considered.
\par
The third main result shows that the capacity of the range varies real-analytically in terms of probability measures of constant support. Here, we assume that there are finitely many values $p_1,\dots,p_d\in (0,1)$, $d\in\N$, such that all strictly positive single-step transition probabilities $p(x,y)$, $x,y\in V$, satisfy $p(x,y)\in \{p_1,\dots,p_d\}$. If we let vary the values $p_1,\dots,p_d$ among all positive values which still allow well-defined random walks on $V$, we may then regard $\mathfrak{c}$ as a function in $(p_1,\dots,p_d)$. Then:
\begin{theorem}\label{th:analyticity}
Assume that $\varrho<1$. Suppose that the single step transition probabilities of the random walk $(X_n)_{n\in\N_0}$ take only $d\in\N$ many non-negative values $p_1,\dots,p_d\in(0,1)$. Then the mapping
$$
(p_1,\dots,p_d) \mapsto \mathfrak{c}=\mathfrak{c}(p_1,\dots,p_d)
$$
varies real-analytically.
\end{theorem}
We note that, for random walks on groups beyond free products, it is unknown whether the asymptotic capacity of the range varies real-analytically. In particular, this problem is open in the case of $\mathbb{Z}^d$ for $d\geq 2$. This underlines the innovative character of \mbox{Theorem \ref{th:analyticity}.} The proof heavily involves generating function techniques, which requires a deep understanding of the interplay of the generating functions on $V$ and on the free factors $V_i$.

The paper is organized as follows: in Section \ref{sec:rw-on-fp} we give a short introduction to free products on which we define a natural class of random walks, and we introduce some basic notation. In Section \ref{sec:existence} we derive existence of the asymptotic capacity of the range for random walks on free products, while in Section \ref{sec:clt}   the proposed central limit theorem is proven. In \mbox{Section \ref{sec:analyticity}} we prove the real-analytic behaviour of $\mathfrak{c}$. 
 Finally, in Section \ref{sec:remarks}, we give some additional remarks about made assumptions and the capacity.

\section{Random Walks on Free Products}
\label{sec:rw-on-fp}
\subsection{Free Products of Graphs}
\label{subsec:free-products}

%Let $V_1$ and $V_2$ be finite or countable sets with at least two elements. For each $i\in\calI:=\{1,2\}$, we select a distinguished element $o_i$ of $V_i$, which we call  the ``root'' of $V_i$. W.l.o.g. we assume that the sets $V_i$ are pairwise disjoint. Consider a (time-)homogeneous random walk with transition matrix $P_i=(p_i(x,y))_{x,y\in V_i}$ on each $V_i$. The corresponding $n$-step transition probabilities are denoted by $p_i^{(n)}(x, y)$, where $x, y\in V_i$. To help visualize this, we think of graphs $\mathcal{X}_i$ with vertex set $V_i$, root $o_i$ such that there is an oriented edge between $x\in V_i$ and $y\in V_i$ if and only if $p_i(x,y)>0$. Furthermore, we make the following technical assumptions:
%\begin{itemize}
%\item For every $x\in V_i\setminus\{o_i\}$, there exists $n_x\in\N$ such that $p_i^{(n_x)}(o_i,x)>0$.
%\item For every $x\in V_i$, $p_i(x,x)=0$.
%\item 
%\end{itemize}

Let $V_1$ and $V_2$ be finite or countable sets with at least two elements. We assume that $V_1\cap V_2=\emptyset$ and we exclude the case $|V_1|=|V_2|=2$; see Remark \ref{rem:case2x2}. For each $i\in\calI:=\{1,2\}$, we 
select a distinguished element $o_i$ of $V_i$, which we call  the ``root'' of $V_i$. 
On each $V_i$ consider a random walk with transition matrix $P_i=(p_i(x,y))_{x,y\in V_i}$.
The corresponding $n$-step transition probabilities are denoted by $p_i^{(n)}(x, y)$, where $x, y\in V_i$. Since only those elements of $V_i$ will be of interest, which can be reached from $o_i$, we may assume w.l.o.g. that, for every $i\in\calI$ and every $x\in V_i$, there exists some $n_x\in\N$ such that $p_i^{(n_x)}(o_i,x) > 0$. Furthermore, for sake of simplicity, we assume $p_i(x, x)= 0$ for every $i\in\calI$ and $x\in V_i$; this assumption can be lifted without restrictions but the general proof would reduce the readability of the proofs; see  \cite[Section 6]{gilch:22}.
%
%
%Furthermore, we make the following basic assumptions:
%\begin{itemize} 
%\item[(A1)] Each $x\in V_i\setminus\{o_i\}$ can be reached from $o_i$ after finitely many steps, that is, for every $i\in\calI$ and every $x\in V_i$ there exists some $n_x\in\N$ such that $p_i^{(n_x)}(o_i,x) > 0$.
%\item[(A2)] There are $i\in\calI$, $x_1,x_2\in V_i$ and $m,n\in\N$ such that $p_i^{(m)}(x,y)>0$ and $p_i^{(n)}(y,x)>0$.
%\item[(A2)] For sake of simplicity, we assume $p_i(x, x)= 0$ for every $i\in\calI$ and $x\in V_i$.
%\end{itemize}
%These assumptions are mild and have rather technical reasons; for further remarks, see Subection \ref{subsec:remarks}.
\par
For better visualization, we may think of rooted graphs $\mathcal{X}_i$
with vertex sets $V_i$ and roots $o_i$ such that there is an oriented edge $x\to y$ if and only if $p_i(x, y) > 0$.
\par
For $i\in\calI$, set $V_i^\times:= V_i\setminus\{o_i\}$ and $V^\times_\ast := V_1^\times \cup V_2^\times$. The \textit{free product of $V_1$ and $V_2$} is given by the set
\begin{equation}\label{equ:free-product}
V:=V_1\ast V_2 := \bigl\lbrace x_1x_2\dots x_n \,\bigl|\, n\in\N, x_j\in V^\times_\ast, x_j\in V_k^\times \Rightarrow x_{j+1}\notin V_k^\times\bigr\rbrace \cup\{o\},
\end{equation}
the set of all finite words over the alphabet $V^\times_\ast$ such that no two consecutive letters come from the same $V_i^\times$,
where $o$ describes the empty word. Note that $V_i^\times\subseteq V$ and we may consider $o_i$  as the ``empty word'' of $V_i$. Throughout this paper we will use the representation in (\ref{equ:free-product}) for elements in $V$.
\par
Observe that there is a natural partial composition law on $V$: if $u=u_1\dots  u_m\in V$ and $v = v_1\dots v_n \in V$ with $u_m\in V_i^\times$, $i\in\calI$, and $v_1\notin V_i^\times$, then $uv\in V$ stands for their concatenation as words. In particular, we set $uo_i := u$ for all $i\in\calI$ and $o_iu := u$; hence, $o_i$ is also identified with the empty word  $o$. 
Since concatenation of words is only partially defined, concatenation is \textit{not} a group operation on $V$; in particular, standard arguments from the group case like Kingman's subadditive ergodic theorem (as used in \cite{mrazovic-sandric-sebek:21}) can \textit{not} be applied directly. 
%things are getting much more complicated than in the case of free products of groups (see, e.g., \cite{lyndon-schupp} for more details on free products of groups); in particular, standard arguments like Kingman's subadditive ergodic theorem can \textit{not} be applied directly. 
\par
The \textit{word length} of a word $u = u_1\dots u_m$ is defined as $\Vert u\Vert:= m$. Additionally, we set $\Vert o\Vert := 0$. The \textit{type} $\t(u)$ of $u$ is defined to be $i\in\calI$ if $u_m\in V_i^\times$; we set $\t(o) := 0$. 
\par
The set $V$ can again be identified as the vertex set of a graph $\mathcal{X}$ which is constructed inductively as follows: take copies of $\mathcal{X}_1$ and $\mathcal{X}_2$ and glue them together at their roots to one single common root, which becomes $o$; inductively, at each vertex $v=v_1\dots v_k$ with $v_k\in V_i$ added in the step before attach a copy of $\mathcal{X}_j$, $j\in\calI \setminus\{i\}$, where $v$ is identified with $o_j$ from the new copy of $\mathcal{X}_j$. Then $\mathcal{X}$ is the \textit{free product of the graphs $\mathcal{X}_1$ and $\mathcal{X}_2$}. The underlying graph structure of free products allows us to define paths: %and lengths of paths: 
a \textit{path} of length $n\in\N$ in $\mathcal{X}$ is  a sequence of vertices $(z_0,z_1,\dots,z_n)$ in $V$ such that there is an oriented edge in $\mathcal{X}$ from $z_{i-1}$ to $z_i$ for each $i\in\{1,\dots,n\}$. %This gives rise to a natural graph distance and, for $x\in V$, we denote by $|x|$ the \textit{length of a shortest path} from $o$ to $x$. 
Recall that, for each $x\in V$, there is a path from $o$ to $x$ by construction of $V$ and the assumption made at the beginning of this subsection.

\begin{example}\label{ex:free-product}
Consider the sets $V_1=\{o_1,a\}$ and $V_2=\{o_2,b,c\}$ equipped with  the following graph structure:
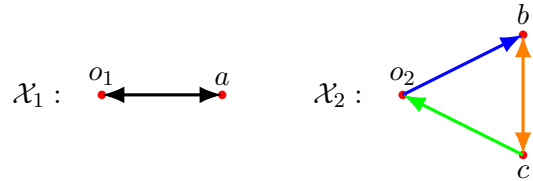
\begin{center}
\begin{tikzpicture}[scale=1]

\coordinate[label=left:$\mathcal{X}_1:$] (f) at (-0.5,0);
\coordinate[label=above:$o_1$] (e) at (0,0);
\coordinate[label=above:$a$] (a) at (2,0);

\coordinate[label=left:$\mathcal{X}_2:$] (g) at (4.5,0);
\coordinate[label=above:$o_2$] (e2) at (5,0);
\coordinate[label=above:$b$] (b) at (7,1);
\coordinate[label=below:$c$] (c) at (7,-1);

\fill[red] (e) circle (2pt);
\fill[red] (a) circle (2pt);
\fill[red] (b) circle (2pt);
\fill[red] (c) circle (2pt);
\fill[red] (e2) circle (2pt);
\draw[{Latex[length=3mm]}-{Latex[length=3mm]},very thick] (e) -- (a);
\draw[-{Latex[length=3mm]},very thick,blue] (e2) -- (b);
\draw[-{Latex[length=3mm]},very thick,green] (c) -- (e2);
\draw[{Latex[length=3mm]}-{Latex[length=3mm]},very thick,orange] (b) -- (c);

\end{tikzpicture}
\end{center}
The graph $\mathcal{X}$ of the free product $V_1\ast V_2$ has then the following structure:
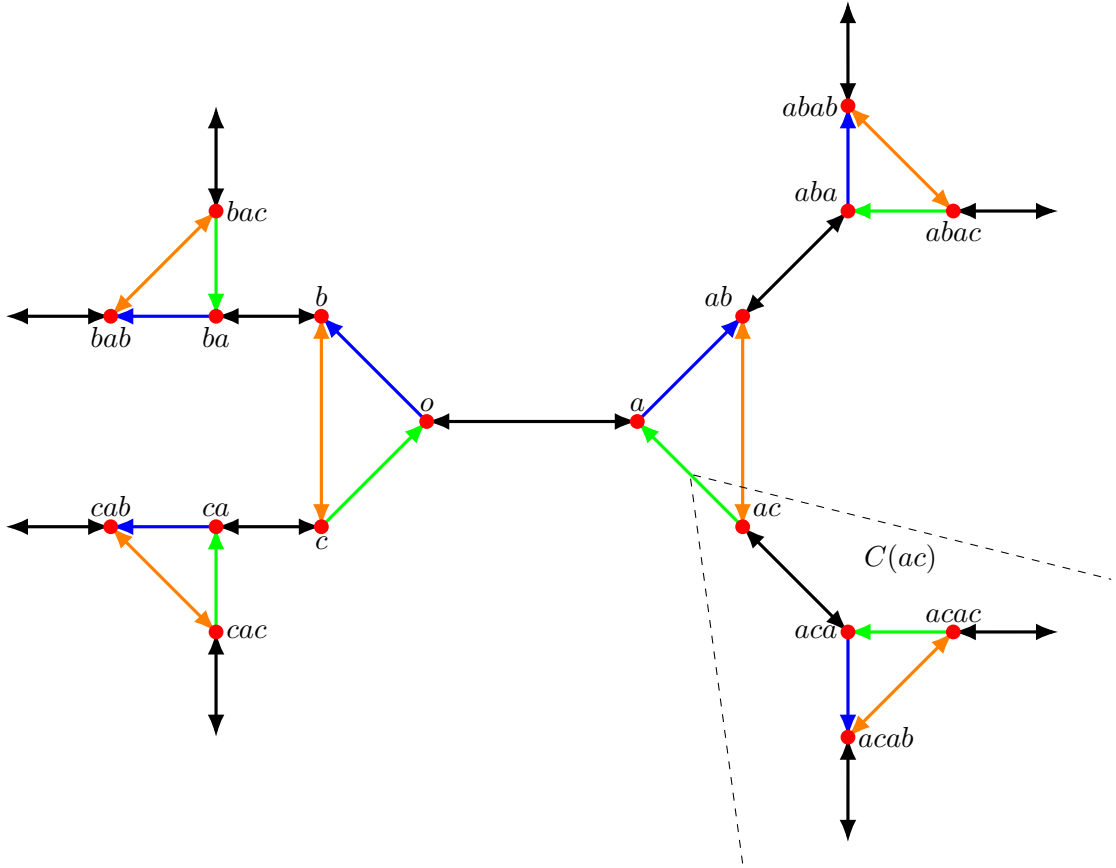
\begin{center}
\begin{tikzpicture}[scale=1.4]
\coordinate[label=above:$o$] (e) at (0,0);
\coordinate[label=above:$a$] (a) at (2,0);
\coordinate[label=170:$ab$] (ab) at (3,1);
\coordinate[label=10:$ac$] (ab2) at (3,-1);
\coordinate[label=170:$aba$] (aba) at (4,2);
\coordinate[label=left:$abab$] (abab) at (4,3);
\coordinate[label=below:$abac$] (abab2) at (5,2);
\coordinate[label=left:$aca$] (ab2a) at (4,-2);
\coordinate[label=right:$acab$] (ab2ab) at (4,-3);
\coordinate[label=above:$acac$] (ab2ab2) at (5,-2);

\coordinate[label=above:$b$] (b) at (-1,1);
\coordinate[label=below:$c$] (b2) at (-1,-1);

\coordinate[label=below:$ba$] (ba) at (-2,1);
\coordinate[label=below:$bab$] (bab) at (-3,1);
\coordinate[label=right:$bac$] (bab2) at (-2,2);

\coordinate[label=above:$ca$] (b2a) at (-2,-1);
\coordinate[label=above:$cab$] (b2ab) at (-3,-1);
\coordinate[label=right:$cac$] (b2ab2) at (-2,-2);

\draw[{Latex[length=3mm]}-{Latex[length=3mm]},very thick] (e) -- (a);
\draw[-{Latex[length=3mm]},very thick,blue] (e) -- (b);
\draw[{Latex[length=3mm]}-{Latex[length=3mm]},very thick,orange] (b) -- (b2);
\draw[-{Latex[length=3mm]},very thick,green] (b2) -- (e);
\draw[-{Latex[length=3mm]},very thick,blue] (a) -- (ab);
\draw[{Latex[length=3mm]}-{Latex[length=3mm]},very thick,orange] (ab) -- (ab2);
\draw[-{Latex[length=3mm]},very thick,green] (ab2) -- (a);
\draw[{Latex[length=3mm]}-{Latex[length=3mm]},very thick] (ab) -- (aba);
\draw[-{Latex[length=3mm]},very thick,blue] (aba) -- (abab);
\draw[{Latex[length=3mm]}-{Latex[length=3mm]},very thick,orange] (abab) -- (abab2);
\draw[-{Latex[length=3mm]},very thick,green] (abab2) -- (aba);
\draw[{Latex[length=3mm]}-{Latex[length=3mm]},very thick] (abab) -- (4,4);
\draw[{Latex[length=3mm]}-{Latex[length=3mm]},very thick] (abab2) -- (6,2);
\draw[{Latex[length=3mm]}-{Latex[length=3mm]},very thick] (ab2) -- (ab2a);
\draw[-{Latex[length=3mm]},very thick,blue] (ab2a) -- (ab2ab);
\draw[{Latex[length=3mm]}-{Latex[length=3mm]},very thick,orange] (ab2ab) -- (ab2ab2);
\draw[-{Latex[length=3mm]},very thick,green] (ab2ab2) -- (ab2a);
\draw[{Latex[length=3mm]}-{Latex[length=3mm]},very thick] (ab2ab2) -- (6,-2);
\draw[{Latex[length=3mm]}-{Latex[length=3mm]},very thick] (ab2ab) -- (4,-4);
\draw[{Latex[length=3mm]}-{Latex[length=3mm]},very thick] (b) -- (ba);
\draw[-{Latex[length=3mm]},very thick,blue] (ba) -- (bab);
\draw[{Latex[length=3mm]}-{Latex[length=3mm]},very thick,orange] (bab) -- (bab2);
\draw[-{Latex[length=3mm]},very thick,green] (bab2) -- (ba);
\draw[{Latex[length=3mm]}-{Latex[length=3mm]},very thick] (bab2) --+ (0,1);
\draw[{Latex[length=3mm]}-{Latex[length=3mm]},very thick] (bab) --+ (-1,0);

\draw[{Latex[length=3mm]}-{Latex[length=3mm]},very thick] (b2) -- (b2a);
\draw[-{Latex[length=3mm]},very thick,blue] (b2a) -- (b2ab);
\draw[{Latex[length=3mm]}-{Latex[length=3mm]},very thick,orange] (b2ab) -- (b2ab2);
\draw[-{Latex[length=3mm]},very thick,green] (b2ab2) -- (b2a);
\draw[{Latex[length=3mm]}-{Latex[length=3mm]},very thick] (b2ab2) --+ (0,-1);
\draw[{Latex[length=3mm]}-{Latex[length=3mm]},very thick] (b2ab) --+ (-1,0);

\draw[dashed] (3,-4.2) -- (2.5,-0.5) -- (6.5,-1.5);
\node at (4.5,-1.3) {$C(ac)$};

\fill[red] (e) circle (2pt);
\fill[red] (a) circle (2pt);
\fill[red] (b) circle (2pt);
\fill[red] (b2) circle (2pt);
\fill[red] (ab) circle (2pt);
\fill[red] (ab2) circle (2pt);
\fill[red] (aba) circle (2pt);
\fill[red] (abab) circle (2pt);
\fill[red] (abab2) circle (2pt);
\fill[red] (ab2a) circle (2pt);
\fill[red] (ab2ab) circle (2pt);
\fill[red] (ab2ab2) circle (2pt);
\fill[red] (ba) circle (2pt);
\fill[red] (bab) circle (2pt);
\fill[red] (bab2) circle (2pt);
\fill[red] (b2a) circle (2pt);
\fill[red] (b2ab) circle (2pt);
\fill[red] (b2ab2) circle (2pt);

\end{tikzpicture}
\end{center}
\end{example}

The tree-like graph structure of free products motivates the following definition: the \textit{cone} rooted at $x\in V$ is given by the set
$$
C(x):=\bigl\lbrace y\in V \mid y \textrm{ has prefix } x\bigr\rbrace.
$$
In particular, for all $y \in C(x)$, each path from $o$ to $y$ has to pass through $x$. Moreover, we have $C(o)=V$. 
E.g., in Example \ref{ex:free-product} we have 
$C(ac)=\{ac,aca,acab,acac,\dots\}$, the set of all words inside the dashed cone.

\subsection{Random Walks on Free Products}
\label{subsec:rw-on-fp}

We now construct a natural random walk on $V$ arising from $P_1$ and $P_2$. For this purpose, we lift the transition matrices $P_1$ and $P_2$ to transition matrices $\bar P_i=\bigl(\bar p_i(x,y)\bigr)_{x,y\in V}$, $i\in\calI$, \mbox{on $V$:} if $x\in V$ with $\t(x)\neq i$ and $v,w\in V_i$, then $\bar p_i(xv,xw):=p_i(v,w)$. Otherwise, we set $\bar p_i(x,y):=0$. Choose $\alpha\in (0,1)$. Then we define a new transition matrix $P$ on $V$  by
$$
P=\alpha \cdot \bar P_1 + (1-\alpha)\cdot \bar P_2,
$$
which governs a nearest neighbour random walk on $\mathcal{X}$. We may interpret the random walk as follows: if the random walker stands at some vertex  $x\in V$ with $\t(x)=i\in\calI$, he first tosses a coin  and afterwards -- in dependence of the outcome of the coin toss --  he either performs one step within the copy of $\mathcal{X}_i$ to which $x$ belongs according to $\bar P_i$ \textit{or}   one step into the new copy of $\mathcal{X}_j$, $j\in\calI\setminus\{i\}$, attached at $x$ according to $\bar P_j$. The sequence of random variables $(X_n)_{n\in\N_0}$ with $X_0:=o$ describes the random walk on $V$ governed by $P$, where $X_n$ denotes the random walker's position at time $n\in\N_0$.  For $x, y\in V$, the correspondung single and $n$-step transition probabilities are denoted by $p(x, y)$ and $p^{(n)}(x, y)$. Thus, $P$ governs a nearest neighbour random walk on the graph $\mathcal{X}$, where $P$ arises from a convex combination of the nearest neighbour random walks on the graphs $\mathcal{X}_1$ and $\mathcal{X}_2$. This definition ensures that every path $(w_0,\dots,w_n)$ in $\mathcal{X}$ has strictly positive probability $\P\bigl[X_1=w_1,\dots,X_n=w_n|X_0=w_0\bigr]>0$ to be performed.
We use the notation \mbox{$\P_x[\,\cdot\,]:=\P[\,\cdot\, \mid X_0=x]$} for $x\in V$.
\par
The spectral radius at $o$ is defined as
$$
\varrho:= \limsup_{n\to\infty} p^{(n)}(o,o)^{1/n}.
$$
As a \textit{basic assumption} throughout this paper we assume that 
$$
\varrho <1.
$$
Equivalently, the Green function $G(o,o|z):=\sum_{n\geq 0} p^{(n)}(o,o)\,z^n$, $z\in\mathbb{C}$, has radius of convergence $\mathscr{R}$ strictly bigger than $1$.
This assumption implies \textit{transience} of the random walk governed by $P$ and excludes degenerate cases; in particular, the recurrent case $|V_1|=|V_2|=2$ is excluded, see Remark \ref{rem:case2x2} . If one out of $P_1$ and $P_2$ is not irreducible, then $\varrho <1$ (easy to check!). If, e.g., $P_1$ and $P_2$ govern irreducible and reversible random walks, then $\varrho<1$; see \cite[Theorem 10.3]{woess}. Note that it is possible to construct null-recurrent random walks with $\mathscr{R}=1$ and $|V_1|>2=|V_2|$:
%In the following we give an example that $\mathscr{R}>1$ -- or equivalently, the spectral radius at $o$ satisfies $\varrho_o=\lim_{n\to\infty}p^{(n)}(o,o)^{\frac1n}<1$ -- must not necessarily hold. In particular, the example demonstrates that there are examples of free products of graphs which lead to null-recurrent random walks. 
\begin{example}\label{ex:recurrent-product}
Consider the sets $V_1=\{o_1,a_1,a_2\}$ and $V_2=\{o_2,b_1,b_2\}$ equipped with  the following graph structure:
\begin{center}
\begin{tikzpicture}[scale=.8]

\coordinate[label=left:$\mathcal{X}_2:$] (f) at (-0.5,0);
\coordinate[label=above:$o_2$] (o1) at (0,0);
\coordinate[label=above:$a_1$] (a1) at (2,1);
\coordinate[label=below:$a_2$] (a2) at (2,-1);

\coordinate[label=left:$\mathcal{X}_2:$] (g) at (4.5,0);
\coordinate[label=above:$o_2$] (o2) at (5,0);
\coordinate[label=above:$b_1$] (b1) at (7,1);
\coordinate[label=below:$b_2$] (b2) at (7,-1);

\fill[red] (o1) circle (2pt);
\fill[red] (a1) circle (2pt);
\fill[red] (a2) circle (2pt);
\fill[red] (b1) circle (2pt);
\fill[red] (b2) circle (2pt);
\fill[red] (o2) circle (2pt);
\draw[{Latex[length=3mm]}-{Latex[length=3mm]},very thick] (o1) -- (a1);
\draw[{Latex[length=3mm]}-{Latex[length=3mm]},very thick] (o1) -- (a2);
\draw[{Latex[length=3mm]}-{Latex[length=3mm]},very thick,blue] (o2) -- (b1);
\draw[{Latex[length=3mm]}-{Latex[length=3mm]},very thick,blue] (o2) -- (b2);
%\draw[{Latex[length=3mm]}-{Latex[length=3mm]},very thick,orange] (b) -- (c);

\end{tikzpicture}
\end{center}
The graph $\mathcal{X}$ of the free product $V_1\ast V_2$ has then the following structure:
\begin{center}
\begin{tikzpicture}[scale=1.4]
\coordinate[label=above:$o$] (e) at (0,0);
\coordinate[label=above:$a_1$] (a1) at (2,1);
\coordinate[label=above:$a_2$] (a2) at (2,0);
\coordinate[label=above:$a_1b_1\ $] (a1b1) at (4,1);
\coordinate[label=above:$a_1b_2\ \ $] (a1b2) at (4,2);
\coordinate[label=above:$a_1b_2a_1$] (a1b2a1) at (6,2.5);
\coordinate[label=above:$a_1b_2a_2$] (a1b2a2) at (6,2);
\coordinate[label=above:$a_1b_1a_1$] (a1b1a2) at (6,1.5);
\coordinate[label=above:$a_1b_1a_1$] (a1b1a1) at (6,1);
\coordinate[label=above:$\ \ a_2b_1$] (a2b1) at (4,-1);
\coordinate[label=above:$a_2b_2$] (a2b2) at (4,0);
\coordinate[label=above:$a_2b_2a_1$] (a2b2a1) at (6,-0.5);
\coordinate[label=above:$a_2b_2a_2$] (a2b2a2) at (6,0);
\coordinate[label=above:$a_2b_1a_1$] (a2b1a2) at (6,-1);
\coordinate[label=above:$a_2b_1a_1$] (a2b1a1) at (6,-1.5);
\coordinate[label=above:$b_1$] (b1) at (-2,1);
\coordinate[label=above:$b_2$] (b2) at (-2,0);

\draw[{Latex[length=3mm]}-{Latex[length=3mm]},very thick] (e) -- (a1);
\draw[{Latex[length=3mm]}-{Latex[length=3mm]},very thick] (e) -- (a2);
\draw[{Latex[length=3mm]}-{Latex[length=3mm]},very thick,blue] (a1) -- (a1b2);
\draw[{Latex[length=3mm]}-{Latex[length=3mm]},very thick,blue] (a1) -- (a1b1);
\draw[{Latex[length=3mm]}-{Latex[length=3mm]},very thick] (a1b2) -- (a1b2a1);
\draw[{Latex[length=3mm]}-{Latex[length=3mm]},very thick] (a1b2) -- (a1b2a2);
\draw[{Latex[length=3mm]}-{Latex[length=3mm]},very thick] (a1b1) -- (a1b1a1);
\draw[{Latex[length=3mm]}-{Latex[length=3mm]},very thick] (a1b1) -- (a1b1a2);
\draw[{Latex[length=3mm]}-{Latex[length=3mm]},very thick,blue] (a2) -- (a2b2);
\draw[{Latex[length=3mm]}-{Latex[length=3mm]},very thick,blue] (a2) -- (a2b1);
\draw[{Latex[length=3mm]}-{Latex[length=3mm]},very thick] (a2b2) -- (a2b2a1);
\draw[{Latex[length=3mm]}-{Latex[length=3mm]},very thick] (a2b2) -- (a2b2a2);
\draw[{Latex[length=3mm]}-{Latex[length=3mm]},very thick] (a2b1) -- (a2b1a1);
\draw[{Latex[length=3mm]}-{Latex[length=3mm]},very thick] (a2b1) -- (a2b1a2);

\draw[{Latex[length=3mm]}-{Latex[length=3mm]},very thick,blue] (e) -- (b1);
\draw[{Latex[length=3mm]}-{Latex[length=3mm]},very thick,blue] (e) -- (b2);

\draw[dashed,blue] (a1b1a1) --+ (1,0.25);
\draw[dashed,blue] (a1b1a1) --+ (1,0);
\draw[dashed,blue] (a1b1a2) --+ (1,0.25);
\draw[dashed,blue] (a1b1a2) --+ (1,0);
\draw[dashed,blue] (a1b2a1) --+ (1,0.25);
\draw[dashed,blue] (a1b2a2) --+ (1,0);
\draw[dashed,blue] (a1b2a2) --+ (1,0.25);
\draw[dashed,blue] (a1b2a2) --+ (1,0);
\draw[dashed,blue] (a2b1a1) --+ (1,0.25);
\draw[dashed,blue] (a2b1a1) --+ (1,0);
\draw[dashed,blue] (a2b1a2) --+ (1,0.25);
\draw[dashed,blue] (a2b1a2) --+ (1,0);
\draw[dashed,blue] (a2b2a1) --+ (1,0.25);
\draw[dashed,blue] (a2b2a2) --+ (1,0);
\draw[dashed,blue] (a2b2a2) --+ (1,0.25);
\draw[dashed,blue] (a2b2a2) --+ (1,0);
\draw[dashed] (b1) --+ (-1,0.5);
\draw[dashed] (b1) --+ (-1,0);
\draw[dashed] (b2) --+ (-1,0.5);
\draw[dashed] (b2) --+ (-1,0);

%\draw[{Latex[length=3mm]}-{Latex[length=3mm]},very thick] (bab2) --+ (0,1);
%\draw[{Latex[length=3mm]}-{Latex[length=3mm]},very thick] (bab) --+ (-1,0);
%
%\draw[{Latex[length=3mm]}-{Latex[length=3mm]},very thick] (b2) -- (b2a);
%\draw[-{Latex[length=3mm]},very thick,blue] (b2a) -- (b2ab);
%\draw[{Latex[length=3mm]}-{Latex[length=3mm]},very thick,orange] (b2ab) -- (b2ab2);
%\draw[-{Latex[length=3mm]},very thick,green] (b2ab2) -- (b2a);
%\draw[{Latex[length=3mm]}-{Latex[length=3mm]},very thick] (b2ab2) --+ (0,-1);
%\draw[{Latex[length=3mm]}-{Latex[length=3mm]},very thick] (b2ab) --+ (-1,0);
%
%\draw[dashed] (3,-4.2) -- (2.5,-0.5) -- (6.5,-1.5);
%\node at (4.5,-1.3) {$C(ac)$};

\fill[red] (e) circle (2pt);
\fill[red] (a1) circle (2pt);
\fill[red] (a2) circle (2pt);
\fill[red] (a1b1) circle (2pt);
\fill[red] (a1b2) circle (2pt);
\fill[red] (a1b1a1) circle (2pt);
\fill[red] (a1b1a2) circle (2pt);
\fill[red] (a1b2a1) circle (2pt);
\fill[red] (a1b2a2) circle (2pt);
\fill[red] (a2b1) circle (2pt);
\fill[red] (a2b2) circle (2pt);
\fill[red] (a2b1a1) circle (2pt);
\fill[red] (a2b1a2) circle (2pt);
\fill[red] (a2b2a1) circle (2pt);
\fill[red] (a2b2a2) circle (2pt);

\fill[red] (b1) circle (2pt);
\fill[red] (b2) circle (2pt);

\end{tikzpicture}
\end{center}

Set $\alpha_1=\alpha_2=\frac12$. Now it is easy to see that the process $(\Vert X_n\Vert)_{n\in\N_0}$ is an irreducible, null-recurrent random walk on $N_0$, which in turn implies that $(X_n)_{n\in\N_0}$ is null-recurrent. Hence, $R=1$.
\end{example}
~\\
Denote by $V_\infty$ the set of \textit{infinite} words $y_1y_2y_3\dots$ over the alphabet $V^\times_\ast$ such that no two consecutive letters arise from the same $V_i^\times$. For $x\in V$ and $y\in V_\infty$, denote by $x\wedge y$ the common prefix of maximal length of $x$ and $y$.
In \cite[Proposition 2.5]{gilch:07} it is shown that the random walk $(X_n)_{n\in\N_0}$ converges to some $V_\infty$-valued random variable $X_\infty$ in the sense that the length of the common prefix of $X_n$ and $X_\infty$ tends to infinity almost surely. In other words, $\lim_{n\to\infty} \Vert X_n\wedge X_\infty\Vert=\infty$ almost surely. 
We will make use of these results in the proofs later. For more details, we refer to \cite{gilch:07}.
\par
Important random walk characteristic numbers are given by the \textit{rate of escape} (or \textit{drift}) and the \textit{asymptotic range} of random walks. In \cite[Theorem 3.3]{gilch:07}  it was shown that there exists a strictly positive number $\ell\in(0,1]$, the \textit{rate of escape w.r.t. the word length (or block length)} of $(X_n)_{n\in\N_0}$, such that 
$$
\ell = \lim_{n\to\infty} \frac{\Vert X_n\Vert}{n} \quad \textrm{almost surely.}
$$
The set of vertices visited by the random walk $(X_n)_{n\in\N_0}$ until time $n$ is given by 
$$
\mathbf{R}_n:=\{X_0,X_1,\dots,X_n\}.
$$
The recent paper \cite[Theorem 1.1]{gilch:22} has proven existence of a strictly positive number $\mathfrak{r}\in(0,1]$, the \textit{(asymptotic) range} of the random walk $(X_n)_{n\in\N_0}$, such that
\begin{equation}\label{equ:range}
\mathfrak{r} = \lim_{n\to\infty} \frac{\bigl|\mathbf{R}_n\bigr|}{n} \quad \textrm{almost surely.}
\end{equation}
While $\ell$ measures the speed at which the random walk escapes to infinity, $\mathfrak{r}$ measures the speed at which new vertices are visited and it serves as a measure for how much of the graph is explored by the random walk. Both characteristic numbers will come into play later.

We mention two final remarks. The following equation will be important, which states that probabilities of paths within a cone depend only on their relative location to the cone's root:
\begin{lemma}\label{lem:cone-probs}
Let  $n\in\mathbb{N}$ and $w\in V$ and $w_1,\dots,w_n\in C(w)$. Write $w_i=ww_i'$ for \mbox{$i\in\{1,\dots,n\}$.} Then:
$$
\P_w[X_1=w_1,\dots,X_n=w_n] = \P_o[X_1=w'_1,\dots,X_n=w'_n].
$$
\end{lemma}
\begin{proof}
See \cite[Lemma 3.2]{gilch:22}.
\end{proof}
Due to the structure of the free product and Lemma \ref{lem:cone-probs} it is easy to verify that, for all $i\in\mathcal{I}$ and all $x\in V\setminus\{o\}$ with $\delta(x)=i$, the probability 
\begin{equation}\label{equ:xi}
\xi_i:=\mathbb{P}\bigl[ \exists n\in\mathbb{N}: X_n\notin C(x)\mid X_0=x\bigr]
\end{equation}
does not depend on $x$. In \cite[Lemma 2.3]{gilch:07} it is shown that $\xi_i<1$ for each $i\in\mathcal{I}$.

\subsection{Asymptotic Capacity of the Range}

We recall the definition of the capacity of a subset $A\subseteq V$.
For $A\subseteq V$, the stopping time of the first return into the set $A$ is defined as
$$
S_A :=\inf\bigl\lbrace m\in\mathbb{N} \mid X_m\in A\bigr\rbrace \in \mathbb{N}\cup\{\infty\}.
$$
The \textit{capacity} of the set $A$ is then defined as
$$
\mathrm{Cap}(A):= \sum_{x\in A} \P_x\bigl[ S_A=\infty\bigr].
$$
This number is a mathematical analogue from physics of the ability of $A$ to store an electrical charge.
\par
If there exists a constant $\mathfrak{c}\in [0,1]$ such that
$$
\lim_{n\to\infty} \frac{\mathrm{Cap}(\mathbf{R}_n)}{n}=\mathfrak{c} \quad \textrm{almost surely},
$$
then we call $\mathfrak{c}$ the \textit{asymptotic capacity of the range} of the random walk $(X_n)_{n\in\N_0}$ on $V$. It measures the asymptotic increase (per unit of time) of the capacity of the range.
The goal of this article is to prove existence of the asymptotic capacity of the range of random walks on free products;
%, to derive a central limit theorem for it and to show its real-analytic behaviour in terms of probability measures of constant support
 see \mbox{Theorem \ref{th:capacity-existence}.} %, \ref{th:clt} and \ref{th:analyticity}.
\par
Existence of the asymptotic capacity of the range of symmetric random walks on finitely generated groups was shown in \cite{mrazovic-sandric-sebek:21} with the help of  Kingman's subadditive ergodic theorem. Since we have only a partial composition law on $V$ and therefore no group operation on $V$, we cannot apply the reasoning from the group case (in particular, Kingman's subadditive ergodic theorem cannot be applied). This was the starting point for the present article to study existence of the asymptotic capacity of the range for general free products of graphs, which form an important class of graphs. Important pre-work has been done in the article \cite{gilch:22},  which  will serve as a base reference for our proofs. 
\par
The following little lemma will be very helpful in the proof of Theorem \ref{th:capacity-existence}:
\begin{lemma}\label{lem:Cap-difference}
For any  finite sets $A,B\subset V$ with $A\subseteq B$, we have
$$
\mathrm{Cap}(B)-\mathrm{Cap}(A) \leq |B\setminus A|.
$$
\end{lemma}
\begin{proof}
Let $A,B\subset V$ be finite sets with $A\subseteq B$. Then $S_B=\infty$ implies $S_A=\infty$, which implies $\P_x\bigl[ S_B=\infty\bigr]  \leq \P_x\bigl[ S_A=\infty\bigr]$ for $x\in A$. This yields:
\begin{eqnarray*}
\mathrm{Cap}(B)-\mathrm{Cap}(A) &=& \sum_{x\in B} \P_x\bigl[ S_B=\infty\bigr] - \sum_{x\in A} \P_x\bigl[ S_A=\infty\bigr] \\
&=&  \sum_{x\in B\setminus A} \underbrace{\P_x\bigl[ S_B=\infty\bigr]}_{\leq 1} + \sum_{x\in A} \Bigl( \underbrace{\P_x\bigl[ S_B=\infty\bigr]  - \P_x\bigl[ S_A=\infty\bigr] }_{\leq 0}\Bigr)\leq |B\setminus A|.
\end{eqnarray*}
\end{proof}

Finally, we explain briefly why we may exclude the case $|V_1|=|V_2|=2$ in the following sections:
\begin{remark}\label{rem:case2x2}
If $V_1=\{o_1,a\}$ and $V_2=\{o_2,b\}$, then $V=V_1\ast V_2$ becomes the free product $\bigl(\Z/(2\Z)\bigr) \times \bigl(\Z/(2\Z)\bigr)$ and the random walk on $V$ is recurrent. This yields $\mathfrak{r}=\lim_{n\to\infty}|\mathbf{R}_n|/n=0$. Since $\mathrm{Cap}(\mathbf{R}_n)\leq |\mathbf{R}_n|$, we obtain in this case $\mathfrak{c}=0$.
\end{remark}

%\subsection{Remarks}\label{subsec:remarks}
%~\par
%\begin{enumerate}
%\item If $|V_1|=|V_2|=2$ and $p_i(x,y)=1-\mathds{1}_{x}(y)$ for $x,y\in V_i$, $i\in\calI$, then $V$ becomes the free product of groups 
%$$
%V=\bigl(\Z/(2\Z)\bigr)\ast \bigl(\Z/(2\Z)\bigr).
%$$
%In this case the underlying random walk is group-invariant and $(X_n)_{n\in\N_0}$ is recurrent. Since $\mathrm{Cap}(R_n)\leq |R_n|$ and 
%$$
%\lim_{n\to\infty}\frac{|R_n|}{n}=1-\P_0[\exists n\in\N: X:n=o]=0,
%$$
%we have $\mathfrak{c}=0$ in this case; the formula for the asymptotic range in the group case was proved by Guivarc'h \cite{guivarch}.
%\par
%If at least one out of $P_1$ and $P_2$ is \textit{not} irreducible, then the random walk on $V$ is transient and we may apply the techniques below.
%\item NOCH NOTWENDIG?? In Subsection \ref{subsec:free-products} we made the assumption (A2) that there are $i\in\calI$ and $x,y\in V_i$ with $p_i^{(m)}(x,y)>0$ and $p_i^{(n)}(y,x)>0$ for some $m,n\in\N$. If this assumption does \textit{not} hold, then there are no circles in the graph $\mathcal{X}$, that is, at every instant of time a new vertex is visited, yielding $\mathbf{R}_n=n$ and therefore $\mathfrak{r}=1$. Hence, we may exclude this case from now on.
%\end{enumerate}

\section{Existence of the Asymptotic Capacity}
\label{sec:existence}

In this section we will prove existence of the asymptotic capacity of the range of the random walk $(X_n)_{n\in\N_0}$ on $V$. For this purpose, we introduce exit times which track the random walk's path to infinity; compare, e.g., with \cite{gilch:07,gilch:22}. Some concepts of \cite{gilch:22} will be crucial for our proofs, but the reasoning in the current article goes far beyond the scope of  \cite{gilch:22}.
\par
Denote by $X_n^{(k)}$ the projection of $X_n$ onto the first $k$ letters.
The \textit{exit times} are defined as follows: set $\e{0}:=0$, and for $k\geq 1$:
\begin{eqnarray*}
\e{k} &:=&\inf \bigl\lbrace m>0\mid  \forall n\geq m: X_n^{(k)} \textrm{ is constant}  \bigr\rbrace \\
&=&\inf \bigl\lbrace m>0 \mid \Vert X_m\Vert = k, \forall n\geq m: X_n\in C(X_m)\bigr\rbrace.
\end{eqnarray*}
That is, the random time $\e{k}$ denotes 
%the first instant of time from which on the random walk will only visit words of length  of at least $k$, or equivalently, 
the first instant of time from which onwards the random walk remains in the cone $C(X_{\e{k}})$.   In \cite[Proposition 2.5]{gilch:07} it is shown that  \mbox{$\Vert X_n\Vert\to\infty$} almost surely as $n\to\infty$, yielding $\e{k}<\infty$ almost surely and $\e{k+1}>\e{k}$ for all $k\in\mathbb{N}$. This motivates to consider only those random walk trajectories such that $\e{k}<\infty$ for all $k\in\N_0$: denote by $\Omega_0$ the set of all random walk trajectories $\omega=(x_0,x_1,\dots)\in V^{\N_0}$ such that $p(x_i,x_{i+1})>0$ for all $i\in\N_0$ and $\lim_{n\to\infty} \Vert x_n\Vert=\infty$;  the latter property implies $\e{k}(\omega)<\infty$ for all $k\in\N_0$, and we have $\P(\Omega_0)=1$. 
\begin{remark}
Note that the exit times used in \cite[Section 3.1]{gilch:22} were informally (and supposed to be also formally) defined as above, but the formal definition in that article was erroneous. This does not affect the results in  \cite{gilch:22}, since our definition above was used in the proofs of \cite{gilch:22}.
\end{remark}

%The increments between two consecutive exit times are denoted by $\i{k}:=\e{k}-\e{k-1}$.
We collect some useful results: by \cite[Proposition 3.2, Theorem 3.3]{gilch:07}, we have
\begin{equation}\label{equ:speed}
\ell=\lim_{n\to\infty} \frac{\Vert X_n\Vert}{n}=\lim_{k\to\infty} \frac{k}{\e{k}}\in (0,1]\quad \textrm{ almost surely.}
\end{equation}
For $n\in\N_0$, set
$$
\mathbf{k}(n) := \max\{k\in\N_0 \mid \e{k}\leq n\},
$$
that is, $\e{\mathbf{k}(n)}$ is the maximal exit time at time $n$. In \cite[Eq. (3.2)]{gilch:22} it is shown that  
\begin{equation}\label{equ:convergence-e_k(n)}
\lim_{n\to\infty}\frac{\e{\mathbf{k}(n)}}{n} = 1 \quad \textrm{almost surely.}
\end{equation}
If $X_{\e{k}}=g_1\dots g_k$, then we set 
$\W{k}:=g_k$.
\par
Since the random walk enters finally the cone $C(X_{\e{1}})$ and stays therein, then finally enters $C(X_{\e{2}})\subset C(X_{\e{1}})$ and stays therein for forever and so on, the idea is now to construct a partition of $\mathbf{R}_n$ into disjoint subsets of
$$
C(X_{\e{k}})\setminus \bigl(C(X_{\e{k+1}}) \cup \{X_{\e{k}}\}\bigr),\quad  k\leq \mathbf{k}(n), 
$$
and some remaining parts. This allows us to decompose $\mathrm{Cap}(\mathbf{R}_n)$ according to the partition of $\mathbf{R}_n$ which in turn will enable us to control the asymptotic increase of $\mathrm{Cap}(\mathbf{R}_n)$ as $n\to\infty$.
\par 
We start by decomposing $\mathbf{R}_{\e{k}}$, $k\in\N$. For this purpose, we define several random sets and quantities in the following.
Denote by $C_0\subset V$ the random set of words which start with a letter in $V_{\t(X_{\e{1}})}$ including $o$. Let  
\begin{eqnarray*}
&& \mathcal{R}_0^{(I)}:= \mathbf{R}_{\e{1}} \cap \Bigl( C_0\setminus \bigl( C(X_{\e{1}})  \cup \{o\}\bigr)\Bigr),\\
\textrm{and  for } k\geq 1: && \mathcal{R}_k^{(I)}:= \mathbf{R}_{\e{k+1}} \cap \Bigl( C(X_{\e{k}})\setminus \bigl(C(X_{\e{k+1}}) \cup \{X_{\e{k}}\}\bigr)\Bigr)
\end{eqnarray*}
be the \text{``interior part of the range''} between two consecutive exit time points. Set
\begin{eqnarray*}
&&\mathcal{R}_0 := \mathcal{R}_0^{(I)} \cup \bigl\lbrace o,X_{\e{1}}\bigr\rbrace, \\ 
\textrm{and  for } k\geq 1: && \mathcal{R}_k := \mathcal{R}_k^{(I)} \cup \bigl\lbrace X_{\e{k}},X_{\e{k+1}}\bigr\rbrace.
\end{eqnarray*} 
Observe that we have established the following disjoint decomposition of $\mathbf{R}_{\e{k}}$, $k\geq 1$:
\begin{equation}\label{equ:R_ek_decomposition}
\mathbf{R}_{\e{k}} = \bigl(\mathbf{R}_{\e{1}}\cap \overline{C_0}\bigr) \ \cup\ \bigcup_{i=0}^{k-1} \bigl(\mathcal{R}_{i}\setminus \{X_{\e{i+1}}\}\bigr)\  \cup \ \bigl(\mathbf{R}_{\e{k}}\cap C(X_{\e{k}})\bigr).
\end{equation}
We remark that, from time $\e{1}$ on, the set $\mathbf{R}_{\e{1}}\cap \overline{C_0}$ does not change any more, since every path from $C(X_{\e{1}})$ to $\overline{C_0}$ has to leave $C(X_{\e{1}})$, but the random walk does not leave $C(X_{\e{1}})$ anymore after time $\e{1}$.
\par
The next step is to decompose $\mathrm{Cap}(\mathbf{R}_{\e{k}})$ according to the decomposition of $\mathbf{R}_{\e{k}}$. In the following we will evaluate random variables and random sets at arbitrary $\omega\in\Omega_0$ in order to ensure clarity of notations. Observe that, for $\omega\in\Omega_0$,
\begin{equation}\label{equ:Cap-Rek}
\mathrm{Cap}\bigl(\mathbf{R}_{\e{k}}(\omega)\bigr)=\sum_{x\in \mathbf{R}_{\e{k}}(\omega)} \P_x\bigl[ S_{\mathbf{R}_{\e{k}}(\omega)}=\infty\bigr].
\end{equation}
%For this purpose, let $\omega$ be any realization of $(X_n)_{n\in\N_0}$ such that $\lim_{n\to\infty} \Vert X_n(\omega)\Vert=\infty$ (recall that the set of these $\omega$'s has probability $1$).
We define for $k\in\N_0$ and $\omega\in\Omega_0$:
$$
\mathcal{C}_k^{(I)}(\omega):= \sum_{x\in \mathcal{R}_k^{(I)}(\omega)} \P_x\bigl[ S_{\mathcal{R}_{k}(\omega)}=\infty\bigr]
$$
and set
\begin{eqnarray*}
\mathcal{C}_0(\omega)&:=& \mathcal{C}_0^{(I)}(\omega) +  \P_{o}\bigl[ S_{\mathcal{R}_{0}(\omega)}=\infty,\forall n\geq 1: X_n\in C_0(\omega)\bigr] \\
&& \quad +  \P_{X_{\e{1}}}\bigl[ S_{\mathcal{R}_{0}(\omega)}=\infty,\forall n\geq 1: X_n\notin C\bigl(X_{\e{1}}(\omega)\bigr)\bigr], \\
\textrm{ for } k\geq 1: \ 
\mathcal{C}_k(\omega)&:=& \mathcal{C}_k^{(I)}(\omega) +  \P_{X_{\e{k}}(\omega)}\bigl[ S_{\mathcal{R}_k(\omega)}=\infty,\forall n\geq 1: X_n\in C\bigl(X_{\e{k}}(\omega)\bigr)\bigr] \\
&& \quad +  \P_{X_{\e{k+1}}(\omega)}\bigl[ S_{\mathcal{R}_k(\omega)}=\infty,\forall n\geq 1: X_n\notin C\bigl(X_{\e{k+1}}(\omega)\bigr)\bigr].
\end{eqnarray*}
Additionally, we define
$$
\mathcal{C}_0^\ast(\omega):= \sum_{x\in \mathbf{R}_{\e{1}}(\omega)\cap\overline{C_0(\omega)}} \P_x\bigl[ S_{\mathbf{R}_{\e{1}}(\omega)}=\infty\bigr] +  \P_{o}\bigl[ S_{\mathbf{R}_{\e{1}}(\omega)}=\infty,\forall n\geq 1: X_n\notin C_0(\omega)\bigr]
$$
and
\begin{eqnarray*}
\mathcal{O}_k(\omega) &:=&  \sum_{x\in \mathbf{R}_{\e{k}}(\omega) \cap C\bigl(X_{\e{k}}(\omega)\bigr)\setminus\{X_{\e{k}}(\omega)\}} \P_x\bigl[ S_{\mathbf{R}_{\e{k}}(\omega)}=\infty\bigr] \\
&& \quad +  \P_{X_{\e{k}}(\omega)}\bigl[ S_{\mathbf{R}_{\e{k}}(\omega)}=\infty,\forall n\geq 1: X_n\in C\bigl(X_{\e{k}}(\omega)\bigr)\bigr].
\end{eqnarray*}
The decomposition of $\mathbf{R}_{\e{k}}$ in (\ref{equ:R_ek_decomposition}) leads to the following decomposition of $\mathrm{Cap}(\mathbf{R}_{\e{k}})$, 
which will be one of the the main keys for the proofs later:
\begin{proposition}\label{prop:capacity-decomposition}
For all $k\in\N$,
\begin{equation}\label{equ:Cap_R_ek}
\mathrm{Cap}(\mathbf{R}_{\e{k}}) = \mathcal{C}_0^\ast + \sum_{i=0}^{k-1} \mathcal{C}_i + \mathcal{O}_k \quad \textrm{almost surely}.
\end{equation}
\end{proposition}
\begin{proof}
Let  $k\in\N$, and take $\omega\in\Omega_0$ such that $\lim_{n\to\infty} \Vert X_n(\omega)\Vert=\infty$. 
%By definition of the capacity, we have
%\begin{equation}\label{equ:Cap_R_ek2}
%\mathrm{Cap}\bigl(\mathbf{R}_{\e{k}}(\omega)\bigr) =\sum_{x\in \mathbf{R}_{\e{k}}(\omega)} \P_x \bigl[ S_{\mathbf{R}_{\e{k}}(\omega)}=\infty\bigr].
%\end{equation}
We check that the summands on the right hand sides of (\ref{equ:Cap_R_ek}) (evaluated at $\omega$) and (\ref{equ:Cap-Rek})
%(\ref{equ:Cap_R_ek2}) 
are the same by using the decomposition of $\mathbf{R}_{\e{k}}$ in (\ref{equ:R_ek_decomposition}).
Let  $x\in \mathcal{R}_i^{(I)}(\omega)$  for some $i\in\{0,1,\dots,k-1\}$, that is, $x\in C\bigl(X_{\e{i}}(\omega)\bigr)\setminus\{X_{\e{i}}(\omega)\}$ but $x\notin C\bigl(X_{\e{i+1}}(\omega)\bigr)$.
Then $\P_x\bigl[S_{\mathbf{R}_{\e{k}}(\omega)}=\infty\bigr]$ is already determined by $\mathcal{R}_i(\omega)$ via
 $$
 \P_x\bigl[S_{\mathbf{R}_{\e{k}}(\omega)}=\infty\bigr] = \P_x\bigl[S_{\mathcal{R}_i(\omega)} =\infty\bigr],
 $$
since every path from $x$ to $\mathbf{R}_{\e{k}}(\omega)\setminus \mathcal{R}_i(\omega)$ has to pass through either $X_{\e{i}}(\omega)\in \mathcal{R}_i(\omega)$ or $X_{\e{i+1}}(\omega)\in \mathcal{R}_i(\omega)$. Therefore, in this case  the summand $\P_x \bigl[ S_{\mathbf{R}_{\e{k}}(\omega)}=\infty\bigr]$ is -- by definition of $\mathcal{C}_i$ -- counted in $\mathcal{C}_i^{(I)}(\omega)$, and thus in  $\mathcal{C}_i(\omega)$.
\par 
If $x= X_{\e{i}}(\omega)$ for some $i\in\{0,1,\dots,k-1\}$, then
\begin{eqnarray*}
\P_x \bigl[ S_{\mathbf{R}_{\e{k}}(\omega)}=\infty\bigr] &=&  \P_x \bigl[ S_{\mathbf{R}_{\e{k}}(\omega)}=\infty,\forall n\geq 1: X_n\notin C\bigl(X_{\e{i}}(\omega)\bigr)\bigr] \\
&&\quad + \P_x \bigl[ S_{\mathbf{R}_{\e{k}}(\omega)}=\infty,\forall n\geq 1: X_n\in C\bigl(X_{\e{i}}(\omega)\bigr)\bigr]\\
&=& \P_x \bigl[ S_{\mathcal{R}_{i-1}(\omega)}=\infty,\forall n\geq 1: X_n\notin C\bigl(X_{\e{i}}(\omega)\bigr)\bigr] \\
&&\quad + \P_x \bigl[ S_{\mathcal{R}_{i}(\omega)}=\infty,\forall n\geq 1: X_n\in C\bigl(X_{\e{i}}(\omega)\bigr)\bigr], 
\end{eqnarray*}
since every path from $C\bigl(X_{\e{i}}(\omega)\bigr)\setminus \{X_{\e{i}}(\omega)\}$ to its complement (and vice versa) has to pass through $X_{\e{i}}(\omega)$. The first summand on the rightmost hand side is counted in $\mathcal{C}_{i-1}(\omega)$, while the second summand is counted in $\mathcal{C}_{i}(\omega)$.
\par
Consider now the case $x\in\mathbf{R}_{\e{1}}(\omega)\cap\overline{C_0(\omega)}$. Since $\mathbf{R}_{\e{k}}\setminus \mathbf{R}_{\e{1}}\subset C(X_{\e{1}}) \subset C_0$ and every path from $x\in \overline{C_0(\omega)}$ to $\mathbf{R}_{\e{k}}\setminus \mathbf{R}_{\e{1}}$ has to pass through $X_{\e{1}}\in \mathbf{R}_{\e{1}}$, we have
$$
\P_x \bigl[ S_{\mathbf{R}_{\e{k}}(\omega)}=\infty\bigr]  = \P_x \bigl[ S_{\mathbf{R}_{\e{1}}(\omega)}=\infty\bigr]. 
$$
This probability is counted in $\mathcal{C}_0^\ast(\omega)$. An analogous argument gives in the case $x=o$:
\begin{eqnarray*}
&&\P_o \bigl[ S_{\mathbf{R}_{\e{k}}(\omega)}=\infty,\forall n\geq 1: X_n\notin C_0(\omega)\bigr]  = \P_o \bigl[ S_{\mathbf{R}_{\e{1}}(\omega)}=\infty,\forall n\geq 1: X_n\notin C_0(\omega)\bigr] \\
\textrm{and} &&\P_o \bigl[ S_{\mathbf{R}_{\e{k}}(\omega)}=\infty,\forall n\geq 1: X_n\in C_0(\omega)\bigr]  = \P_o \bigl[ S_{\mathcal{R}_{0}(\omega)}=\infty,\forall n\geq 1: X_n\in C_0(\omega)\bigr] .
\end{eqnarray*}
This implies
\begin{eqnarray*}
&&\P_o \bigl[ S_{\mathbf{R}_{\e{k}}(\omega)}=\infty\bigr] \\
&=& \P_o \bigl[ S_{\mathbf{R}_{\e{k}}(\omega)}=\infty,\forall n\geq 1: X_n\notin C_0(\omega)\bigr] + \P_o \bigl[ S_{\mathbf{R}_{\e{k}}(\omega)}=\infty,\forall n\geq 1: X_n\in C_0(\omega)\bigr]\\
&=& \P_o \bigl[ S_{\mathbf{R}_{\e{1}}(\omega)}=\infty,\forall n\geq 1: X_n\notin C_0(\omega)\bigr] + \P_o \bigl[ S_{\mathcal{R}_{0}(\omega)}=\infty,\forall n\geq 1: X_n\in C_0(\omega)\bigr];
\end{eqnarray*}
the first summand on the rightmost hand side is counted in $\mathcal{C}_0^\ast(\omega)$, while the second one is counted in $\mathcal{C}_0(\omega)$.
\par
If $x\in \mathbf{R}_{\e{k}}(\omega)\cap C\bigl(X_{\e{k}}(\omega)\bigr)\setminus \{X_{\e{k}}(\omega)\}$, then $\P_x \bigl[ S_{\mathbf{R}_{\e{k}}(\omega)}=\infty\bigr]$ is counted in $\mathcal{O}_k(\omega)$; in the case $x= X_{\e{k}}(\omega)$, we get with an analogous argument as above that
\begin{eqnarray*}
&&\P_{X_{\e{k}}(\omega)} \bigl[ S_{\mathbf{R}_{\e{k}}(\omega)}=\infty\bigr]\\
& =& \P_{X_{\e{k}}(\omega)} \bigl[ S_{\mathbf{R}_{\e{k}}(\omega)}=\infty,\forall n\geq 1: X_n\notin C\bigl(X_{\e{k}}(\omega)\bigr)\bigr] \\
&&\quad + \P_{X_{\e{k}}(\omega)} \bigl[ S_{\mathbf{R}_{\e{k}}(\omega)}=\infty,\forall n\geq 1: X_n\in C\bigl(X_{\e{k}}(\omega)\bigr)\bigr]\\
&=& \P_{X_{\e{k}}(\omega)} \bigl[ S_{\mathcal{R}_{k-1}(\omega)}=\infty,\forall n\geq 1: X_n\notin C\bigl(X_{\e{k}}(\omega)\bigr)\bigr] \\
&& \quad + \P_{X_{\e{k}}(\omega)} \bigl[ S_{\mathbf{R}_{\e{k}}(\omega)}=\infty,\forall n\geq 1: X_n\in C\bigl(X_{\e{k}}(\omega)\bigr)\bigr],
\end{eqnarray*}
that is, the first summand on the rightmost hand side is counted in $\mathcal{C}_{k-1}(\omega)$, while the second one is counted in $\mathcal{O}_k(\omega)$.
\par
Since we have compared all summands in (\ref{equ:Cap-Rek}) and (\ref{equ:Cap_R_ek}), we have proven the proposed equality.
\end{proof}
The following two corollaries will be needed in  the proof of Proposition \ref{prop:capacity-exit-decomposition}.
\begin{corollary}\label{cor:C0ast-convergence}
$$
\lim_{k\to\infty} \frac{\mathcal{C}_0+\mathcal{C}_0^\ast}{k}=0 \quad \textrm{almost surely.}
$$
\end{corollary}
\begin{proof}
Since $\mathcal{C}_0+\mathcal{C}_0^\ast \leq |\mathbf{R}_{\e{1}}| \leq \e{1}+1<\infty$ almost surely, the claim follows immediately.% from $\e{1}<\infty$ almost surely.
\end{proof}

\begin{corollary}\label{lem:Ok-convergence}
$$
 \lim_{k\to\infty} \frac{\mathcal{O}_{k}}{k}=0 \quad \textrm{almost surely.}
 $$
\end{corollary}
\begin{proof}
Define $\mathbf{O}_k:=\bigl| \mathbf{R}_{\e{k}}\cap C(X_{\e{k}})\bigr|$.
By \cite[Corollary 3.13]{gilch:22}, we have $\mathbf{O}_k/k\to 0$ almost surely as $k\to\infty$; we note that in \cite[Corollary 3.13]{gilch:22} the common prefix of the elements in $\mathbf{R}_{\e{k}}\cap C(X_{\e{k}})$ are cancelled which obviously does not affect the cardinality of this set. Since $0\leq \mathcal{O}_{k}\leq \mathbf{O}_k$,  the claim follows.
\end{proof}
%Moreover, we have:
%\begin{lemma}\ref{lem:integral-finite}
%\int \mathcal{C}_1\,d\pi <\infty.
%\end{lemma}
%\begin{proof}
%
%
%\end{proof}
In order to track the range between two consecutive exit times we introduce the following random functions on $V$: 
for $k\in\N$, define 
$$
\psi_k: V\to \{0,1\}, \  x \mapsto \begin{cases}
1, & \textrm{if } \mathbf{W}_1\dots\mathbf{W}_{k-1}x\in \mathbf{R}_{\e{k}},\\
0, & \textrm{otherwise}.
\end{cases}
$$
That is, $\psi_k$ describes the elements visited by the random walk in $C(X_{\e{k-1}})$ up to time $\e{k}$, where the common first $k-1$ letters are deleted. A main key is the following:
\begin{proposition}\label{prop:pos-rec-process}
The stochastic process $(\mathbf{W}_k,\psi_k)_{k\in\N}$  forms an homogeneous, irreducible positive-recurrent Markov chain. 
\end{proposition}
\begin{proof}
See  \cite[Propositions 3.4 \& 3.10]{gilch:22}.
\end{proof}

This fact will play a crucial role in the following proposition which will serve as a key ingredient in the proof of Theorem \ref{th:capacity-existence}:

\begin{proposition} \label{prop:capacity-exit-decomposition}
Assume that $|V_2|\geq 3$. Then there exists a real number $\overline{\mathfrak{c}}>0$ such that
$$
\lim_{k\to\infty} \frac{\mathrm{Cap}(\mathbf{R}_{\e{k}})}{k} = \overline{\mathfrak{c}} \quad \textrm{almost surely.}
$$
\end{proposition}
\begin{proof}
In view of Proposition \ref{prop:capacity-decomposition}, Corollaries \ref{cor:C0ast-convergence} and  \ref{lem:Ok-convergence} it suffices to show existence of a real number $\overline{\mathfrak{c}}>0$ such that
$$
\lim_{k\to\infty} \frac{1}{k}\sum_{i=1}^{k-1} \mathcal{C}_i  = \overline{\mathfrak{c}} \quad \textrm{ almost surely}.
$$
%For $k\in\N$, define the random function $\psi_k:V\to \{0,1\}$ such that $\psi_k(x)=1$ if and only if $\mathbf{W}_1\dots\mathbf{W}_{k-1}x\in \mathbf{R}_{\e{k}}$, that is, $\psi_k$ describes the elements visited by the random walk in $C(X_{\e{k}})$ up to time $\e{k}$, where the first $k-1$ letters are deleted. In \cite{gilch:22} it is shown that 
%the stochastic process $(\mathbf{W}_k,\psi_k)_{k\in\N}$ from \cite{gilch:22}  forms an homogeneous, \mbox{positive-recurrent} Markov chain. 
First, we claim that 
$\mathcal{C}_i$ can be rewritten as a function in $(\mathbf{W}_i,\psi_{i})$ and $(\mathbf{W}_{i+1},\psi_{i+1})$ for all $i\in\mathbb{N}$:  an analogous argument as in the proof of Proposition \ref{prop:capacity-decomposition} yields  for   all $\omega\in\Omega_0$ and $x\in  \mathcal{R}_k^{(I)}(\omega)$ that
\begin{equation}\label{equ:prop3.4-1}
\P_x\bigl[S_{\mathbf{R}_{\e{k}}(\omega)}=\infty\bigr] = \P_x\bigl[S_{\mathcal{R}_{k}(\omega)}=\infty\bigr] 
\end{equation}
and
\begin{eqnarray}
 \P_{X_{\e{k}}(\omega)}\bigl[S_{\mathbf{R}_{\e{k}}(\omega)}=\infty\bigr] &=& \P_{X_{\e{k}}(\omega)}\bigl[S_{\mathcal{R}_{k-1}(\omega)}=\infty,\forall m\geq 1: X_m\notin C\bigl(X_{\e{k}}(\omega)\bigr)\bigr]\label{equ:prop3.4-2} \\
&& \quad + \P_{X_{\e{k}}(\omega)}\bigl[S_{\mathcal{R}_{k}(\omega)}=\infty,\forall m\geq 1: X_m\in C\bigl(X_{\e{k}}(\omega)\bigr)\bigr];\nonumber
\end{eqnarray}
with the help of  Lemma \ref{lem:cone-probs} (remove the first $k-1$ letters of $X_{\e{k}}(\omega)$ and in each \mbox{$v\in \mathcal{R}_{k-1}(\omega)$,} $w\in \mathcal{R}_{k}(\omega)$) it is easy to check that the above probabilities on the right hand sides of (\ref{equ:prop3.4-1}) and (\ref{equ:prop3.4-2})
%occuring in $\mathcal{C}_i$, namely \mbox{$\P_{X_{\e{k}}(\omega)}\bigl[S_{\mathbf{R}_{\e{k}}(\omega)}=\infty\bigr]$} and \mbox{$\P_x\bigl[S_{\mathbf{R}_{\e{k}}(\omega)}=\infty\bigr]$}, 
can be rewritten in terms of $\mathbf{W}_i(\omega),\mathbf{W}_{i+1}(\omega),\psi_{i}(\omega),\psi_{i+1}(\omega)$, that is,   $\mathcal{C}_i$ can be formulated as a function in $\mathbf{W}_i,\mathbf{W}_{i+1},\psi_{i},\psi_{i+1}$.
Letting $\pi$ be the equilibrium of the homogeneous, positive-recurrent Markov chain $\bigl((\mathbf{W}_k,\psi_k),(\mathbf{W}_{k+1},\psi_{k+1})\bigr)_{k\in\N}$, 
we can apply the ergodic theorem for positive recurrent Markov chains and obtain
$$
\lim_{k\to\infty}  \frac{1}{k}\sum_{i=1}^{k-1} \mathcal{C}_i =\int \mathcal{C}_1 \,d\pi =: \overline{\mathfrak{c}},
$$
where the integral is well-defined since $\mathcal{C}_1\geq 0$. 
\par
We now show that the integral is finite. To this end, assume that $\int \mathcal{C}_1 \,d\pi=\infty$. This implies together with (\ref{equ:speed}) and (\ref{equ:convergence-e_k(n)}):
$$
\frac1n \sum_{j=1}^{\mathbf{k}(n)}\mathcal{C}_j = \underbrace{\frac{\mathbf{e}_{\mathbf{k}(n)}}{n}}_{\to 1} \underbrace{\frac{\mathbf{k}(n)}{\mathbf{e}_{\mathbf{k}(n)}}}_{\to \ell}\underbrace{\frac{1}{\mathbf{k}(n)}\sum_{j=1}^{\mathbf{k}(n)}\mathcal{C}_j}_{\to \infty} \xrightarrow{n\to\infty} \infty \textrm{ almost surely.}
$$
On the other hand side, we have
$$
\limsup_{n\to\infty}\frac1n \sum_{j=1}^{\mathbf{k}(n)}\mathcal{C}_j \leq \limsup_{n\to\infty}\frac1n \cdot |\mathbf{R}_{n}|  \leq 1 \textrm{ almost surely},
$$
which yields a contradiction. Therefore, $\overline{\mathfrak{c}}=\int \mathcal{C}_1 \,d\pi<\infty$.
\par
% Thus, it remains to show that the integral is finite. But this follows from \cite[Lemma 3.12]{gilch:22} and the fact that $\mathcal{C}_0 \leq \mathbf{\widetilde{R}}_1 := \bigl|\mathbf{R}_{\e{1}} \cap  C_0 \cap \overline{C(X_{\e{1}})}\bigr|+1$.
%\par
Finally, it remains to show that $\overline{\mathfrak{c}}>0$. For this purpose, take $g_1\in V_1\setminus\{o_1\}$ and \mbox{$g_2\in V_2\setminus\{o_2\}$} with  $p_1(o_1,g_1)>0$ and   $p_2(o_2,g_2)>0$. Choose now any $\bar g_2\in V_2\setminus\{o_2,g_2\}$ such that $p(g_2,\bar g_2)>0$ or $p^{(2)}(g_2,\bar g_2)\geq p(g_2,o_2)\cdot p(o_2,\bar g_2)>0$; recall that this choice of $\bar g_2$ is possible since $|V_2|\geq 3$, due to non-existence of loops and stochasticity of $P_2$. 
For any \mbox{$M\subseteq V$,} denote by $\mathds{1}_{M}:V\to \{0,1\}$ the indicator function w.r.t. $M$, that is, $\mathds{1}_{M}(x)=1$ for $x\in V$ if and only if $x\in M$. Consider now the event that the random walk's first step goes from $o$ to $g_1$, followed by a step to $g_1g_2$ and staying inside the cone $C(g_1g_2)$ afterwards. This event has positive probability to occur, namely $p(o,g_1)\cdot p(g_1,g_1g_2)\cdot (1-\xi_2)>0$.
Moreover, this event is a subevent of  
$$
\mathcal{W}:=\bigl[(\mathbf{W}_1,\psi_1)=(g_1,\mathds{1}_{\{o,g_1\}}),(\mathbf{W}_2,\psi_2)=(g_2,\mathds{1}_{\{o,g_2\}})\bigr].
$$ 
On the event $\mathcal{W}$,  the term 
$$
\bigl( p(g_2,\bar g_2) + p(g_2,o)\cdot p(o,\bar g_2)\bigr)\cdot (1-\xi_2)>0
$$ 
contributes to $\mathcal{C}_1$, yielding
$$
\overline{\mathfrak{c}}=\int \mathcal{C}_1 \,d\pi \geq \bigl( p(g_2,\bar g_2) + p(g_2,o)\cdot p(o,\bar g_2)\bigr)\cdot (1-\xi_2) \cdot \pi\bigl((g_1,\mathds{1}_{\{o,g_1\}}),(g_2,\mathds{1}_{\{o,g_2\}}) \bigr) >0.
$$
This finishes the proof.
\end{proof}
Recall that the assumption $|V_2|\geq 3$ in the last proposition is just stated for completeness, since the case $|V_1|=|V_2|=2$ was excluded at the beginning of Subsection \ref{subsec:free-products}.
Finally, we can prove:
\begin{proof}[Proof of Theorem \ref{th:capacity-existence}]
By Proposition \ref{prop:capacity-exit-decomposition} and (\ref{equ:speed}), 
\begin{equation}\label{equ:ap-convergence-ek}
\lim_{k\to\infty} \frac{\mathrm{Cap}(\mathbf{R}_{\e{k}})}{\e{k}} = \lim_{k\to\infty} \frac{\mathrm{Cap}(\mathbf{R}_{\e{k}})}{k} \frac{k}{\e{k}} =  \overline{\mathfrak{c}}\cdot \ell \quad \textrm{almost-surely}.
\end{equation}
Furthermore, by Lemma \ref{lem:Cap-difference} and $\mathbf{R}_{\e{\mathbf{k}(n)}}\subseteq \mathbf{R}_n \subseteq \mathbf{R}_{\e{\mathbf{k}(n)+1}}$, we have:
\begin{eqnarray}
\mathrm{Cap}\bigl(\mathbf{R}_n\bigr)-\mathrm{Cap}\bigl(\mathbf{R}_{\e{\mathbf{k}(n)}}\bigr) &\leq & \bigl| \mathbf{R}_n\setminus \mathbf{R}_{\e{\mathbf{k}(n)}}\bigr| \leq \bigl|  \mathbf{R}_{\e{\mathbf{k}(n)+1}}\setminus \mathbf{R}_{\e{\mathbf{k}(n)}}\bigr|, \label{equ:Cap-estimate1}\\
\mathrm{Cap}\bigl(\mathbf{R}_{\e{\mathbf{k}(n)+1}}-\mathrm{Cap}\bigl(\mathbf{R}_n\bigr)\bigr) &\leq & \bigl| \mathbf{R}_{\e{\mathbf{k}(n)+1}}\setminus \mathbf{R}_n \bigr| \leq \bigl|  \mathbf{R}_{\e{\mathbf{k}(n)+1}}\setminus \mathbf{R}_{\e{\mathbf{k}(n)}}\bigr|. \label{equ:Cap-estimate2}
\end{eqnarray}
Moreover, we have
\begin{eqnarray*}
&&\frac{\bigl|  \mathbf{R}_{\e{\mathbf{k}(n)+1}}\setminus \mathbf{R}_{\e{\mathbf{k}(n)}}\bigr|}{\e{\mathbf{k}(n)}} 
= \frac{\bigl|  \mathbf{R}_{\e{\mathbf{k}(n)+1}}\bigr| - \bigl| \mathbf{R}_{\e{\mathbf{k}(n)}}\bigr|}{\e{\mathbf{k}(n)}} \\
&\stackrel{(\ref{equ:range}),(\ref{equ:speed})}{=}&   \underbrace{\frac{\bigl|\mathbf{R}_{\e{\mathbf{k}(n)+1}}\bigr|}{\e{\mathbf{k}(n)+1}}}_{\to \mathfrak{r}} \underbrace{\frac{\e{\mathbf{k}(n)+1}}{\mathbf{k}(n)+1}}_{\to1/\ell}\underbrace{\frac{\mathbf{k}(n)+1}{\mathbf{k}(n)}}_{\to 1} \underbrace{\frac{\mathbf{k}(n)}{\e{\mathbf{k}(n)}}}_{\to \ell}
-\underbrace{\frac{\bigl|\mathbf{R}_{\e{\mathbf{k}(n)}}\bigr|}{\e{\mathbf{k}(n)}}}_{\to\mathfrak{r}} \xrightarrow{n\to\infty} 0 \quad \textrm{almost surely.}
\end{eqnarray*}
% OLD PART
%The above estimate provides
%\begin{eqnarray*}
%&&\frac{\mathrm{Cap}(\mathbf{R}_n)-\mathrm{Cap}(\mathbf{R}_{\e{\mathbf{k}(n)}})}{n} \\
%%&\leq & \frac{\bigl|\mathbf{R}_{\e{\mathbf{k}(n)+1}}\bigr|- \bigl|\mathbf{R}_{\e{\mathbf{k}(n)}} \bigr|}{n} \\
%&\leq& \frac{\bigl|\mathbf{R}_{\e{\mathbf{k}(n)+1}}\bigr|- \bigl|\mathbf{R}_{\e{\mathbf{k}(n)}}\bigr|}{\e{\mathbf{k}(n)}} \frac{\e{\mathbf{k}(n)}}{n} \\
%&\stackrel{(\ref{equ:speed}),(\ref{equ:convergence-e_k(n)}), (\ref{equ:range})}{=}&   \underbrace{\frac{\bigl|\mathbf{R}_{\e{\mathbf{k}(n)+1}}\bigr|}{\e{\mathbf{k}(n)+1}}}_{\to \mathfrak{r}} \underbrace{\frac{\e{\mathbf{k}(n)+1}}{\mathbf{k}(n)+1}}_{\to1/\ell}\underbrace{\frac{\mathbf{k}(n)+1}{\mathbf{k}(n)}}_{\to 1} \underbrace{\frac{\mathbf{k}(n)}{\e{\mathbf{k}(n)}}}_{\to \ell}
%-\underbrace{\frac{\bigl|\mathbf{R}_{\e{\mathbf{k}(n)}}\bigr|}{\e{\mathbf{k}(n)}}}_{\to\mathfrak{r}} \xrightarrow{n\to\infty} 0 \quad \textrm{a.s..}
%\end{eqnarray*}
Therefore, the bounds in (\ref{equ:Cap-estimate1}) and (\ref{equ:Cap-estimate2}) yield:
\begin{eqnarray*}
\limsup_{n\to\infty} \frac{\mathrm{Cap}(\mathbf{R}_n)-\mathrm{Cap}(\mathbf{R}_{\e{\mathbf{k}(n)}})}{\e{\mathbf{k}(n)}}  &\leq & 0 \quad \textrm{ and}\\
\limsup_{n\to\infty} \frac{\mathrm{Cap}(\mathbf{R}_{\e{\mathbf{k}(n)+1}})-\mathrm{Cap}(\mathbf{R}_n)}{\e{\mathbf{k}(n)}}  &\leq & 0 \quad \textrm{almost surely.}
\end{eqnarray*}
Together with (\ref{equ:ap-convergence-ek}) the above  inequalities imply 
\begin{eqnarray*}
\limsup_{n\to\infty} \frac{\mathrm{Cap}(\mathbf{R}_n)}{\e{\mathbf{k}(n)}}  &=& \limsup_{n\to\infty} \frac{\mathrm{Cap}(\mathbf{R}_n)-\mathrm{Cap}(\mathbf{R}_{\e{\mathbf{k}(n)}})}{\e{\mathbf{k}(n)}} + \underbrace{\frac{\mathrm{Cap}(\mathbf{R}_{\e{\mathbf{k}(n)}})}{\e{\mathbf{k}(n)}}}_{\to  \overline{\mathfrak{c}}\cdot \ell } \leq  \overline{\mathfrak{c}}\cdot \ell,\\
\liminf_{n\to\infty} \frac{\mathrm{Cap}(\mathbf{R}_n)}{\e{\mathbf{k}(n)}}  &=& \liminf_{n\to\infty} \frac{\mathrm{Cap}(\mathbf{R}_n)-\mathrm{Cap}(\mathbf{R}_{\e{\mathbf{k}(n)+1}})}{\e{\mathbf{k}(n)}} + \underbrace{\frac{\mathrm{Cap}(\mathbf{R}_{\e{\mathbf{k}(n)+1}})}{\e{\mathbf{k}(n)}}}_{\to  \overline{\mathfrak{c}}\cdot \ell } \geq \overline{\mathfrak{c}}\cdot \ell,
\end{eqnarray*}
that is, we have shown that
$$
\lim_{n\to\infty} \frac{\mathrm{Cap}(\mathbf{R}_n)}{\e{\mathbf{k}(n)}} =\overline{\mathfrak{c}}\cdot \ell  \quad \textrm{almost surely.}
$$
Finally, we obtain the proposed convergence statement with (\ref{equ:convergence-e_k(n)}):
\begin{eqnarray}
\frac{\mathrm{Cap}(\mathbf{R}_n)}{n} &= &  \underbrace{\frac{\mathrm{Cap}(\mathbf{R}_n)}{\e{\mathbf{k}(n)}}}_{\to \overline{\mathfrak{c}}\cdot \ell } \cdot \underbrace{\frac{\e{\mathbf{k}(n)}}{n}}_{\to 1} \label{equ:capacity-formula}
\xrightarrow{n\to\infty}  \mathfrak{c}:= \overline{\mathfrak{c}}\cdot \ell >0 \quad \textrm{almost surely}.\nonumber
\end{eqnarray}
\end{proof}

\begin{remark}
The capacity of the range of random walks on $\mathbb{Z}^d$ is linked with 
hitting probabilities from afar. In the case of random walks on free products we have some different behaviour which we want to discuss briefly.
\par
For simple random walk on $\mathbb{Z}^d$, a well-known formula for the capacity of a finite set $A\subset \mathbb{Z}^d$ is given by
$$
\mathrm{Cap}(A)=\lim_{\Vert x\Vert_2\to\infty} \frac{\P_x[S_A<\infty]}{G(x,0|1)},
$$
where $G(x,0|1)=\sum_{n\geq 0} p^{(n)}_d(x,0)$ is the Green function w.r.t. simple random walk on $\mathbb{Z}^d$ with $p^{(n)}_d(x,0)$ denoting the $n$-step transition probabilities of walking from $x\in\mathbb{Z}^d$ to $0$.
\par
In the case of random walks on free products we have a similar, but different connection between the hitting probabilities from afar and the probabilities $\P_x[S_A=\infty]$. For any $x\in V$, we obtain by decomposing according to the last visit of any  finite set $R\subset V$ (recall transience of our random walks) the following last passage decomposition for the random walk on the free product $V$:
$$
\P_x[S_R<\infty] = \sum_{y\in R} G(x,y|1)\cdot \P_y[S_R=\infty].
$$
Dividing this equation by $G(x,o|1)$ gives
$$
\frac{\P_x[S_R<\infty]}{G(x,o|1)} = \sum_{y\in R} \frac{G(x,y|1)}{G(x,o|1)}\cdot \P_y[S_R=\infty] \stackrel{(\ref{equ:G-F-L-equations})}{=} \sum_{y\in R} \frac{F(x,y|1)}{F(x,o|1)} \cdot \frac{G(y,y|1)}{G(o,o|1)}\cdot \P_y[S_R=\infty].
$$
The quotients $ \frac{F(x,y|1)}{F(x,o|1)}$ can be simplified as in the proof of Proposition \ref{prop:G-convergence1}. E.g., if $x=x_1\dots x_m$, $y=y_1\dots y_n\in V$ and if the common prefix of $x$ and $y$ of maximal length is given by $x_1\dots x_k=y_1\dots y_k$, where $k<\min\{m,n\}$, then $ \frac{F(x,y|1)}{F(x,o|1)}$ simplifies to
$$
\frac{F(x,y|1)}{F(x,o|1)} = \frac{\prod_{i=k+1}^m F(x_i,o|1) \cdot F(x_1\dots x_{k+1},y|1)}{\prod_{i=1}^m F(x_i,o|1)} = \frac{F(x_1\dots x_{k+1},y|1)}{\prod_{i=1}^k F(x_i,o|1)}.
$$
Similar simplifications can be performed if $k=m$ or $k=$n.
In contrast to $\mathrm{Cap}(R)$, where the summands $\P_y[S_R=\infty]$ have weight $1$ for all $y\in R$, these summands now get  the weights $\frac{G(x,y|1)}{G(x,o|1)}$, which depend only on the maximal common prefix of $x$ and $y$, but not on how far away $x$ is located from $R$.

\end{remark}

\section{Central Limit Theorem for the Capacity of the Range}
\label{sec:clt}

In this section we will prove the Central Limit Theorem \ref{th:clt}. To this end, we will follow a similar basic reasoning as in \cite[Section 4]{gilch:22} with the introduction of regeneration times whose properties we will use in our proofs. However,  several non-straightforward, essential extensions to the present setting are required. Throughout this section we assume that there exists $x_0\in V$ and $\kappa\in\N$ such that $p^{(\kappa)}(x_0,x_0)>0$, which excludes degenerate cases; see \mbox{Remark \ref{rem:degen-cases}.} This assumption will be needed in the proofs of Lemma \ref{lem:VarD>0} and  \mbox{Theorem \ref{th:clt}.} In particular, we still assume $\varrho<1$, which excludes the recurrent case $|V_1|=|V_2|=2$.
\par
For $x\in V$, denote by 
$$
T_x:=\inf\bigl\lbrace m\in\N_0 \mid X_m=x\bigr\rbrace,
$$
the stopping time of the first visit to $x$.
Choose and  fix now for the rest of this section any $\mathfrak{g}\in V_1^\times$. In the following we stop the Markov chain $(\mathbf{W}_k,\psi_k)_{k\in\N}$ at those intermediate random times $k$ when $X_{\e{k}}$ ends with letter $\mathfrak{g}$ (that is, when $\mathbf{W}_k=\mathfrak{g}$) and when the vertex $X_{\e{k}}$ is hit for the first time at time $\e{k}$ (that is, when $\e{k}=T_{X_{\e{k}}}$). 
More formally, define the random times
\begin{eqnarray*}
\tau_0 &:=& \inf\bigl\lbrace m\in\N \mid \mathbf{W}_{m}=\mathfrak{g}, \e{m}=T_{X_{\e{m}}}\bigr\rbrace,\\
\forall k\geq 1: \ \tau_k &:=& \inf\bigl\lbrace m> \tau_{k-1} \mid \mathbf{W}_{m}=\mathfrak{g}, \e{m}=T_{X_{\e{m}}}\bigr\rbrace;
\end{eqnarray*}
compare with \cite[Section 4]{gilch:22}.
Recall from Proposition \ref{prop:pos-rec-process} that $(\mathbf{W}_k,\psi_k)_{k\in\mathbb{N}_0}$ is positive recurrent; hence, the event $[\mathbf{W}_k=\mathfrak{g}]$ occurs for infinitely many indices $k$  with \mbox{probability $1$.} Each time when the random walk $(X_n)_{n\in\N_0}$ visits a word $w\in V$ ending  with $\mathfrak{g}$ for the first time, the random walk has strictly positive probability to remain in $C(w)$ for forever (namely with probability $1-\xi_1>0$ which is independent of $w$; see end of Subsection \ref{subsec:rw-on-fp}). Positive recurrence of $(\mathbf{W}_k,\psi_k)_{k\in\mathbb{N}_0}$ together with a standard geometric argument yields that  $\tau_k<\infty$ almost surely for all $k\in\N_0$. For $k\in\N_0$, set
$$
\TT_k:=\e{\tau_k}
$$
and define for $i\in\N$
\begin{eqnarray*}
\widetilde{\mathcal{C}}_i &:=& \sum_{j=\tau_{i-1}}^{\tau_i-1}\mathcal{C}_j,\\
\mathcal{D}_i &:=&  \widetilde{\mathcal{C}}_i - \mathfrak{c}\cdot (\TT_i-\TT_{i-1}).
\end{eqnarray*}
The following proposition will play an important role  in the proofs later:
\begin{proposition}\label{prop:TT-exp-moments}
$(\TT_i-\TT_{i-1})_{i\in\N}$ is an i.i.d. sequence of random variables. Furthermore,  $\TT_0$ and $\TT_i-\TT_{i-1}$, $i\in\N$, have exponential moments.
\end{proposition}
\begin{proof}
See \cite[Prop. 4.2, Lemma 4.3, Prop. 4.5]{gilch:22}
\end{proof}
%In \cite[Prop. 4.2 \& Lemma 4.3]{gilch:22} it is shown that $\TT_0$ and $\TT_i-\TT_{i-1}$, $i\in\N$, have exponential moments. Furthermore,  $(\TT_i-\TT_{i-1})_{i\in\N}$ is an i.i.d. sequence of random variables, see \cite[Prop. 4.5]{gilch:22}. 
Moreover, we have the following important inequality:
\begin{lemma} \label{lem:Ci-estimate}
For all $\omega\in\Omega_0$,
$$
\widetilde{\mathcal{C}}_1(\omega) \leq \TT_1(\omega)-\TT_0(\omega) +1.
$$
\end{lemma}
\begin{proof}
Let  $\omega\in\Omega_0$. Then:
\begin{eqnarray*}
\widetilde{\mathcal{C}}_1(\omega) %&=& \sum_{j=\tau_0}^{\tau_1-1} \mathcal{C}_j \\
&=& \underbrace{\sum_{j=\tau_0}^{\tau_1-1} C_j^{(I)}(\omega)}_{\leq \e{\tau_1}(\omega)-\e{\tau_0}(\omega)-\bigl(\tau_1(\omega)-\tau_0(\omega)\bigr)} + \P_{X_{\e{j}}(\omega)}\bigl[ S_{\mathcal{R}_j(\omega)}=\infty, \forall n\geq 1: X_n\in C\bigl(X_{\e{j}}(\omega)\bigr)\bigr] \\
&&\quad + \P_{X_{\e{j+1}}(\omega)}\bigl[ S_{\mathcal{R}_j(\omega)}=\infty, \forall n\geq 1: X_n\notin C\bigl(X_{\e{j}+1}(\omega)\bigr)\bigr] 
\end{eqnarray*}
\begin{eqnarray*}
&\leq & \TT_1(\omega)-\TT_0(\omega) -\bigl(\tau_1(\omega)-\tau_0(\omega)\bigr) \\
&&\quad + \sum_{j=\tau_0+1}^{\tau_1-1} \underbrace{P_{X_{\e{j}}(\omega)}\left[ \substack{S_{\mathcal{R}_j(\omega)}=\infty, \\ \forall n\geq 1: X_n\in C\bigl(X_{\e{j}}(\omega)\bigr)}\right]  + \P_{X_{\e{j}}(\omega)}\left[\substack{S_{\mathcal{R}_j(\omega)}=\infty, \\ \forall n\geq 1: X_n\notin C\bigl(X_{\e{j}}(\omega)\bigr)}\right]}_{\leq \mathbb{P}_{X_{\e{j}}(\omega)}[S_{\mathcal{R}_j(\omega)}=\infty]\leq 1} \\
&&\quad + P_{X_{\e{\tau_0}}(\omega)}\left[ \substack{S_{\mathcal{R}_{\tau_0}(\omega)}=\infty, \\ \forall n\geq 1: X_n\in C\bigl(X_{\e{\tau_0}}(\omega)\bigr)}\right]  + \P_{X_{\e{\tau_1}}(\omega)}\left[ \substack{S_{\mathcal{R}_{\tau_1}(\omega)}=\infty, \\ \forall n\geq 1: X_n\notin C\bigl(X_{\e{\tau_1}}(\omega)\bigr)}\right] \\
&\leq &\TT_1(\omega)-\TT_0(\omega) -\bigl(\tau_1(\omega)-\tau_0(\omega)\bigr) + \bigl(\tau_1(\omega)-\tau_0(\omega)-1\bigr)+2 =  \TT_1(\omega)-\TT_0(\omega)  +1.
\end{eqnarray*}
\end{proof}

In the following we collect some basic properties of the sequence $(\mathcal{D}_i)_{i\in\N_0}$.

\begin{lemma}
$\displaystyle  \mathrm{Var}(\mathcal{D}_1)<\infty$.
\end{lemma}
\begin{proof}
%By construction of $\widetilde{\mathcal{C}}_1$, $\TT_0$ and $\TT_1$, we have
%\begin{equation}\label{equ:estimate-C1}
%0\leq \widetilde{\mathcal{C}}_1 \leq \TT_1-\TT_0 +1;
%\end{equation}
%indeed, each summand of 
Since $\mathfrak{c}> 0$ and $\widetilde{\mathcal{C}}_1\geq 0$ we have  
$$
-\mathfrak{c} \cdot (\TT_1-\TT_{0}) \leq \mathcal{D}_1 \leq \widetilde{\mathcal{C}}_1.
$$
Lemma \ref{lem:Ci-estimate} together with $\mathfrak{c}\leq 1$ yields:
$$
\bigl|  \mathcal{D}_1\bigr| \leq \max\bigl\lbrace \widetilde{\mathcal{C}}_1,\mathfrak{c} \cdot (\TT_1-\TT_{0}) \bigr\rbrace \leq \TT_1-\TT_{0}+1\quad \textrm{almost surely.} 
$$
Since $ \TT_1-\TT_{0}$ has exponential moments, we get $\mathcal{D}_1\in\mathcal{L}_2$, that is, $\mathrm{Var}(\mathcal{D}_1)<\infty$.
\end{proof}

The reasoning in the proof of the following proposition is analogously to the proof of \mbox{\cite[Proposition 4.5]{gilch:22}.} Nonetheless, we give a proof since the proposition is another key ingredient in the proof of Theorem \ref{th:clt}.

\begin{proposition}
$\displaystyle (\mathcal{D}_i)_{i\in\N}$ is an i.i.d. sequence of random variables.
\end{proposition}
\begin{proof}
We introduce some notation in order to decompose some set of paths accordingly.
Let  $i,m\in\N$, $j\in\N_0$ and $z\in\mathbb{R}$. For $x_0\in V$ with $\P[X_{\TT_j}=x_0]>0$, denote by $\mathcal{P}^{(1)}_{j,x_0,m}$ the set of  all paths \mbox{$(o,w_1,\dots,w_m=x_0)\in V^{m+1}$} of length $m$ such that 
$$
\mathbb{P}\Bigl[X_1=w_1,\dots,X_{m-1}=w_{m-1},X_m=x_0,\TT_j=m\Bigr]>0,
$$
that is, each path in $\mathcal{P}^{(1)}_{j,x_0,m}$ allows to generate $\TT_j$ with $X_{\TT_j}=x_0$ at time $m$ with positive probability. In particular, there exists such a path due to choice of $x_0$ with $\P[X_{\TT_j}=x_0]>0$.
\par
Furthermore, for $x_0\in V$ with $\P[X_{\TT_{i-1}}=x_0]>0$, denote by $\mathcal{P}^{(2)}_{i,x_0,n,z}$ the set of paths $(x_0,y_1,\dots,y_n)\in V^{n+1}$ of length $n\in\N$ such that 
$$
\mathbb{P}\Bigl[ \exists t\in\N: X_{t}=x_0,X_{t+1}=y_1,\dots,X_{t+n}=y_n,\TT_{i-1}=t,\TT_{i}=t+n,\mathcal{D}_i=z\Bigr]> 0,
$$
that is, each path in $\mathcal{P}^{(2)}_{i,x_0,n,z}$ allows to generate $X_{\TT_{i-1}}=x_0$, $\TT_{i}-\TT_{i-1}=n$ and $\mathcal{D}_i=z$. 
In particular, we have $y_i\in C(x_0)$, that is, we can write $y_i=x_0y_i'$.
\par
First, we show that $(\mathcal{D}_i)_{i\in\N}$ is a sequence of  identically distributed random variables. 
By decomposing all paths until time $\TT_{i}$ into the part until time $\TT_{i-1}$ and into the part between the random times $\TT_{i-1}$ and $\TT_{i}$  we obtain with Lemma \ref{lem:cone-probs}:
\begin{eqnarray*}
\P[\mathcal{D}_i=z] &=& \sum_{\substack{x_0\in V:\\ \P[X_{\TT_{i-1}}=x_0]>0}} \P\bigl[ X_{\TT_{i-1}}=x_0,\mathcal{D}_i=z\bigr]\\
&=& \sum_{\substack{x_0\in V:\\ \P[X_{\TT_{i-1}}=x_0]>0}}  \sum_{\substack{ n\in\N,\\ (x_0,x_0y_1',\dots,x_0y_n')\in \mathcal{P}^{(2)}_{i,x_0,n,z}}}  \P\left[ \substack{X_{\TT_{i-1}}=x_0,\\ X_{\TT_{i-1}+1}=x_0y_1',\dots, X_{\TT_{i-1}+n}=x_0y_n',\\ \forall l\geq 1: X_l \in C(x_0y_n')}
\right]\\
&=&  \sum_{\substack{x_0\in V:\\ \P[X_{\TT_{i-1}}=x_0]>0}} \sum_{m\geq 1} \sum_{(o,w_1,\dots,w_m)\in \mathcal{P}^{(1)}_{i-1,x_0,m}} \P\bigl[X_1=w_1,\dots,X_m=w_m\bigr] \\
&&\quad \cdot \sum_{n\geq 1} \sum_{(x_0,x_0y_1',\dots,x_0y_n')\in \mathcal{P}^{(2)}_{i,x_0,n,z}} \P_{x_0}\bigl[X_1=x_0y_1',\dots,X_n=x_0y_n'\bigr] \\
&&\quad \cdot \P_{x_0y_n'}\bigl[\forall l\geq 1: X_l \in C(x_0y_n')\bigr]\\
&\stackrel{L.\ref{lem:cone-probs}}{=}&\sum_{\substack{x_0\in V:\\ \P[X_{\TT_{i-1}}=x_0]>0}} \sum_{m\geq 1} \sum_{(o,w_1,\dots,w_m)\in \mathcal{P}^{(1)}_{i-1,x_0,m}} \P\bigl[X_1=w_1,\dots,X_m=w_m\bigr] \\
&&\quad \cdot \sum_{n\geq 1} \sum_{(x_0,x_0y_1',\dots,x_0y_n')\in \mathcal{P}^{(2)}_{i,x_0,n,z}} \P_{\mathfrak{g}}\bigl[X_1=\mathfrak{g}y_1,\dots,X_n=\mathfrak{g}y_n\bigr] \cdot \bigl(1-\xi_1\bigr).
\end{eqnarray*}
Note that each path $(x_0,x_0y_1',\dots,x_0y_n')\in \mathcal{P}^{(2)}_{i,x_0,n,z}$ lies completely in the cone $C(x_0)$ and that $x_0$ ends with letter $\mathfrak{g}$. Therefore, there is a natural 1-to-1 correspondence between paths in $\mathcal{P}^{(2)}_{i,x_0,n,z}$ and $\mathcal{P}^{(2)}_{1,\mathfrak{g},n,z}$ established by the vertex-wise shift \mbox{$C(x_0)\ni x_0g \mapsto \mathfrak{g}g\in C(\mathfrak{g})$,} that is, we can map
$$
\mathcal{P}^{(2)}_{i,x_0,n,z} \ni (x_0,x_0y_1',\dots,x_0y_n')\mapsto (\mathfrak{g},\mathfrak{g}y_1',\dots,\mathfrak{g}y_n')\in \mathcal{P}^{(2)}_{1,\mathfrak{g},n,z},
$$ 
which is a bijective and measure preserving mapping according to Lemma \ref{lem:cone-probs}.
At this point recall once again that $x_0$ ends with letter $\mathfrak{g}$ by choice of $x_0$ with $\P[X_{\TT_{j-1}}=x_0]>0$, which ensures that the words $\mathfrak{g}y_i'$ are well-defined. Furthermore, we have
\begin{eqnarray*}
&&\sum_{\substack{x_0\in V:\\ \P[X_{\TT_{i-1}}=x_0]>0}}  \sum_{\substack{m\geq 1,\\ (o,w_1,\dots,w_m)\in \mathcal{P}^{(1)}_{i-1,x_0,m}}} \P\bigl[X_1=w_1,\dots,X_m=w_m\bigr] \cdot \underbrace{(1-\xi_1)}_{=\mathbb{P}_{w_m}\bigl[\forall n\geq 1: X_n\in C(w_m)\bigr]}\hfill\\
%&=&\sum_{\substack{x_0\in V:\\ \P[X_{\TT_{i-1}}=x_0]>0}} \sum_{\substack{m\geq 1,\\ (o,w_1,\dots,w_m)\in \mathcal{P}^{(1)}_{i-1,x_0,m}}} \P\bigl[X_1=w_1,\dots,X_m=w_m\bigr] \cdot \mathbb{P}_{w_m}\bigl[\forall n\geq 1: X_n\in C(w_m)\bigr]\\
&=& \sum_{\substack{x_0\in V:\\ \P[X_{\TT_{i-1}}=x_0]>0}} \P\bigl[X_{\TT_{i-1}}=x_0\bigr]
=\P[\TT_{i-1}<\infty]=1.
\end{eqnarray*}
Therefore,
\begin{eqnarray}\label{equ:Li-formula}
\P[\mathcal{D}_i=z] = \sum_{n\geq 1} \sum_{(\mathfrak{g},y_1,\dots,y_n)\in \mathcal{P}^{(2)}_{1,\mathfrak{g},n,z}} \P_{\mathfrak{g}}\bigl[X_1=y_1,\dots,X_n=y_n\bigr].
\end{eqnarray}
Since the probabilities on the right hand side do not depend on $i$ any more, we have proven that the $\mathcal{D}_i$'s all have the same distribution. 
\par
For the proof of independence of the sequence $(\mathcal{D}_i)_{i\in\N}$, we only show independence of $\mathcal{D}_1$ and $\mathcal{D}_2$; the general proof follows completely analogously, is however very lengthy and  therefore we omit it. Let  $d_1,d_2\in\mathbb{R}$. We make decomposition according to the values of $\TT_0,\TT_1,\TT_2$ and $X_{\TT_0}, X_{\TT_1},X_{\TT_2}$ and obtain:
\begin{eqnarray}
&&\P[\mathcal{D}_1=d_1,\mathcal{D}_2=d_2] \nonumber\\
%&=&  \sum_{\substack{x_0\in V:\\ \P[X_{\e{\tau_{0}}}=x_0]>0}}\sum_{\substack{y\in V:\\ \P[X_{\e{\tau_{1}}}=y]>0}}\sum_{\substack{z\in V:\\ \P[X_{\e{\tau_{2}}}=z]>0}}\sum_{m,n_1,n_2\geq 1}
%\P\left[\begin{array}{c} e_{\tau_0}=m, X_{e_{\tau_0}}=x_0,\\  e_{\tau_1}=m+n_1, X_{e_{\tau_1}}=y, \\  e_{\tau_2}=m+n_1+n_2, X_{e_{\tau_2}}=z,\\ \mathcal{D}_1=d_1, \mathcal{D}_2=d_2\end{array}\right]\nonumber\\
&=& \sum_{\substack{x_0\in V:\\ \P[X_{\TT_{0}}=x_0]>0}}\sum_{m,n_1,n_2\geq 1} \sum_{\substack{(o,w_1,\dots,w_m)\in \mathcal{P}^{(1)}_{0,x_0,m}, \\ (x_0,y_1,\dots,y_{n_1})\in \mathcal{P}^{(2)}_{1,x_0,n_1,d_1},\\ (y_{n_1},z_1,\dots,z_{n_2})\in \mathcal{P}^{(2)}_{2,y_{n_1},n_2,d_2}}} \P\left[\begin{array}{c}
X_1=w_1,\dots,X_m=w_m,\\
X_{m+1}=y_1,\dots,X_{m+n_1}=y_{n_1},\\
X_{m+n_1+1}=z_1,\dots,\\ X_{m+n_1+n_2}=z_{n_2},\\
\forall l\geq 1: X_l\in C(z_{n_2})
\end{array}\right]\nonumber\\
&=& \sum_{\substack{x_0\in V:\\ \P[X_{\TT_{0}}=x_0]>0}} \sum_{m\geq 1} \sum_{(o,w_1,\dots,w_m)\in \mathcal{P}^{(1)}_{0,x_0,m}} \P\bigl[X_1=w_1,\dots,X_m=w_m\bigr]\nonumber \\
&&\quad \cdot \sum_{n_1\geq 1} \sum_{(x_0,y_1,\dots,y_{n_1})\in \mathcal{P}^{(2)}_{1,x_0,n_1,d_1}} \P_{x_0}\bigl[X_1=y_1,\dots,X_{n_1}=y_{n_1}\bigr] \nonumber\\
&& \quad \cdot \sum_{n_2\geq 1} \sum_{(y_{n_1},z_1,\dots,z_{n_2})\in \mathcal{P}^{(2)}_{2,y_{n_1},n_2,d_2}} \P_{y_{n_1}}\bigl[X_1=z_1,\dots,X_{n_2}=z_{n_2}\bigr]  \nonumber\\
&&\quad 
\cdot \P_{z_{n_2}}\bigl[\forall l\geq 1: X_l \in C(z_{n_2})\bigr].\label{equ:L1-L2}
\end{eqnarray}
For the last equation we recall the equation $\mathbb{P}_{z_{n_2}}\bigl[ \forall k\geq 1: X_k\in C(z_{n_2})\bigr]=1-\xi_1$.
Similarily, decomposing according to  the  values of $\TT_1$ and $X_{\TT_1}$, we get
\begin{eqnarray}
1 &=& \P[\TT_1<\infty ]\nonumber\\
&=& \sum_{\substack{z_0\in V:\\ \P[X_{\TT_{1}}=z_0]>0}} \sum_{m_1\geq 1} \sum_{(o,w_1,\dots,w_{m_1})\in \mathcal{P}^{(1)}_{1,z_0,m_1}} \P\left[ \begin{array}{c} X_1=w_1,\dots,X_{m_1}=w_{m_1},\\ \forall l>m: X_l\in C(z_0)\end{array}\right] \label{equ:e-tau1<infty}\\
&=& \sum_{\substack{z_0\in V:\\ \P[X_{\TT_{1}}=z_0]>0}} \sum_{m_1\geq 1} \sum_{(o,w_1,\dots,w_{m_1})\in \mathcal{P}^{(1)}_{1,z_0,m_1}} \P\bigl[X_1=w_1,\dots,X_{m_1}=w_{m_1}\bigr] \cdot (1-\xi_1).\nonumber
\end{eqnarray}
Now observe that, for $z_0\in V$ with $\P[X_{\TT_{1}}=z_0]>0$, the mapping
$$
\mathcal{P}^{(2)}_{2,y_{n_1},n_2,d_2} \ni (y_{n_1},z_1,\dots,z_{n_2}) \mapsto
(z_0,z_0z_1',\dots,z_0z_{n_2}')\in \mathcal{P}^{(2)}_{1,z_0,n_2,d_2},
$$
where $z_i:=y_{n_1}z_i'$ for $i\in\{1,\dots,n_2\}$, is measure-preserving (see Lemma \ref{lem:cone-probs}), that is, 
\begin{equation}\label{equ:path-shift}
\P_{y_{n_1}}[X_1=y_{n_1}z_1',\dots,X_{n_2}=y_{n_1}z_{n_2}']=\P_{z_0}[X_1=z_0z_1',\dots,X_{n_2}=z_0z_{n_2}'].
\end{equation}
Furthermore, we remark that
\begin{eqnarray*}
\P[\mathcal{D}_1=d_1] &=& \sum_{\substack{x_0\in V:\\ \P[X_{\TT_{0}}=x_0]>0}} \sum_{m\geq 1} \sum_{(o,w_1,\dots,w_m)\in \mathcal{P}^{(1)}_{0,x_0,m}} \P\bigl[X_1=w_1,\dots,X_m=w_m\bigr] \\
&&\quad \cdot \sum_{n_1\geq 1} \sum_{(x_0,y_1,\dots,y_{n_1})\in \mathcal{P}^{(2)}_{1,x_0,n_1,d_1}} \P_{x_0}\bigl[X_1=y_1,\dots,X_{n_1}=y_{n_1}\bigr] \cdot (1-\xi_1), \\
\P[\mathcal{D}_2=d_2] &=& \sum_{\substack{z_0\in V:\\ \P[X_{\TT_{1}}=z_0]>0}} \sum_{m_1\geq 1} \sum_{(o,w_1,\dots,w_{m_1})\in \mathcal{P}^{(1)}_{1,z_0,m_1}} \P\bigl[X_1=w_1,\dots,X_{m_1}=w_{m_1}\bigr]  \\
&&\quad \cdot \sum_{n_2\geq 1} \sum_{(z_0,z_1,\dots,z_{n_2})\in \mathcal{P}^{(2)}_{2,z_0,n_2,d_2}} \P_{z_0}\bigl[X_1=z_1,\dots,X_{n_2}=z_{n_2}\bigr] \cdot (1-\xi_1).
\end{eqnarray*}
The required independence equation follows now from  (\ref{equ:L1-L2}) with the help of  (\ref{equ:e-tau1<infty}) and (\ref{equ:path-shift}):
\begin{eqnarray*}
&&\P[\mathcal{D}_1=d_1,\mathcal{D}_2=d_2] \\
&=& \sum_{\substack{x_0\in V:\\ \P[X_{\TT_{0}}=x_0]>0}} \sum_{m\geq 1} \sum_{(o,w_1,\dots,w_m)\in \mathcal{P}^{(1)}_{0,x_0,m}} \P\bigl[X_1=w_1,\dots,X_m=w_m\bigr] \\
&&\quad \cdot \sum_{n_1\geq 1} \sum_{(x_0,y_1,\dots,y_{n_1})\in \mathcal{P}^{(2)}_{1,x_0,n_1,d_1}} \P_{x_0}\bigl[X_1=y_1,\dots,X_{n_1}=y_{n_1}\bigr] \\
&&\quad \cdot \underbrace{\sum_{\substack{z_0\in V:\\ \P[X_{\TT_{1}}=z_0]>0}} \sum_{m_1\geq 1} \sum_{(o,\bar w_1,\dots,\bar w_{m_1})\in \mathcal{P}^{(1)}_{1,z_0,m_1}} \P\bigl[X_1=\bar w_1,\dots,X_{m_1}=\bar w_{m_1}\bigr] \cdot (1-\xi_1)}_{=1} \\
&&\quad \cdot \sum_{n_2\geq 1} \sum_{(y_{n_1},z_1,\dots,z_{n_2})\in \mathcal{P}^{(2)}_{2,y_{n_1},n_2,d_2}} \P_{y_{n_1}}\bigl[X_1=z_1,\dots,X_{n_2}=z_{n_2}\bigr] \cdot (1-\xi_1)\\
\end{eqnarray*}
\pagebreak[5]
\begin{eqnarray*}
&\stackrel{(\ref{equ:path-shift})}{=}& \biggl(\sum_{\substack{x_0\in V:\\ \P[X_{\TT_{0}}=x_0]>0}} \sum_{m\geq 1} \sum_{(o,w_1,\dots,w_m)\in \mathcal{P}^{(1)}_{0,x_0,m}} \P\bigl[X_1=w_1,\dots,X_m=w_m\bigr] \\
&&\quad \cdot \sum_{n_1\geq 1} \sum_{(x_0,y_1,\dots,y_{n_1})\in \mathcal{P}^{(2)}_{1,x_0,n_1,d_1}} \P_{x_0}\bigl[X_1=y_1,\dots,X_{n_1}=y_{n_1}\bigr] \cdot (1-\xi_1)\biggr)\\
&&\quad \cdot \biggl( \sum_{\substack{z_0\in V:\\ \P[X_{\TT_{1}}=z_0]>0}} \sum_{m_1\geq 1} \sum_{(o,\bar w_1,\dots,\bar w_{m_1})\in \mathcal{P}^{(1)}_{1,z_0,m_1}} \P\bigl[X_1=\bar w_1,\dots,X_{m_1}=\bar w_{m_1}\bigr]  \\
&&\quad \cdot \sum_{n_2\geq 1} \sum_{(z_0,z_1,\dots,z_{n_2})\in \mathcal{P}^{(2)}_{2,z_0,n_2,d_2}} \P_{z_0}\bigl[X_1=z_1,\dots,X_{n_2}=z_{n_2}\bigr]\cdot (1-\xi_1)\biggr) \\
&=& \P[L_1=d_1]\cdot \P[L_2=d_2].
\end{eqnarray*} 
This finishes the proof of the proposition.
%The proof works completely analogously to the proof of \cite[Proposition 4.5]{gilch:22} without any additional technical adaptions required. Therefore, we omit the lengthy proof and leave it as an exercise to the reader.
\end{proof}
The next step is to consider those times $\TT_m$ which occur until time $n\in\N$. For this purpose, define for $n\in\N_0$ 
$$
\mathbf{t}(n):= \sup\bigl\lbrace m\in\N_0 \mid \TT_m \leq n\bigr\rbrace.
$$
In \cite[(4.8)]{gilch:22} it is shown that
\begin{eqnarray}
%\frac{1}{\mathbf{t}(n)}\sum_{j=1}^{\mathbf{t}(n)} \TT_j-\TT_{j-1} & \xrightarrow{n\to\infty} & \mathbb{E}\bigl[\TT_1-\TT_0\bigr] \quad \textrm{almost surely},\nonumber\\
%\textrm{and } \quad\quad \quad\quad
%\frac{\e{\tau_{\mathbf{t}(n)}}}{\mathbf{t}(n)}=
\frac{\TT_{\mathbf{t}(n)}}{\mathbf{t}(n)}& \xrightarrow{n\to\infty} & \mathbb{E}\bigl[\TT_1-\TT_0\bigr] \quad \textrm{almost surely.}\label{equ:TT-conv}
\end{eqnarray}

\begin{lemma}\label{lem:VarD>0}
We have:
\begin{enumerate}%[label=(\roman*)]
\item $\displaystyle \mathfrak{c}=\frac{\mathbb{E}\bigl[\widetilde{\mathcal{C}}_1 \bigr]}{\E[\TT_1-\TT_0]}$ almost surely.
\item $\mathbb{E}[\mathcal{D}_1]=0$.
\item Assume that there are $\bar g_0\in V$ and $\kappa\in\N$ such that $\mathbb{P}[X_\kappa=\bar g_0\mid X_0=\bar g_0]>0$. Then: $\mathrm{Var}(\mathcal{D}_1) >0$.
\end{enumerate}
\end{lemma}
\begin{proof}\,
\begin{enumerate}%[label=(\roman*)]
\item Due to Lemma \ref{lem:Ci-estimate} together with existence of exponential moments of $\TT_1-\TT_0$ (see Proposition \ref{prop:TT-exp-moments}), we have $\E\bigl[ \widetilde{\mathcal{C}}_1 \bigr] <\infty$. Now we obtain, completely analogously to the proof of \mbox{Proposition \ref{prop:capacity-exit-decomposition}} (replace the exit times $\mathbf{e}_i$ by $\TT_i$),
$$
\lim_{n\to\infty} \frac{\mathrm{Cap}(\mathbf{R}_{\TT_k})}{k} = \E\bigl[ \widetilde{\mathcal{C}}_1 \bigr] \quad \textrm{almost surely.}
$$
Together with (\ref{equ:TT-conv}) we obtain
%\begin{equation}\label{equ:c2}
$$
  \mathfrak{c}=\lim_{n\to\infty} \frac{\mathrm{Cap}(\mathbf{R}_{\TT_{\mathbf{t}(n)}})}{\TT_{\mathbf{t}(n)}}
=\lim_{n\to\infty} \frac{\mathrm{Cap}(\mathbf{R}_{\TT_{\mathbf{t}(n)}})}{\mathbf{t}(n)}\frac{\mathbf{t}(n)}{\TT_{\mathbf{t}(n)}} = \frac{\mathbb{E}\bigl[\widetilde{\mathcal{C}}_1 \bigr]}{\E[\TT_1-\TT_0]} \  \textrm{ almost surely.}
%\end{equation}
$$
\item This follows now immediately from (i) with
$$
\mathbb{E}\bigl[\widetilde{\mathcal{C}}_1 \bigr] = \mathfrak{c}\cdot \E[\TT_1-\TT_0].
$$
and by definition of $\mathcal{D}_1$: 
$$
\E[\mathcal{D}_1] = \mathbb{E}\bigl[\widetilde{\mathcal{C}}_1 \bigr]  - \mathfrak{c}\cdot \E[\TT_1-\TT_0]=0.
$$
\item 
It suffices to show that $\mathcal{D}_1=\widetilde{\mathcal{C}}_1- \mathfrak{c}\cdot (\TT_1-\TT_0)$ is not almost surely constant by constructing two paths of different length but which visit the same vertices between time $\TT_0$ and $\TT_1$.
%Let be  $\bar g\in V_1\cup V_2$ the last letter of $x_0$. Then, by the initial assumption in the statement of the lemma, $p^{(\ell)}(\bar g,\bar g)>0$ for some $\ell\in \N$. 
Assume now for a moment that $\bar g_0\in V_2^\times$ and take now any path inside $C(\mathfrak{g})$ from $\mathfrak{g}$ to $\mathfrak{g}\bar g_0 \mathfrak{g}$ which visits $\mathfrak{g}\bar g_0$ twice, say 
$$ 
\Pi_1 := (\mathfrak{g},g_1,\dots,g_{j-1},\mathfrak{g}\bar g_0,g_{j+1},\dots,g_{j+k-1},\mathfrak{g}\bar g_0,g_{j+k+1},\dots,g_{j+k+l-1},\mathfrak{g}\bar g_0 \mathfrak{g}).
$$
Consider also a second path, where we add another loop at $\mathfrak{g}\bar g_0$ as follows:
\begin{eqnarray*}
\Pi_2&:=& (\mathfrak{g},g_1,\dots,g_{j-1},\mathfrak{g}\bar g_0,g_{j+1},\dots,g_{j+k-1},\mathfrak{g}\bar g_0,\\
&& \quad g_{j+1},\dots,g_{j+k-1},\mathfrak{g}\bar g_0 ,g_{j+k+1},\dots,g_{j+k+l-1},\mathfrak{g}\bar g_0 \mathfrak{g}).
\end{eqnarray*}
Both paths visit the same elements of $V$, but have different lengths. Then a trajectory in $V^{\N_0}$ starting at $o$, which goes on a shortest path to $\mathfrak{g}$, followed by $\Pi_1$ and stays afterwards inside $C(\mathfrak{g}\bar g_0 \mathfrak{g})$ has strictly positive probability to occur and it leads to $\TT_1-\TT_0=k+l$ and to some value of $\widetilde{\mathcal{C}}_1=c_1\in\mathbb{R}$. Analogously, a trajectory in $V^{\N_0}$ starting at $o$, which goes on a shortest path to $\mathfrak{g}$, followed by $\Pi_2$ and stays afterwards inside $C(\mathfrak{g}\bar g_0 \mathfrak{g})$ has also strictly positive probability to occur and it leads to $\TT_1-\TT_0=2k+l$, but to the same value of $\widetilde{\mathcal{C}}_1=c_1$. Hence, both paths lead to different realizations of $\mathcal{D}_1$ due to $\mathfrak{c}>0$.
That is, there are $d_1,d_2\in\mathbb{R}$, $d_1\neq d_2$, such that $\P[\mathcal{D}_1=d_1],\P[\mathcal{D}_1=d_2]>0$, providing $\mathrm{Var}(\mathcal{D}_1)>0$.
\par
The cases $\bar g_0\in  V_1^\times$ and $\bar g_0=o$ can be handled completely analogously, and the case $\bar g_0 \in V\setminus (V_1^\ast \cup V_2^\ast)$ can be easily traced back to the case $\bar g_0\in V_1\cup V_2$.
%
%
%
%------
%
%The same reasoning as in \cite[Corollary 4.6]{gilch:22} gives rise to two different finite paths in $V$ having both strictly positive probability to be realized such that both paths lead to the same value of $\widetilde{\mathcal{C}}_1$, but  different values of  $\TT_1 - \TT_0$, that is, the corresponding realizations of $\mathcal{D}_1$ differ. Thus, $\mathcal{D}_1$ is \textit{not} almost surely constant, providing $\mathrm{Var}(\mathcal{D}_1)>0$. Note that we have used existence of $x_0\in V$ and $\kappa\in\N$ with $p^{(\kappa)}(x_0,x_0)>0$ at this point.
\end{enumerate}
\end{proof}

For $k\in\N$, set 
$$
\mathfrak{S}_k :=  \sum_{i=1}^k \mathcal{D}_i \quad \textrm{ and }\quad 
\widetilde{\mathfrak{S}}_k  :=  \sum_{i=1}^k \widetilde{\mathcal{C}}_i.
$$

\begin{proposition}\label{prop:4.7-neu} We have:
$$
\frac{\mathrm{Cap}(\mathbf{R}_n)-\widetilde{\mathfrak{S}}_{\mathbf{t}(n)}}{\sqrt{n}} \xrightarrow{\Prob} 0.
$$
\end{proposition}

\begin{proof}
Let  $\omega\in\Omega_0$. Observe that $\mathbf{R}_{\e{j}}(\omega)\subseteq \mathbf{R}_{\e{\mathbf{t}(n)}}(\omega)\subseteq \mathbf{R}_{n}(\omega)$ for $j\in\{1,\dots,\mathbf{t}(n)\}$. If $x\in \mathcal{R}_{k}^{(I)}(\omega)$, $k\in\{1,\dots,\e{\mathbf{t}(n)-1}\}$, then every path from $x$ to $C\bigl(X_{\e{k+1}}(\omega)\bigr)$ (or vice versa) has to pass through $X_{\e{k+1}}(\omega)$; therefore, for $k\in\{1,\dots,\e{\mathbf{t}(n)-1}\}$,
\begin{eqnarray*}
\Prob_x\bigl[S_{\mathcal{R}_{k}(\omega)}=\infty\bigr] &=& \Prob_x\bigl[S_{\mathbf{R}_{n}(\omega)}=\infty\bigr],\\
\Prob_{X_{\e{k}}(\omega)}\left[\begin{array}{c} S_{\mathcal{R}_{k-1}(\omega)}=\infty,\\ \forall n\geq 1: X_n\notin C\bigl(X_{\e{k}}(\omega)\bigr)\end{array}\right] &=& \Prob_{X_{\e{k}}(\omega)}\left[\begin{array}{c}S_{\mathbf{R}_{n}(\omega)}=\infty,\\ \forall n\geq 1: X_n\notin C\bigl(X_{\e{k}}(\omega)\bigr)\end{array}\right],\\
\Prob_{X_{\e{k}}(\omega)}\left[\begin{array}{c} S_{\mathcal{R}_{k}(\omega)}=\infty,\\ \forall n\geq 1: X_n\in C\bigl(X_{\e{k}}(\omega)\bigr)\end{array}\right] &=& \Prob_{X_{\e{k}}(\omega)}\left[\begin{array}{c}S_{\mathbf{R}_{n}(\omega)}=\infty,\\ \forall n\geq 1: X_n\in C\bigl(X_{\e{k}}(\omega)\bigr)\end{array}\right].
\end{eqnarray*}

%$$
%\mathrm{Cap}(\mathbf{R}_n)-\widetilde{\mathfrak{S}}_{\mathbf{t}(n)} \leq \TT_{\mathbf{t}(n)+1}-  \TT_{\mathbf{t}(n)}+ \TT_0.
%$$
Comparing the summands in $\mathrm{Cap}\bigl(\mathbf{R}_n(\omega)\bigr)$ and $\widetilde{\mathfrak{S}}_{\mathbf{t}(n)}(\omega)$ and using the above 
 equations we get for $n$ large enough:
\begin{eqnarray*}
 \mathrm{Cap}\bigl(\mathbf{R}_n(\omega)\bigr)-\widetilde{\mathfrak{S}}_{\mathbf{t}(n)}(\omega) 
&=& \sum_{x\in \mathbf{R}_n(\omega)\cap C\bigl(X_{\e{\mathbf{t}(n)}}(\omega)\bigr)\setminus\{ X_{\e{\mathbf{t}(n)}}(\omega)\}} \Prob_x\bigl[ S_{\mathbf{R}_n(\omega)}=\infty\bigr]\\
&&\quad + \Prob_{X_{\e{\mathbf{t}(n)}}(\omega)}\Bigl[S_{\mathbf{R}_n(\omega)}=\infty, \forall n\geq 1: X_n\in C\bigl(X_{\e{\mathbf{t}(n)}}(\omega)\bigr)\Bigr]\\
&&\quad + \sum_{x\in  R_{\TT_0}(\omega): x\notin C\bigl(X_{\TT_0}(\omega)\bigr)} \Prob_x\bigl[ S_{\mathbf{R}_{\TT_0}(\omega)}=\infty\bigr] \\
&&\quad +  \Prob_{X_{\TT_0}(\omega)}\Bigl[S_{\mathbf{R}_{\TT_0}(\omega)}=\infty, \forall n\geq 1: X_n\notin C\bigl(X_{\TT_0}(\omega)\bigr)\Bigr] \\
&\leq & \TT_{\mathbf{t}(n)+1}(\omega) - \TT_{\mathbf{t}(n)}(\omega) + \TT_0(\omega) + 2.
\end{eqnarray*}
In particular, the first equation shows that $ \mathrm{Cap}\bigl(\mathbf{R}_n(\omega)\bigr)-\widetilde{\mathfrak{S}}_{\mathbf{t}(n)}(\omega) \geq 0$. Recall from Proposition \ref{prop:TT-exp-moments} that  $\TT_i-\TT_{i-1}$, $i\in\N$, are i.i.d. and that $\TT_0$ and $\TT_i-\TT_{i-1}$ have exponential moments.
For any $\varepsilon >0$ and $n\in\mathbb{N}$, we obtain then:
%Now let be $\varepsilon >0$. By the above upper bound and by  \cite[Lemma 4.7]{gilch:22}, we get 
\begin{eqnarray*}
&&\Prob\Bigl[ \mathrm{Cap}(\mathbf{R}_n)-\widetilde{\mathfrak{S}}_{\mathbf{t}(n)} > \varepsilon \sqrt{n}, \mathbf{t}(n)\geq 1\Bigr]\\
&\leq &  \Prob\Bigl[ \TT_{\mathbf{t}(n)+1}-  \TT_{\mathbf{t}(n)}+ \TT_0 + 2>\varepsilon \sqrt{n}, \mathbf{t}(n)\geq 1 \Bigr] \\
&\leq & \Prob\Bigl[ \TT_{\mathbf{t}(n)+1}-  \TT_{\mathbf{t}(n)}+ \TT_0 >\frac{\varepsilon}{2} \sqrt{n}, \mathbf{t}(n)\geq 1 \Bigr] +  \Prob\Bigl[  2>\frac{\varepsilon}{2} \sqrt{n}, \mathbf{t}(n)\geq 1 \Bigr] \\
&\leq & \Prob\Bigl[ \exists k\in\{1,\dots,n\}: \TT_{k+1}-  \TT_{k}+ \TT_0 >\frac{\varepsilon}{2} \sqrt{n} \Bigr] +  \Prob\Bigl[  4>\varepsilon \sqrt{n}, \mathbf{t}(n)\geq 1 \Bigr] \\
&\leq & \Prob\Bigl[ \exists k\in\{1,\dots,n\}: \TT_{k+1}-  \TT_{k}>\frac{\varepsilon}{4} \sqrt{n} \Bigr] + \Prob\Bigl[ \TT_0 >\frac{\varepsilon}{4} \sqrt{n}\Bigr]+  \Prob\Bigl[ \substack{ 4>\varepsilon \sqrt{n}, \\ \mathbf{t}(n)\geq 1 }\Bigr] \\
&\stackrel{\textrm{Prop.\ref{prop:TT-exp-moments}}}{\leq} & n\cdot \Prob\Bigl[ \TT_{1}-  \TT_{0}>\frac{\varepsilon}{4} \sqrt{n} \Bigr] + \Prob\Bigl[ \TT_0 >\frac{\varepsilon}{4} \sqrt{n}\Bigr]+  \Prob\Bigl[  4>\varepsilon \sqrt{n}, \mathbf{t}(n)\geq 1 \Bigr] \\
&\leq & n\cdot \Prob\Bigl[ (\TT_{1}-  \TT_{0})^4>\frac{\varepsilon^4}{4^4} n^2 \Bigr] + \Prob\Bigl[ \TT_0 >\frac{\varepsilon}{4} \sqrt{n}\Bigr]+  \Prob\Bigl[  4>\varepsilon \sqrt{n}, \mathbf{t}(n)\geq 1 \Bigr] \\
&\leq & n\cdot \frac{\mathbb{E}\bigl[ (\TT_{1}-  \TT_{0})^4\bigr]}{\frac{\varepsilon^4}{4^4} n^2 } + \frac{\mathbb{E}[\TT_0]}{\frac{\varepsilon}{4} \sqrt{n}} + \Prob\Bigl[  4>\varepsilon \sqrt{n}, \mathbf{t}(n)\geq 1 \Bigr] 
\xrightarrow{n\to\infty} 0.
\end{eqnarray*}
The last inequality applied Markov's Inequality twice. Since $\mathbf{t}(n)\to \infty$  almost surely as $n\to\infty$, we have proven the claim.
%On the other hand side, we get with the help of the Markov Inequality:
%\begin{eqnarray*}
%&&\Prob\Bigl[ \widetilde{\mathfrak{S}}_{\mathbf{t}(n)} - \mathrm{Cap}(\mathbf{R}_n) > \varepsilon \sqrt{n}, \mathbf{t}(n)\geq 1\Bigr]\\
%&\leq & \Prob\bigl[ \TT_0 >  \varepsilon \sqrt{n}, \mathbf{t}(n)\geq 1\bigr]\\
%&\leq & \Prob\bigl[ \TT_0 >  \varepsilon \sqrt{n}\bigr]\\
%&\stackrel{\textrm{M.I.}}{\leq} & \frac{\mathbb{E}[\TT_0]}{ \varepsilon \sqrt{n}} \xrightarrow{n\to\infty} 0.
%\end{eqnarray*}
\end{proof}

\begin{proof}[Proof of Theorem \ref{th:clt}]
By Billingsley \cite[Theorem 14.4]{billingsley:99}, we get the following convergence in distribution:
$$
\frac{\mathfrak{S}_{\mathbf{t}(n)}}{\sqrt{\mathrm{Var}(\mathcal{D}_1)} \cdot  \sqrt{\mathbf{t}(n)}} \xrightarrow{\textrm{d}} N(0,1).
$$
In \cite[p. 398]{gilch:22} it is shown that
$$
\frac{n}{\mathbf{t}(n)} \xrightarrow{n\to\infty} \E[\TT_1-\TT_0] \quad \textrm{almost surely}.
$$
An application of the Lemma of Slutsky yields
\begin{equation}\label{equ:N(0,1)}
\frac{\mathfrak{S}_{\mathbf{t}(n)}}{\sqrt{n}} = \frac{\mathfrak{S}_{\mathbf{t}(n)}}{\sqrt{\mathrm{Var}(\mathcal{D}_1)} \sqrt{\mathbf{t}(n)}}\frac{\sqrt{\mathbf{t}(n)}}{\sqrt{n}}\sqrt{\mathrm{Var}(\mathcal{D}_1)} \xrightarrow{\textrm{d}} N(0,\sigma^2),
\end{equation}
where 
\begin{equation}\label{equ:sigma}
\sigma^2 := \frac{\mathrm{Var}(\mathcal{D}_1) }{\E[\TT_1-\TT_0] } = \frac{\E\bigl[\bigl( \widetilde{\mathcal{C}}_1-\mathfrak{c}\cdot (\TT_1-\TT_0)\bigr)^2\bigr] }{\E[\TT_1-\TT_0] }>0.
\end{equation}
The next goal is to show that
$$
\frac{\bigl(\mathrm{Cap}(\mathbf{R}_n)-n\cdot \mathfrak{c}\bigr) -\mathfrak{S}_{\mathbf{t}(n)}}{\sqrt{n}} \xrightarrow{\Prob} 0.
$$
In order to prove this convergence we note that
$$
\mathfrak{S}_{\mathbf{t}(n)} = \widetilde{\mathfrak{S}}_{\mathbf{t}(n)} - \bigl(\TT_{\mathbf{t}(n)}-\TT_0\bigr)\cdot \mathfrak{c}.
$$
This equation together with  $\mathrm{Cap}(\mathbf{R}_n) - \widetilde{\mathfrak{S}}_{\mathbf{t}(n)}\geq 0$ almost surely (see proof of \mbox{Proposition \ref{prop:4.7-neu})} and $n\geq \TT_{\mathbf{t}(n)} -\TT_0$ yields:
\begin{eqnarray*}
&&\Prob\Bigl[ \Bigl|\bigl(\mathrm{Cap}(\mathbf{R}_n)-n\cdot \mathfrak{c}\bigr) - \mathfrak{S}_{\mathbf{t}(n)} \Bigr|>\varepsilon   \cdot \sqrt{n}, \mathbf{t}(n)\geq 1\Bigr]\\
&\leq & \Prob\biggl[\begin{array}{c} \bigl| \mathrm{Cap}(\mathbf{R}_n)- \widetilde{\mathfrak{S}}_{\mathbf{t}(n)} \bigr|> \frac{\varepsilon}{2} \cdot  \sqrt{n},\\ \mathbf{t}(n)\geq 1\end{array}\biggr]
+ \Prob\biggl[\begin{array}{c} \mathfrak{c} \cdot \bigl(n - (\TT_{\mathbf{t}(n)}-\TT_0)\bigr) > \frac{\varepsilon}{2} \cdot  \sqrt{n},\\ \mathbf{t}(n)\geq 1\end{array}\biggr] \\
&\leq & \underbrace{\Prob\biggl[ \begin{array}{c}\mathrm{Cap}(\mathbf{R}_n) - \widetilde{\mathfrak{S}}_{\mathbf{t}(n)} > \frac{\varepsilon}{2}\cdot   \sqrt{n},\\ \mathbf{t}(n)\geq 1\end{array}\biggr]}_{(\ast)}
+ \underbrace{\Prob\biggl[\begin{array}{c} \mathfrak{c} \cdot \bigl(n - (\TT_{\mathbf{t}(n)}-\TT_0)\bigr) > \frac{\varepsilon}{2} \cdot  \sqrt{n},\\  \mathbf{t}(n)\geq 1\end{array}\biggr]}_{(\ast\ast)}.
\end{eqnarray*}
By Proposition \ref{prop:4.7-neu}, $(\ast)$ tends to $0$ as $n\to\infty$. Furthermore, in \cite[p. 398]{gilch:22} it is shown that $(\ast\ast)$ tends also to $0$ as $n\to\infty$. Hence, since $\mathbf{t}(n)\to\infty$ almost surely,
\begin{equation}\label{equ:convergence-in-prob-clt}
\Prob\Bigl[ \Bigl|\bigl(\mathrm{Cap}(\mathbf{R}_n)-n\cdot \mathfrak{c}\bigr) -\mathfrak{S}_{\mathbf{t}(n)} \Bigr| >\varepsilon \cdot   \sqrt{n}\Bigr] \xrightarrow{n\to\infty} 0.
\end{equation}

Another application of the Lemma of Slutsky together with  (\ref{equ:N(0,1)})  and (\ref{equ:convergence-in-prob-clt}) yields the proposed central limit theorem:
$$
\frac{\mathrm{Cap}(\mathbf{R}_n)-n\cdot \mathfrak{c}}{\sigma \cdot \sqrt{n}} \xrightarrow{d} N(0,1).
$$

\end{proof}

\begin{remark}\label{rem:degen-cases}
In this section we have assumed that there are $x_0\in V$ and $\kappa\in\N$ such that $p^{(\kappa)}(x_0,x_0)>0$. We present an example in which this assumption does not hold and where we have $\mathrm{Var}(\mathcal{D}_1)=0$ such that a central limit theorem becomes redundant.
\par
Let  $V_1=\{o_1=g_0,g_1,g_2,\dots\}$, $V_2=\{o_2=h_0,h_1,h_2,\dots\}$ be infinite, but countable sets, and set the transition probabilities $p_1(g_n,g_{n+1}):=1$, $p_2(h_n,h_{n+1}):=1$ for $n\in\N_0$. Set $\alpha:=\frac12$. It is easy to check that $\mathbf{R}_n(\omega)=n+1$ and $\P_x[S_{\mathbf{R}_n(\omega)}=\infty]=\frac12$ for all $\omega\in\Omega_0$ and $x\in\mathbf{R}_n(\omega)\setminus\{X_n(\omega)\}$. This implies $\mathfrak{c}=\frac12$. 
\par
Furthermore, set $\mathfrak{g}:=g_1$. Then:
\begin{eqnarray*}
\widetilde{\mathcal{C}}_1(\omega) &= &  \bigl(\TT_1(\omega)-\TT_0(\omega)-1\bigr)\cdot \frac12 \\
&&\quad + 
\P_{X_{\TT_0(\omega)}}\left[ \substack{
S_{\mathcal{R}_{\tau_0}(\omega)}=\infty, \\
\forall n\geq 1: X_n\in C\bigl(X_{\TT_0(\omega)}\bigr)
}\right]
+
\P_{X_{\TT_1(\omega)}}\left[ \substack{
S_{\mathcal{R}_{\tau_1}(\omega)}=\infty, \\
\forall n\geq 1: X_n\notin C\bigl(X_{\TT_1(\omega)}\bigr)
}\right] \\
&=& \bigl(\TT_1(\omega)-\TT_0(\omega)-1\bigr)\cdot \frac12 
+ \frac12+0= \bigl(\TT_1(\omega)-\TT_0(\omega)\bigr)\cdot \frac12 \\
&=& \bigl(\TT_1(\omega)-\TT_0(\omega)\bigr)\cdot \mathfrak{c},
\end{eqnarray*}
that is, $\mathcal{D}_1=0$ almost surely, implying $\mathrm{Var}(\mathcal{D}_1)=0$. A central limit theorem is redundant in this example.
\end{remark}

\section{Analyticity of the Asymptotic Capacity}
\label{sec:analyticity}

In this section we prove that $\mathfrak{c}$ varies real-analytically if the transition matrix of the underlying  random walk depends  on finitely many parameters only. For this purpose, we 
fix graphs $\mathcal{X}_1,\mathcal{X}_2$ (arising from any given transition matrices $P_1,P_2$) from Section \ref{sec:rw-on-fp} and set \mbox{$E_i:=\{(x,y)\in V_i^2\mid p_i(x,y)>0\}$} for $i\in\calI$, the set of oriented edges of $\mathcal{X}_i$. We assume from now on that  each non-negative single-step transition probability of the random walk on $V$ takes one out of finitely many parameters $p_1,\dots,p_d$, $d\in\mathbb{N}$, which take values in $(0,1)$. That is, if $p(x,y)>0$ for $x,y\in V$ then $p(x,y)=p_j$ for some $j\in\{1,\dots,d\}$. 
\par 
The idea is now to vary the values of the parameters slightly such that we still obtain a well-defined random walk on $V$, and to study the behaviour of $\mathfrak{c}$ as a function in $(p_1,\dots,p_d)$. For this purpose, let $\eta: E_1\cup E_2 \to \{p_1,\dots,p_d\}$ be a mapping. Then a parameter vector $\underline{p}:=(p_1,\dots,p_d)\in (0,1)^d$ gives rise to a \textit{well-defined random walk} on $V$ if
\begin{eqnarray*}
\forall x\in V_1: \ \sum_{y\in V_1: (x,y)\in E_1} \eta(x,y) + \sum_{z\in V_2: (o_2,z)\in E_2} \eta(o_2,z)&=&1 \quad \textrm{ and}\\
\forall x\in V_2: \ \sum_{y\in V_2: (x,y)\in E_2} \eta(x,y) + \sum_{z\in V_1: (o_1,z)\in E_1} \eta(o_1,z)&=&1.
\end{eqnarray*}
In other words, the probabilities of the outgoing edges at each $x\in V_1\cup V_2$ must sum up to $1$. Then $\alpha:=\sum_{y\in V_1: (o_1,y)\in E_1}\eta(o_1,y)$, and for $x_1,y_1\in V_1$, $x_2,y_2\in V_2$
$$
p_1(x_1,y_1):=\begin{cases}
\frac{\eta(x_1,y_1)}{\alpha}, & \textrm{if } (x_1,y_1)\in E_1,\\
0, & \textrm{otherwise,}
\end{cases}
\ \textrm{ and } \
p_2(x_2,y_2):=\begin{cases}
\frac{\eta(x_2,y_2)}{1-\alpha}, & \textrm{if } (x_2,y_2)\in E_2,\\
0, & \textrm{otherwise}.
\end{cases}
$$
Denote by
$$
\mathcal{P}:=\bigl\lbrace \underline{p}=(p_1,\dots,p_d)\in (0,1)^d\mid \underline{p} \textrm{ defines a well-defined random walk on $V$}\bigr\rbrace
$$
the set of parameter vectors which allow well-defined random walks of constant support induced by $E_1\cup E_2$.
Observe that each random walk defined by any $\underline{p}\in\mathcal{P}$  leads to the same transition graph $\mathcal{X}$. 
Moreover, we may then regard $\mathfrak{c}$ as a mapping
$$
\mathfrak{c}: \mathcal{P} \to [0,1], (p_1,\dots,p_d)\mapsto \mathfrak{c}=\mathfrak{c}(p_1,\dots,p_d).
$$
We want to show that $\mathfrak{c}(p_1,\dots,p_d)$ varies real-analytically in $(p_1,\dots,p_d)$, that is, one can extend $\mathfrak{c}(p_1,\dots,p_d)$ as a multivariate power series in $(p_1,\dots,p_d)$ in a neighbourhood of any $\underline{p}_0\in\mathcal{P}$. To this end, we use the formula from Lemma \ref{lem:VarD>0}.(i) given by
\begin{equation}\label{equ:c-formula}
\mathfrak{c}=\lim_{n\to\infty} \frac{\mathrm{Cap}(\mathbf{R}_n)}{n}=\frac{\mathbb{E}[\widetilde{\mathcal{C}}_1]}{\mathbb{E}[\TT_1-\TT_0]}\quad \textrm{almost surely.}
\end{equation}
Hence, it suffices to  show that both numerator and denominator vary real-analytically in $(p_1,\dots,p_d)$. We already have the following result:
\begin{lemma}\label{lem:T1-T0-analytic}
The mapping
$$
\mathcal{P}\ni (p_1,\dots,p_d)\mapsto \mathbb{E}[\TT_1-\TT_0]
$$
varies real-analytically in $(p_1,\dots,p_d)$.
\end{lemma}
\begin{proof}
See \cite[Lemma 5.1]{gilch:22}. %$\mathbb{E}[\TT_1-\TT_0]$ varies  real-anayltically in $(p_1,\dots,p_d)$. 
\end{proof}
Thus, it remains to show that also $\mathbb{E}[\widetilde{\mathcal{C}}_1]$ varies real-analytically.

\begin{remark}\label{rem:procedure-analyticity}
%The idea is to rewrite $\mathbb{E}[\widetilde{\mathcal{C}}_1]$ as a power series of the form $T(z)=\sum_{n\geq 0} a_n z^n$ in a complex variable $z$ and with non-negative coefficients, where $a_n$ is a sum of monomials of the form $a(n_1,\dots,n_d) \cdot p_1^{n_1}\cdot \ldots \cdot p_d^{n_d}$ with $n_1,\dots,n_d,a(n_1,\dots,n_d)\in\mathbb{N}_0$, $n_1+\ldots + n_d=n$ such that $\mathbb{E}[\widetilde{\mathcal{C}}_1]=T(1)$.
The idea is to construct a  power series in the form $T(z)=\sum_{n\geq 0} a_n z^n$  such that $\mathbb{E}[\widetilde{\mathcal{C}}_1]=T(1)$, and that $a_n$ is a sum of monomials of the form $a(n_1,\dots,n_d) \cdot p_1^{n_1}\cdot \ldots \cdot p_d^{n_d}$ with $n_1,\dots,n_d,a(n_1,\dots,n_d)\in\mathbb{N}_0$, $n_1+\ldots + n_d=n$.
If $T(z)$ has radius of convergence bigger than $1$, then there exists $\delta>0$ small enough such that
$$
\infty > T(1+\delta)= \sum_{n\geq 0} a_n \cdot (1+\delta)^n = \sum_{n\geq 0} \sum_{\substack{n_1,\dots,n_d\in\N_0: \\ n_1+\ldots + n_d=n}} a(n_1,\dots,n_d) \bigl(p_1(1+\delta)\bigr)^{n_1}\cdot \ldots \cdot \bigl(p_d(1+\delta)\bigr)^{n_d}.
$$
From this follows  analyticity  of  $\mathbb{E}[\widetilde{\mathcal{C}}_1]$  in a neighbourhood of any \mbox{$(p_1,\dots,p_d)\in\mathcal{P}$,} provided existence of $T(z)$. 
\end{remark}

Before we can show that such a power series $T(z)$ as in Remark \ref{rem:procedure-analyticity} exists, we have to introduce further generating functions and we will prove essential properties of them in the following subsection.
%
%write the occuring probabilities in $\mathbb{E}[\widetilde{\mathcal{C}}_1]$ as power series in one complex variable $z$, to show that these power series have radii of convergence strictly bigger than $1$ and to derive the real-analytic behaviour of $\mathbb{E}[\widetilde{\mathcal{C}}_1]$ in $(p_1,\dots,p_d)$. 

%\begin{remark}
%For random walks on groups beyond free products of groups, it is unknown whether the asymptotic capacity of the range varies real-analytically. In particular, this question is open for the well-studied case of $\mathbb{Z}^d$. This demonstrates the power of the innovative statement in Theorem \ref{th:analyticity}.
%\end{remark}

\subsection{Uniform Bounds of some Generating Functions}
\label{subsec:uniform-bounds-generating-functions}

We introduce further generating functions, summarize some essential properties and prove some important uniform bounds in this subsection. For finite $R\subset V$, $x,y\in V$ and $z\in\mathbb{C}$, set 
\begin{eqnarray*}
G(x,y|z)&:=& \sum_{n\geq 0}\P_{x}[X_n=y]z^n \quad \textrm{(Green function)},\\
U(x,y|z)&:=& \sum_{n\geq 1}\P_{x}[S_{y}=n]z^n,\\
U(x,R|z) &:=& \sum_{n\geq 1}\P_x[S_R=n]z^n,\\
\overline{U}(x,R|z)&:= & 1-\sum_{n\geq 1}\P_x[S_R=n]z^n =1-U(x,R|z).
\end{eqnarray*}
In particular, we have $\overline{U}(x,R|1)=\P_x[S_R=\infty]$. Moreover, define the \textit{first visit generating function} by
$$
F(x,y|z)=\sum_{n\geq 0} \mathbb{P}_x[T_y=n]\,z^n, \textrm{ where } T_y:=\inf\{m\in\mathbb{N}_0\mid X_m=y\},
$$
and the \textit{last visit generating function} is defined by
$$
L(x,y|z)=\sum_{n\geq 0} \mathbb{P}_x\bigl[X_n=y,\forall k<n: X_k\neq x\bigr]\,z^n.
$$
If $x\neq y$, then $F(x,y|z)=U(x,y|z)$. Analogously, we denote by $F_i$ and $L_i$, $i\in\calI$, the associated first visit and last visit generating functions of the random walks on $V_i$ governed by $P_i$.
If all paths from $x\in V$ to $y\in V$ have to pass through $w\in V$, then
\begin{equation}\label{equ:F-L-decomposition}
F(x,y|z)=F(x,w|z)\cdot F(w,y|z) \quad \textrm{ and }\quad L(x,y|z)=L(x,w|z)\cdot L(w,y|z);
\end{equation}
these equations are  obtained by conditioning with respect to the first/last visit of $w$, which must be visited before finally walking to $y$ (see, e.g. Woess \cite[Prop. 1.43]{woess:09}). Similarily, one can easily show that
\begin{equation}\label{equ:G-F-L-equations}
G(x,y|z)=F(x,y|z)\cdot G(y,y|z) \ \textrm{ and } \ G(x,y|z)=G(x,x|z)\cdot L(x,y|z);
\end{equation}
see, e.g. Woess \cite[Thm. 1.38]{woess:09}.
Furthermore, there are probability generating functions 
$$
\xi_i(z):=\sum_{n\geq 1} \mathbb{P}[S_{V_i^\times}=n]z^n,\ i\in\calI,
$$
such that for all $x,y\in V_i$
\begin{equation}\label{equ:F-L-projection}
F(x,y|z) = F_i\bigl(x,y| \xi_i(z)\bigr) \quad \textrm{ and }\quad L(x,y|z) = L_i\bigl(x,y| \xi_i(z)\bigr).
\end{equation}
%and $\xi(z)< R$ for all $|z|<R$; see, e.g., \cite[Proposition 9.18 (c)]{woess}. 
In particular, $\xi_i(z)$ converges for all $|z|<\mathscr{R}$ and we have $0<\xi_i(1)<1$;  see \cite[Proposition 9.18 (c)]{woess}, \cite[Proposition 2.7]{gilch} and \cite[Lemma 2.3]{gilch:07}. Obviously,  $\xi_i(1)=\xi_i$; compare with (\ref{equ:xi}). Due to the tree-like structure of $\mathcal{X}$, we have further important identities: for each $w\in V$ and $g\in V_1^\times \cup V_2^\times$ with $g\notin V_{\delta(w)}$, we have % for all $z\in\mathbb{C}$ that
\begin{equation}\label{equ:F-projection}
\begin{array}{rcl}
F(wg,w|z) &=& F(g,o|z) = F_{\delta(g)}\bigl(g,o_{\delta(g)}|\xi_{\delta(g)}(z)\bigr) \textrm{ and } \\[1ex]
L(w,wg|z) &= & L(o,g|z) = L_{\delta(g)}\bigl(o_{\delta(g)},g|\xi_{\delta(g)}(z)\bigr);
\end{array}
\end{equation}
these equations follow directly from the fact that $F(wg,w|z)$ and $L(w,wg|z)$ consider only paths inside the cone $C(w)$, the rest follows with Lemma \ref{lem:cone-probs}.
\par
In the following we will derive uniform upper bounds for some important generating functions.
\begin{lemma}\label{lem:sup-L-L}
There exists $\varrho_0>1$ such that
$$
%\sup_{x\in V_1\cup V_2} L(o,x|\varrho_0)<\infty \ \textrm{ and } 
\sup_{x\in V_1^\times, y\in V_2^\times} L(o,x|\varrho_0)L(o,y|\varrho_0)< 1.
$$
\end{lemma}
\begin{proof}
For $z\in\mathbb{C}$ and $i\in\calI$, define
%$$
%\mathcal{L}(z):= 1+ \sum_{x\in V\setminus\{o\}} L(o,x|z),
%$$
%and for $i\in\calI$
$$
\mathcal{L}^+_i(z):= \sum_{x\in V_i^\times} L(o,x|z) \stackrel{(\ref{equ:F-L-projection})}{=}  \sum_{x\in V_i^\times} L_i\bigl(o_i,x|\xi_i(z)\bigr). 
$$
Since $\xi_i(z)$ is continuous and monotonously increasing for real $z>0$, $\xi_i(1)<1$ and $\xi_i(z)$ has radius of convergence of at least $\mathscr{R}$, there exists some $\bar\varrho_0\in (1,\mathscr{R})$ such that $\xi_1(\bar\varrho_0)<1$ and $\xi_2(\bar\varrho_0)<1$, implying
$$
\mathcal{L}^+_i(\bar\varrho_0) \leq \sum_{x\in V_i^\times} \sum_{n\geq 1} p_i^{(n)}(o_i,x) \cdot \xi_i(\bar\varrho_0)^n \leq \sum_{n\geq 1}\xi_i(\bar\varrho_0)^n < \frac{1}{1-\xi_i(\bar\varrho_0)} <\infty,
$$
where $p_i^{(n)}(x,y)$, $x,y\in V_i$, denotes the $n$-step transition probabilities of the random walk on $V_i$ governed by $P_i$.
Consequently, there are at most finitely many $x\in V_i^\times$ with \mbox{$L(o,x|\bar\varrho_0)\geq 1$}, and we have
\begin{equation*}\label{equ:L(o,x)-bound}
\sup_{x\in V_1^\times\cup V_2^\times} L(o,x|\bar\varrho_0) \leq \max\Bigl\lbrace \frac{1}{1-\xi_1(\bar\varrho_0)},\frac{1}{1-\xi_2(\bar\varrho_0)}\Bigr\rbrace<\infty.
\end{equation*}
Choose now any $x\in V_1^\times$ and $y\in V_2^\times$
and set $w_k:=xy\dots xy$, where $xy$ is repeated $k\in\N$ times. By decomposition of all paths from $o$ to $w_k$ with respect to the last visit of $o$ and $w_k$ before finally staying in $C(w_k)\setminus\{w_k\}$, we obtain for all $k\in\N$:
\begin{eqnarray*}
1 &\geq & \mathbb{P}\bigl[ X_\infty \textrm{ starts with prefix } w_k\bigr] \\
&=& G(o,w_k|1) \cdot \mathbb{P}_{w_k}\bigl[\forall n\geq 1: X_n\in C(w_k)\setminus\{w_k\}\bigr] \\
&\stackrel{(\ref{equ:G-F-L-equations})}{=}& G(o,o|1)\cdot L(o,w_k|1)\cdot \mathbb{P}_{w_k}\bigl[\forall n\geq 1: X_n\in C(w_k)\setminus\{w_k\}\bigr] \\
&\stackrel{(\ref{equ:F-L-decomposition}),(\ref{equ:F-projection})}{=}& G(o,o|1) \cdot L(o,x|1)^k \cdot L(o,y|1)^k\cdot  \underbrace{\mathbb{P}_{w_k}\bigl[\forall n\geq 1: X_n\in C(w_k)\setminus\{w_k\}\bigr]}_{\geq \mathbb{P}[X_1\in V_1^\times,\forall n>1: X_n\notin V_1]} \\
&\geq & G(o,o|1) \cdot L(o,x|1)^k \cdot L(o,y|1)^k \cdot \alpha \cdot \bigl(1-\xi_1(1)\bigr)>0.
\end{eqnarray*}
In the last step we have used that $1-\xi_1(1)$ is the probability that $V_1^\times$ is not visited any more when starting at any state in $V_1^\times$.
Since $k$ can be chosen arbitrarily large, we must have $L(o,x|1)\cdot L(o,y|1)\leq 1$ for all $x\in V_1^\times$, $y\in V_2^\times$. Assume now for a moment that $L(o,x|1)\cdot L(o,y|1)= 1$ would hold for some $x\in V_1^\times$ and $y\in V_2^\times$. 
%Then:
%$$
%\mathbb{P}\bigl[ \exists k\in\N: X_\infty \textrm{ starts with prefix } w_k\bigr] \geq G(o,o|1)\cdot \alpha \cdot \bigl(1-\xi_2(1)\bigr)=:C_0>0.
%$$ 
If $|V_1|>2$, then there is some $x_0\in V_1^\times\setminus\{x\}$ such that we obtain analogously as above for every $j\in\N$:
\begin{eqnarray*}
&&\mathbb{P}\bigl[X_\infty \textrm{ starts with prefix } w_jx_0\bigr] \\
&\geq & G(o,o|1) \cdot L(o,x|1)^j \cdot L(o,y|1)^j \cdot L(o,x_0|1) \cdot (1-\alpha)\cdot \bigl(1-\xi_2(1)\bigr)\\
&= & G(o,o|1) \cdot L(o,x_0|1) \cdot (1-\alpha)\cdot \bigl(1-\xi_2(1)\bigr) =:C_{x_0}>0.
\end{eqnarray*}
If we choose $k$ large enough, we get a contradiction due to
\begin{eqnarray*}
1 &\geq & \mathbb{P}\bigl[ X_\infty \textrm{ starts with a prefix in } \{w_1x_0,\dots,w_kx_o\}\bigr] \\
&=& \sum_{j=1}^k\mathbb{P}\bigl[ X_\infty \textrm{ starts with prefix } w_jx_0\bigr]  \geq k\cdot C_{x_0} >1.%+ \mathbb{P}\bigl[ \exists k\in\N: X_\infty \textrm{ starts with prefix } w_k\bigr]\geq k\cdot C_x+C_0>1. 
\end{eqnarray*}
Therefore, we must have $L(o,x|1)\cdot L(o,y|1)< 1$ for all $x\in V_1^\times$, $y\in V_2^\times$. If $|V_1|=2$, then we must have $|V_2|>2$ and the reasoning follows analogously by exchanging the roles of $x$ and $y$.
\par
As mentioned above there are at most finitely many $x\in V_1^\times$ with $ L(o,x|\bar\varrho_0)\geq 1$. For each such $x\in V_1^\times$ there are also at most finitely many $y\in V_2^\times$ with $L(o,y|\bar\varrho_0)\geq L(o,x|\bar\varrho_0)^{-1}$. Analogously, there are finitely many $y\in V_2^\times$ with $L(o,y|\bar\varrho_0)\geq 1$ and finitely many $x\in V_1^\times$ with $L(o,x|\bar\varrho_0)\geq L(o,y|\bar\varrho_0)^{-1}$.

Since $L(o,x|1)\cdot L(o,y|1)< 1$ and by continuity of the involved generating functions, there exists $\varrho_0\in (1,\bar\varrho_0)$ such that 
$$
L(o,x|\varrho_0)\cdot L(o,y|\varrho_0)< 1
$$ 
%Since $\varrho_0<\bar\varrho_0$, we also have $\sup_{x\in V_1\cup V_2} L(o,x|\varrho)<\infty$. 
for all $x\in V_1^\times$ and $y\in V_2^\times$. This finishes the proof.
\end{proof}

From the last lemma follows immediately that there exists $\varrho_0>1$ such that
\begin{equation}\label{equ:sup-L}
\sup_{x\in V_1\cup V_2} L(o,x|\varrho_0)<\infty.
\end{equation}
We will use this fact in the next lemma.
\begin{lemma}\label{lem:sup-Lxy}
There exists $\varrho_0>1$ such that
\begin{equation}\label{equ:L(x,y)}
\sup_{x\in V, y\in C(x)} L(x,y|\varrho_0)<\infty.
\end{equation}
\end{lemma}
\begin{proof}
From Lemma \ref{lem:sup-L-L} follows existence of $\varrho_0>1$ with 
$$
\sup_{x\in V_1^\times, y\in V_2^\times} L(o,x|\varrho_0)L(o,y|\varrho_0)< 1.
$$
Now take any $x\in V$ and $y\in C(x)\setminus\{x\}$, that is, we can rewrite $y$ as $y=xy_1\dots y_k$ in the form of (\ref{equ:free-product}). Then:
\begin{eqnarray*}
&& L(x,y|\varrho_0) \stackrel{\textrm{Lemma } \ref{lem:cone-probs}}{=} L(o,y_1\dots  y_k|\varrho_0) 
\stackrel{(\ref{equ:F-L-decomposition}),(\ref{equ:F-L-projection})}{=} \prod_{j=1}^k L_{\delta(y_i)} \bigl(o_{\delta(y_i)},y_i| \xi_{\delta(y_i)}(\varrho_0)\bigr)\\
&=& \prod_{j=1}^{\lfloor k/2 \rfloor} \underbrace{\Bigl( L_{\delta(y_{2i-1})} \bigl(o_{\delta(y_{2i-1})},y_{2i-1}| \xi_{\delta(y_{2i-1})}(\varrho_0)\bigr)L_{\delta(y_{2i})} \bigl(o_{\delta(y_{2i})},y_{2i}| \xi_{\delta(y_{2i})}(\varrho_0)\bigr)\Bigr)}_{< 1 \textrm{ (by Lemma \ref{lem:sup-L-L})}}  \cdot L_0,
\end{eqnarray*}
where $L_0:=L(o,y_k|\varrho_0)=L_{\delta(y_k)} \bigl(o_{\delta(y_k)},y_k| \xi_{\delta(y_k)}(\varrho_0)\bigr)$, if $k$ is odd, and $L_0:=1$, if $k$ is even. In both cases, since $L(w,w|\varrho_0)=1$ for all $w\in V$, (\ref{equ:sup-L}) yields. 
$$
L(x,y|\varrho_0)\leq   \sup_{x\in V_1\cup V_2}L(o,x|\varrho_0)<\infty.
$$
\end{proof}
We show another uniform upper bound for some family of Green functions:
\begin{lemma}\label{lem:r1}
There exists $\varrho_1\in (1,\mathscr{R})$ such that %and a constant $c_0<\infty$ such that %, for all $w\in V$ and all $x,y\in V_i^\times$, $i\in\calI\setminus \{\delta(w)\}$,
\begin{equation}\label{equ:G(wx,wy)}
\sup \Bigl\lbrace G(wx,wy|\varrho_1) \,\Bigl|\, w\in V, x,y\in V_i^\times \textrm{ with } i\in\calI\setminus \{\delta(w)\}\Bigr\rbrace <\infty.
\end{equation}
\end{lemma}
\begin{proof}
Let  $w\in V$ and $x,y\in V_i^\times$, where $i\in\calI\setminus \{\delta(w)\}$. For $z\in\mathbb{C}$, define
\begin{eqnarray*}
\check{F}(wx,wy|z)&:=&\sum_{n\geq 0} \mathbb{P}_{wx}\bigl[ T_{wy}=n,\forall j<n: X_j\neq w\bigr]\, z^n,\\
\check{G}(wy,wy|z)&:=&\sum_{n\geq 0} \mathbb{P}_{wx}\bigl[ X_n=wy,\forall j<n: X_j\neq w\bigr]\, z^n.
\end{eqnarray*}
Then, by distinguishing whether $w$ is visited on a path from $wx$ to $wy$ or not, we obtain
\begin {eqnarray}
&&G(wx,wy|z) =F(wx,w|z)\cdot G(w,wy|z)+  \check{F}(wx,wy|z)\cdot \check{G}(wy,wy|z)\nonumber \\
&\stackrel{(\ref{equ:G-F-L-equations})}{=}& F(wx,w|z)\cdot G(w,w|z) \cdot L(w,wy|z) +  \check{F}(wx,wy|z)\cdot \check{G}(wy,wy|z)\nonumber \\
&\stackrel{(\ref{equ:F-projection})}{=}& F_{i}\bigl(x,o_i |\xi_i(z)\bigr) \cdot G(w,w| z)\cdot L_i\bigl(o_i, y| \xi_i(z)\bigr) + \check{F}(wx,wy|z)\cdot \check{G}(wy,wy|z).\label{equ:G-decomposition}
\end{eqnarray}
From  \cite[Lemma 3.6]{gilch:22} follows existence of  $\bar\varrho_1\in (1,\mathscr{R})$ such that
$$
\sup_{v\in V} G(v,v|\bar\varrho_1)<\infty.
$$
Choose $\varrho_1\in(1,\bar\varrho_1)$ such that $\xi_j(\varrho_1)<1$ for all $j\in\calI$. Then $F_{i}\bigl(x,o_i |\xi_i(\varrho_1)\bigr)\leq 1$. Furthermore,
$$
L_i\bigl(o_i, y| \xi_i(\varrho_1)\bigr) \leq \sum_{n\geq 1} \xi_i(\varrho_1)^n <\frac{1}{1-\xi_i(\varrho_1)}<\infty.
$$
Moreover, % for real $z>0$,
\begin{eqnarray*}
 \check{F}(wx,wy|\varrho_1) & =&\sum_{n\geq 0} \mathbb{P}_{wx}\bigl[ T_{wy}=n,\forall j<n: X_j\neq w\bigr]\, \varrho_1^n,\\
 &\stackrel{\textrm{Lemma }\ref{lem:cone-probs}}{=} & \sum_{n\geq 0} \mathbb{P}_x\bigl[ T_{y}=n,\forall j<n: X_j\neq o\bigr]\, \varrho_1^n,\\
 &\leq &  \sum_{n\geq 0} \mathbb{P}_x\bigl[ T_{y}=n\bigr]\, \varrho_1^n =F(x,y|\varrho_1)\\
 &=&F_i\bigl(x,y| \xi_i(\varrho_1)\bigr) \leq \frac{1}{1-\xi_i(\varrho_1)}<\infty.
 \end{eqnarray*}
% That is, 
% $$
% \check{F}(wx,wy|r_1) \leq \frac{1}{1-\xi_i(r_1)}<\infty.
%$$
Since $\check{G}(wy,wy|\varrho_1)\leq G(wy,wy|\varrho_1)$, we have shown -- in view of Equation (\ref{equ:G-decomposition}) -- that  the values $G(wx,wy|\varrho_1)$ are uniformly bounded.
%Since $\check{G}(wy,wy|z)\leq G(wy,wy|z)$ for real $z>0$, we have that 
%$$
%\sup_{x_1,\dots,x_{k+1},y_{k+1}} G(x_1\dots x_{k+1},x_1\dots x_k y_{k+1}|r_0) <\infty.
%$$
\end{proof}

\begin{proposition}\label{prop:G-convergence1}
There exists a real number $\varrho>1$ such that
$$
\sup_{x,y\in V} G(x,y|\varrho)<\infty.
$$
\end{proposition}
\begin{proof}
Let  $x,y\in V$ and write $x=x_1\dots x_m$ and $y=y_1\dots y_n$ in the form of (\ref{equ:free-product}), where $m,n\in\N_0$. Let  $k\in\N_0$ be maximal such that $w:=x_1\dots x_k=y_1\dots y_k$ (if $x_1\neq y_1$ then $k=0$), that is,  $x_1\dots x_k$ is the common prefix of $x$ and $y$ of maximal length. First, we consider the case $k<\min\{m,n\}$. Observe that each path from $x$ to $y$ has to pass through $wx_{k+1}$ and $wy_{k+1}$.Therefore, we obtain for all real $z\in (0,\mathscr{R})$:
\begin{eqnarray}
&&G(x,y|z)\nonumber\\ 
&\stackrel{(\ref{equ:G-F-L-equations})}{=}& F(x,wx_{k+1}|z)\cdot G(wx_{k+1},y|z)\nonumber \\
&\stackrel{(\ref{equ:G-F-L-equations})}{=}& F(x,wx_{k+1}|z)\cdot G(wx_{k+1},wx_{k+1}|z)\cdot L(wx_{k+1},y|z)\nonumber \\
&\stackrel{(\ref{equ:F-L-decomposition})}{=}& F(x,wx_{k+1}|z)\cdot G(wx_{k+1},wx_{k+1}|z)\cdot L(wx_{k+1},wy_{k+1}|z)\cdot  L(wy_{k+1},y|z)\nonumber\\
&\stackrel{(\ref{equ:G-F-L-equations})}{=}& F(x,wx_{k+1}|z)\cdot G(wx_{k+1},wy_{k+1}|z) \cdot L(wy_{k+1},y|z) \nonumber\\
&\stackrel{(\ref{equ:F-L-decomposition})}{=}& \prod_{i=k+2}^{m} F(x_1\dots x_i,x_1\dots x_{i-1}|z) \cdot G(wx_{k+1},wy_{k+1}|z) \cdot   L(wy_{k+1},y|z) \nonumber\\
&\stackrel{(\ref{equ:F-projection})}{=}& \prod_{i=k+2}^{m} F_{\delta(x_i)}\bigl(x_i,o_{\delta(x_i)}|\xi_{\delta(x_i)}(z)\bigr) \cdot G(wx_{k+1},wy_{k+1}|z) \cdot   L(wy_{k+1},y|z).\label{equ:G-decomposition-in-Prop}
%
%
%&=& \prod_{i=k+2}^{m} F(x_1\dots x_i,x_1\dots x_{i-1}|z) \cdot G(wx_{k+1},wy_{k+1}|z) \cdot  \prod_{j=k+1}^{m-1} L(y_1\dots y_j,y_1\dots y_{j+1}|z) \\
%&=& \prod_{i=k+2}^{m} F_{\delta(x_i)}\bigl(x_i,o_{\delta(x_i)}|\xi_{\delta(x_i)}(z)\bigr) \cdot G(wx_{k+1},wy_{k+1}|z) \cdot  \prod_{j=k+2}^{m} L_{\delta(y_{j})}\bigl(o_{\delta(y_{j})},y_{j}|\xi_{\delta(y_{j})}(z)\bigr) 
%
%
%&=& \prod_{i=k+1}^{m} F_{\tau(x_{i-1})}\bigl(x_{i},o_{\tau(x_{i-1})}|\xi_{\tau(x_{i-1})}(z)\bigr) \cdot G(x_1\dots x_{k+1},x_1\dots x_ky_{k+1}|z) \\
%&&\quad \cdot  \prod_{j=k}^{m-1} L_{\tau(y_j)}\bigl(o_{\tau(y_j)},y_i|\xi_{\tau(y_j)}(z)\bigr)
%
%\prod_{i=k+1}^{m} F(x_1\dots,x_{i},x_1\dots,x_{i-1}|z) \cdot \prod_{i=k}^{m-1} F(y_1\dots,y_{i},y_1\dots,y_{i+1}|z)\\
%&=&\prod_{i=k+1}^{m} F_{\tau(x_{i})}\bigl(x_1\dots,x_{i},x_1\dots,x_{i-1}|\xi_{\tau(x_i)}}(z)\bigr) \cdot \prod_{i=k}^{m-1} F_{\tau(y_{i+1})}\bigl(y_1\dots,y_{i},y_1\dots,y_{i+1}|\xi_{\tau(y_{i+1})}(z)\bigr)
\end{eqnarray}
Let $\varrho_0\in(1,\mathscr{R})$ be a real number satisfying (\ref{equ:L(x,y)}) and let  $\varrho_1\in(1,\mathscr{R})$ satisfying (\ref{equ:G(wx,wy)}). %Lemma \ref{lem:sup-Lxy}, Lemma \ref{lem:r1} respectively. 
Furthermore, choose any $\bar\varrho_1\in(1,\mathscr{R})$ such that $\sup_{v\in V}G(v,v|\bar\varrho_1)<\infty$, which exists due to  \cite[Lemma 3.6]{gilch:22}.
Now take any $\varrho\in \bigl(1,\min\{\varrho_0,\varrho_1,\bar\varrho_1\}\bigr)$ such that $\xi_i(\varrho)<1$ for all $i\in\calI$. Then:
%We have $\xi_i(r_1)<1$, and therefore 
$$
\prod_{i=k+2}^{m} F_{\delta(x_i)}\bigl(x_i,o_{\delta(x_i)}|\xi_{\delta(x_i)}(\varrho)\bigr)<1.
$$
By Lemma \ref{lem:r1}, $G(wx_{k+1},wy_{k+1}|\varrho)$ is uniformly bounded. Moreover, since $y\in C(wy_{k+1})$, Lemma \ref{lem:sup-Lxy} guarantees that $L(wy_{k+1},y|\varrho)$ is also uniformly bounded.
In view of (\ref{equ:G-decomposition-in-Prop}) we have proven the claim in the case $k<\min\{m,n\}$.
\par
If $k=m<n$, we have $x=w$, $y\in C(x)$ and
$$
G(x,y|\varrho) \stackrel{(\ref{equ:G-F-L-equations})}{=} G(x,x|\varrho) \cdot L(x,y|\varrho).
$$
The claim follows now directly from $\varrho<\min\{\varrho_0,\bar\varrho_1\}$ and Lemma \ref{lem:sup-Lxy}.
\par
Similarily, in the case $k=n<m$, we have $y=w$ and
$$
G(x,y|z) = \prod_{i=k+1}^{m} F_{\delta(x_i)}\bigl(x_i,o_{\delta(x_i)}|\xi_{\delta(x_i)}(z)\bigr) \cdot G(y,y|z).
$$
The claim follows now directly from $\varrho <\bar\varrho_1$ and  $\xi_i(\varrho)<1$ for all $i\in\calI$.
\par
Finally, if $k=m=n$, then $x=y$ and the claim follows immediately from the choice of $\varrho<\bar\varrho_1$. This finishes the proof.
\end{proof}

\begin{proposition}\label{prop:U-convergence2}
There exists $\varrho>1$  such that, 
%Let $r_0>1$ be such that $\sup_{x_1,x_2\in V} U(x_1,x_2|r_0)<\infty$. Then, 
for all finite $R\subset V$ and for all $x\in V$,
\begin{equation}\label{equ:U(x,R)}
U(x,R|\varrho)<\infty.
\end{equation}
In particular, there exists a finite constant $M_0>0$ such that
$$
U(x,R|\varrho)\leq |R|\cdot M_0 \quad \textrm{ for all finite $R\subset V$ and for all $x\in V$.}
$$
\end{proposition}
\begin{proof}
Let  $R\subset V$ be finite and $x\in V$. Choose $\varrho >1$ which satisfies Proposition \ref{prop:G-convergence1}.  Then:
\begin{eqnarray*}
U(x,R|\varrho) &=& \sum_{n\geq 1} \mathbb{P}_x[S_R=n]\varrho^n 
= \sum_{n\geq 1}  \mathbb{P}_x\bigl[X_1,\dots,X_{n-1}\notin R, X_n\in R\bigr]\varrho^n \\
&=& \sum_{n\geq 1}  \sum_{y\in R}\mathbb{P}_x\bigl[X_1,\dots,X_{n-1}\notin R, X_n=y\bigr]\varrho^n \\
&\leq & \sum_{n\geq 1}  \sum_{y\in R}\mathbb{P}_x\bigl[X_1,\dots,X_{n-1}\neq  y, X_n=y\bigr]\varrho^n \\
&=& \sum_{y\in R} \sum_{n\geq 1} \P_x[T_y=n]\varrho^n
= \sum_{y\in R} U(x,y|\varrho) \leq |R| \cdot \sup_{x_1,x_2\in V} U(x_1,x_2|\varrho).
\end{eqnarray*}
Obviously, $U(x_1,x_2|\varrho)\leq G(x_1,x_2|\varrho)$ for all $x_1,x_2\in V$.
From Proposition \ref{prop:G-convergence1} follows now
$$
U(x,R|\varrho) \leq |R| \cdot \sup_{x_1,x_2\in V} U(x_1,x_2|\varrho)\leq  |R| \cdot \sup_{x_1,x_2\in V} G(x_1,x_2|\varrho)<\infty.
$$
Setting $M_0:=\sup_{x_1,x_2\in V} G(x_1,x_2|\varrho)<\infty$ yields the second part of the claim.
\end{proof}
An almost immediate consequence is the following:
\begin{corollary}
$\overline{U}(x,R|z)$ has radius of convergence strictly bigger than $1$. 
\end{corollary}
\begin{proof}
Let $\varrho>1$ be a suitable real number which satisfies (\ref{equ:U(x,R)}). 
Pringsheim's Theorem yields that $U(x,R|z)$ has radius of convergence of at least $\varrho$. 
\end{proof}

%By Lemma \ref{}, $\sup_{x\in V} G(x,x|R_0)<\infty$ for some $R_0>1$ and due to $G(x,x|R_0)=\bigl(1-U(x,x|R_0)\bigr)^{-1}$, we have
%$$
%\sup_{x\in V} U(x,x|R_0)<1.
%$$
%That is, $f_{x,R}(z)$ has radius of convergence of at least $R_0>1$.

\subsection{Analyticity of $\mathbb{E}[\widetilde{\mathcal{C}}_1]$}
In this subsection we prove the real-analytic behaviour of $\mathbb{E}[\widetilde{\mathcal{C}}_1]$ as a function in $(p_1,\dots,p_d)\in\mathcal{P}$. Define 
$$
\widetilde{\mathcal{R}}_1:=\bigcup_{j=\tau_0}^{\tau_1-1} \mathcal{R}_j = \mathbf{R}_{\TT_1} \cap \Bigl( C(X_{\TT_0}) \setminus \bigl( C(X_{\TT_1})\setminus\{X_{\TT_1}\}\bigr)\Bigr), \quad \widetilde{\mathcal{R}}_1^{(I)}:=\widetilde{\mathcal{R}}_1 \setminus\{X_{\TT_0},X_{\TT_1}\}.
$$
Then for all $\omega\in\Omega_0$:
\begin{eqnarray}
\widetilde{\mathcal{C}}_1(\omega)&=&\sum_{x\in \widetilde{\mathcal{R}}_1^{(I)}\setminus\{X_{\TT_0},X_{\TT_1}\}} \P_x[S_{\widetilde{\mathcal{R}}_1(\omega)}=\infty] 
 + \P_{X_{\TT_0}(\omega)}\biggl[\substack{S_{\widetilde{\mathcal{R}}_1(\omega)}=\infty, \\ \forall n\geq 1: X_n\in C\bigl(X_{\TT_0}(\omega)\bigr)}\biggr] \nonumber\\
 &&+  \P_{X_{\TT_1}(\omega)}\biggl[\substack{S_{\widetilde{\mathcal{R}}_1(\omega)}=\infty,\\ \forall n\geq 1: X_n\notin C\bigl(X_{\TT_1}(\omega)\bigr)}\biggr];\label{equ:C1}
\end{eqnarray}
compare with the calculations in Section \ref{sec:existence}.
The next step is to construct a ``normalized version'' of $\widetilde{\mathcal{R}}_1$ and $X_{\TT_1}$ by removing the common prefix $X_{\TT_0}$: set
$$
\widetilde{\mathcal{R}}_{\textrm{norm}}:= \big\lbrace w \in V \,\bigl|\, X_{\TT_0}w\in \widetilde{\mathcal{R}}_1\bigr\rbrace.
$$
Moreover, define
$$
%V_{\mathfrak{g}}:=\bigl\lbrace v_1\dots v_k\in V \mid k\in\N, v_k=\mathfrak{g},v_1,\dots,v_{k-1}\neq \mathfrak{g}\bigr\rbrace,
V_{\mathfrak{g}}:=\bigl\lbrace v_1\dots v_k\in V \mid k\in\N, v_k=\mathfrak{g}\bigr\rbrace,
$$ 
the set of words in $V$ which end up with the letter $\mathfrak{g}$ with no further occurrence of this letter except at the end; these are the possible values of $X_{\TT_0}$. Furthermore, define
$$
%V_{2,\mathfrak{g}}:=\bigl\lbrace v_1\dots v_k\in V \mid k\in\N, v_1\in V_2^\times, v_k=\mathfrak{g},v_1,\dots,v_{k-1}\neq \mathfrak{g}\bigr\rbrace,
V_{2,\mathfrak{g}}:=\bigl\lbrace v_1\dots v_k\in V \mid k\in\N, v_1\in V_2^\times, v_k=\mathfrak{g}\bigr\rbrace,
$$ 
the set of words in $V$ which start with a letter in $V_2^\times$ and end up with the letter $\mathfrak{g}$ with no further occurrence of $\mathfrak{g}$ except at the end; then $X_{\TT_1}$ has the form $w_1w_2$, where $w_1\in V_{\mathfrak{g}}$ and $w_2\in V_{2,\mathfrak{g}}$. The increment (a ``normalized version'' of $X_{\TT_1}$) between $X_{\TT_0}=w_1$ and $X_{\TT_1}=w_1w_2$ is given by
$$
\mathbf{i}:=w_2.
$$
By definition, $\mathbf{i}$ takes values in $V_{2,\mathfrak{g}}$, while $\widetilde{\mathcal{R}}_{\textrm{norm}}$ takes almost surely values in a set
$$
\mathcal{W} := \bigl\lbrace R\subset V\,\bigl|\, \mathbb{P}[\widetilde{\mathcal{R}}_{\textrm{norm}}=R]>0\bigr\rbrace.
$$
For all $x_0\in V_{\mathfrak{g}}$ and $R\in \mathcal{W}$,  the set $x_0R:=\{x_0w| w\in R\}$ is well-defined 
since all words in $\mathcal{W}$ start with a letter in $V_2$ while $x_0$ ends with the letter $\mathfrak{g}\in V_1$.
\par
In a next step we introduce further generating functions which will play a crucial role in our proof. Recall that $\alpha=\sum_{w\in V_1} p_i(o_i,w)$.
Define for finite $R\in \mathcal{W}$, $x\in V_{2,\mathfrak{g}}$, and $z\in\mathbb{C}$
\begin{eqnarray*}
g(x,R|z)&:=& \sum_{m\geq 1} \P\bigl[\Rn =R,\mathbf{i}=x,\TT_1-\TT_0=m\bigr]z^m,\\
%U_0(R|z) &:=& \sum_{m\geq 1} \sum_{n>m} \P_{\mathfrak{g}}\bigl[ S_{\mathfrak{g}R}=m,X_n\notin C(\mathfrak{g}),\forall t< n: X_t\in C(\mathfrak{g})\bigr]\, z^n, \\
%\overline{U}_0(R|z) &:=& \bigl(1-\xi_1(z)\bigr) - \bigl(U(\mathfrak{g},\mathfrak{g}R|z)-U_0(R|z),\\
%U_1(x,R|z) &:=& \sum_{m\geq 1} \sum_{n>m} \P_{\mathfrak{g}x}\bigl[ S_{\mathfrak{g}R}=m,X_n\in C(\mathfrak{g}x),\forall t< n: X_t\notin C(\mathfrak{g}x)\bigr]\, z^n, \\
%\overline{U}_0(R|z) &:=& \bigl(1-\xi_1(z)\bigr) - \bigl(U(\mathfrak{g},\mathfrak{g}R|z)-U_0(R|z).
%%
%% OLD
U_0(\mathfrak{g},\mathfrak{g}R|z) &:=& \sum_{n\geq 1} \P_{\mathfrak{g}}\bigl[ S_{\mathfrak{g}R}=n,\forall m\leq n: X_m\in C(\mathfrak{g})\bigr]\, z^n, \\
\overline{U}_0(\mathfrak{g},\mathfrak{g}R|z) &:=& (1-\alpha)\cdot z- U_0(\mathfrak{g},\mathfrak{g}R|z),\\
U_1(\mathfrak{g}x,\mathfrak{g}R|z) &:=& \sum_{n\geq 1} \P_{\mathfrak{g}x}\bigl[ S_{\mathfrak{g}R}=n,\forall m\in\{1,\dots,n-1\}: X_m\notin C(\mathfrak{g}x)\bigr]\, z^n, \\
\overline{U}_1(\mathfrak{g}x,\mathfrak{g}R|z) &:=& \alpha\cdot z- U_1(\mathfrak{g}x,\mathfrak{g}R|z).
\end{eqnarray*}
Observe that $g(x,R|1)=\mathbb{P}[\Rn=R,\ii=x]$ and, for real $z>0$,
$$
g(x,R|z) \leq  \sum_{m\geq 1} \P\bigl[\TT_1-\TT_0=m\bigr]z^m =:\mathbb{T}(z).
$$
Note that the power series  $\mathbb{T}(z)$ has radius of convergence strictly bigger than $1$ due to existence of exponential moments of $\TT_1-\TT_0$; see Proposition \ref{prop:TT-exp-moments}. This together with Proposition \ref{prop:U-convergence2} yields that there exists some $r_1>1$ such that, for all  finite $R\in\mathcal{W}$, all $w\in V$, $x\in V_{2,\mathfrak{g}}$ and for all real  $z\in(0,r_1)$, $U(w,R|z)<\infty$ and $g(x,R|z)<\infty$.
\par
Furthermore, we have:
\begin{lemma}
For all $R\in \mathcal{W}$ and $x\in V_{2,\mathfrak{g}}$ with $\P\bigl[\Rn=R,\mathbf{i}=x\bigr]>0$,
\begin{enumerate}%[label=(\roman*)]
\item $\displaystyle \overline{U}_0(\mathfrak{g},\mathfrak{g}R|1)= \P_{\mathfrak{g}}\bigl[ S_{\mathfrak{g}R}=\infty ,\forall n\geq 1: X_n\in C(\mathfrak{g})\bigr].$
\item  $\displaystyle \overline{U}_1(\mathfrak{g}x,\mathfrak{g}R|1)= \P_{\mathfrak{g}x}\bigl[ S_{\mathfrak{g}R}=\infty ,\forall n\geq 1: X_n\notin C(\mathfrak{g}x)\bigr].$
\end{enumerate}
\end{lemma}
\begin{proof}
Let $R\in \mathcal{W}$ and $x\in V_{2,\mathfrak{g}}$ with $\P\bigl[\Rn=R,\mathbf{i}=x\bigr]>0$. Then $o,x\in R$ by definition of $\Rn$ and $\mathbf{i}$.
\begin{enumerate}%[label=(\roman*)]
\item In the following we consider only trajectories $\omega=(w_0=\mathfrak{g},w_1,\dots)\in V^{\N_0}$ starting at $\mathfrak{g}$ with $p(w_i,w_{i+1})>0$ for all $i\in\N_0$, that is, we condition on the event $[X_0=\mathfrak{g}]$.
We claim:
\begin{eqnarray*}
A_0&:=&\bigl\lbrace X_1\in C(\mathfrak{g})\bigr\rbrace \setminus \bigcup_{n\geq 1}\bigl\lbrace S_{\mathfrak{g}R}=n ,\forall m\leq n: X_m\in C(\mathfrak{g})\bigr\rbrace \\
&=& \bigl\lbrace S_{\mathfrak{g}R}=\infty ,\forall n\geq 1: X_n\in C(\mathfrak{g})\bigr\rbrace=:B_0.
\end{eqnarray*}
Indeed, if $\omega\in A_0$, then $X_1(\omega) \in C(\mathfrak{g})$ and we either must have $S_{\mathfrak{g}R}(\omega)=\infty$ with $w_i\in C(\mathfrak{g})$ for all $i\in\N_0$ (because otherwise we would have $S_{\mathfrak{g}R}(\omega)<\infty$ since $\mathfrak{g}\in \mathfrak{g}R$) or there exists some $n\in\mathbb{N}$ with $S_{\mathfrak{g}R}(\omega)=n$ and $X_{m}\notin C(\mathfrak{g})$ for some $m\leq n$; in the latter case $\mathfrak{g}\in \mathfrak{g}R$ must be visited before leaving $C(\mathfrak{g})$ implying $S_{\mathfrak{g}R}(\omega)<n$, a contradiction. Consequently, $\omega\in B_0$. Vice versa, if $\omega\in B_0$, then we have $X_1(\omega) \in C(\mathfrak{g})$ and $\omega$ is obviously contained in the set $A_0$.
\par
The above equation of sets implies:
\begin{eqnarray*}
&&\P_{\mathfrak{g}}\bigl[ S_{\mathfrak{g}R}=\infty ,\forall n\geq 1: X_n\in C(\mathfrak{g})\bigr]\\
&=& \P_{\mathfrak{g}}\bigl[ X_1\in C(\mathfrak{g})\bigr] - \sum_{n\geq 1} \P_{\mathfrak{g}}\bigl[ S_{\mathfrak{g}R}=n,\forall m\leq n: X_m\in C(\mathfrak{g})\bigr] \\
&=& (1-\alpha) - U_0(\mathfrak{g},\mathfrak{g}R|1)=\overline{U}_0(\mathfrak{g},\mathfrak{g}R|1).
%&=& \P_{\mathfrak{g}}\bigl[ \forall n\geq 1: X_n\in C(\mathfrak{g})\bigr] - \sum_{k\geq 1} \P_{\mathfrak{g}}\bigl[ S_{\mathfrak{g}R}=k,\forall n\geq 1: X_n\in C(\mathfrak{g})\bigr] \\
%&=& \bigl(1-\xi_1(1)\bigr) - \Bigl[ \sum_{k\geq 1} \P_{\mathfrak{g}}\bigl[ S_{\mathfrak{g}R}=k\bigr] - \sum_{m>k} \P_{\mathfrak{g}}\bigl[ S_{\mathfrak{g}R}=k,X_m\notin C(\mathfrak{g}) ,\forall t<n 1: X_t\in C(\mathfrak{g})\bigr] \\
%&=& \bigl(1-\xi_1(1)\bigr) - \Bigl[ U(\mathfrak{g},\mathfrak{g}R|1) -U_0(\mathfrak{g},\mathfrak{g}R|1) \Bigr].
\end{eqnarray*}
\item The proof works analogously to the first part. In the following we consider only trajectories $\omega=(w_0=\mathfrak{g}x,w_1,\dots)\in V^{\N_0}$ starting at $\mathfrak{g}x$ with $p(w_i,w_{i+1})>0$ for all $i\in\N_0$, that is, we condition on the event $[X_0=\mathfrak{g}x]$. We claim:
\begin{eqnarray*}
A_1&:= &\bigl\lbrace X_1\notin C(\mathfrak{g}x)\bigr\rbrace \setminus \bigcup_{n\geq 1}\bigl\lbrace S_{\mathfrak{g}R}=n ,\forall m\in\{1,\dots,n-1\}: X_m\notin C(\mathfrak{g}x)\bigr\rbrace \\
&=& \bigl\lbrace S_{\mathfrak{g}R}=\infty ,\forall n\geq 1: X_n\notin C(\mathfrak{g}x)\bigr\rbrace=:B_1.
\end{eqnarray*}
Indeed, if $\omega\in A_1$, then $X_1(\omega) \notin C(\mathfrak{g}x)$. If we would have some $n\in\N$ with $S_{\mathfrak{g}R}(\omega)=n$ and $m\in\{1,\dots,n-1\}$ with $X_m\in C(\mathfrak{g}x)$, then there is some $m'\leq m$ with $X_{m'}=\mathfrak{g}x\in \mathfrak{g}R$, a contradiction to $S_{\mathfrak{g}R}(\omega)=n>m'$. Therefore,  
 we must have $S_{\mathfrak{g}R}(\omega)=\infty$ with  $w_i\notin C(\mathfrak{g}x)$ for all $i\in\N$ (because otherwise we would have $S_{\mathfrak{g}R}(\omega)<\infty$ due to $\mathfrak{g}x\in\mathfrak{g}R$)). Hence, $\omega\in B_1$. Vice versa, if $\omega\in B_1$, then we have $X_1(\omega) \notin C(\mathfrak{g}x)$ and $\omega$ is obviously contained in the set $A_1$.
\par
The above equation of set implies:
\begin{eqnarray*}
&&\P_{\mathfrak{g}x}\bigl[ S_{\mathfrak{g}R}=\infty ,\forall n\geq 1: X_n\notin C(\mathfrak{g}x)\bigr]\\
&=& \P_{\mathfrak{g}x}\bigl[ X_1\notin C(\mathfrak{g}x)\bigr] - \sum_{n\geq 1} \P_{\mathfrak{g}x}\bigl[ S_{\mathfrak{g}R}=n,\forall m\in \{1,\dots,n-1\}: X_m\notin C(\mathfrak{g}x)\bigr] \\
&=& \alpha   - U_1(\mathfrak{g}x,\mathfrak{g}R|1) =\overline{U}_1(\mathfrak{g}x,\mathfrak{g}R|1).
\end{eqnarray*}
%
%
%
%
%
%\begin{eqnarray*}
%&&\P_{\mathfrak{g}x}\bigl[ S_{\mathfrak{g}R}=\infty ,\forall n\geq 1: X_n\notin C(\mathfrak{g}x)\bigr]\\
%&=& \P_{\mathfrak{g}x}\bigl[ \forall n\geq 1: X_n\notin C(\mathfrak{g}x)\bigr] - \sum_{k\geq 1} \P_{\mathfrak{g}x}\bigl[ S_{\mathfrak{g}R}=k,\forall n\geq 1: X_n\notin C(\mathfrak{g}x)\bigr] \\
%&=& \bigl(1-\sum_{y\in V_2} p(gx,gy)\cdot F(gy,gx|1)\bigr) - \Bigl[ \sum_{k\geq 1} \P_{\mathfrak{g}x}\bigl[ S_{\mathfrak{g}R}=k\bigr] - \sum_{m>k} \P_{\mathfrak{g}x}\bigl[ S_{\mathfrak{g}R}=k,X_m\in C(\mathfrak{g}x) ,\forall t<n: X_t\notin C(\mathfrak{g}x)\bigr] \\
%&=& \bigl(1-\sum_{y\in V_2} p(gx,gy)\cdot F(gy,gx|1)\bigr) - \Bigl[ U(\mathfrak{g}x,\mathfrak{g}R|1) -U_1(\mathfrak{g}x,\mathfrak{g}R|1) \Bigr].
%\end{eqnarray*}
\end{enumerate}
\end{proof}

We will rewrite $\mathbb{E}[\widetilde{\mathcal{C}}_1]$ with the help of the above introduced generating functions: 
%$g(x,R|z)$ and $\overline{U}(x,R|z)$, $\overline{U}_0(x,R|z)$ and $\overline{U}_1(x,R|z)$. 
for $z\in\mathbb{C}$, set
\begin{eqnarray*}
\mathcal{E}^{(I)}(z) &:=& \sum_{\substack{R\in \mathcal{W}, \\ x_1\in V_{2,\mathfrak{g}}}} \sum_{x\in R\setminus\{o,x_1\}} \overline{U}(x,R|z) \cdot g(x_1,R|z),\\
\mathcal{E}^{(0)}(z) &:=& \sum_{\substack{R\in \mathcal{W}, \\ x_1\in V_{2,\mathfrak{g}}}} \overline{U}_0(\mathfrak{g},\mathfrak{g}R|z) \cdot g(x_1,R|z),\\
\mathcal{E}^{(1)}(z) &:=& \sum_{\substack{R\in \mathcal{W}, \\ x_1\in V_{2,\mathfrak{g}}}}\  \overline{U}_1(\mathfrak{g}x_1,\mathfrak{g}R|z) \cdot g(x_1,R|z), \\
\mathcal{E}(z) & := & \mathcal{E}^{(I)}(z)+ \mathcal{E}^{(0)}(z)+\mathcal{E}^{(1)}(z).
\end{eqnarray*}
The following proposition plays a key role in the proofs later:
\begin{proposition}\label{prop:C1-decomposition}
$$
\mathbb{E}[\widetilde{\mathcal{C}}_1] = \mathcal{E}(1).
$$
\end{proposition}
\begin{proof}
Recall once again that we can shift paths $\bigl(C(X_{\TT_0}),v_1,\dots,v_m\bigr)\in C(X_{\TT_0})^{m+1}$ in a measure-preserving way to paths in $C(\mathfrak{g})$ by substituting the common prefix $X_{\TT_0}$ with $\mathfrak{g}$; compare with Lemma \ref{lem:cone-probs}.
By (\ref{equ:C1}), we can then rewrite $\mathbb{E}[\widetilde{\mathcal{C}}_1]$ in the following way:
\begin{eqnarray}
&&\mathbb{E}[\widetilde{\mathcal{C}}_1] \nonumber\\
&=& \sum_{\substack{R\in\mathcal{W},\\ x_0\in V_{\mathfrak{g}}, x_1\in V_{2,\mathfrak{g}}}}  \P\bigl[\Rn=R,X_{\TT_0}=x_0,X_{\TT_1}=x_0x_1\bigr] \nonumber \\
&& \cdot 
\biggl(  \sum_{x\in x_0R\setminus\{x_0,x_0x_1\}} \mathbb{P}_x[S_{x_0R}=\infty]+ 
 \P_{x_0}\biggl[\substack{S_{x_0R}=\infty,\\ \forall n\geq 1: X_n\in C(x_0)}\biggr] + \P_{x_0x_1}\biggl[\substack{S_{x_0R}=\infty,\\ \forall n\geq 1: X_n\notin C(x_0x_1)}\biggr] \biggr)\nonumber\\
&\stackrel{\textrm{L.}\ref{lem:cone-probs}}{=}& \sum_{\substack{R\in\mathcal{W},\\ x_0\in V_{\mathfrak{g}}, x_1\in V_{2,\mathfrak{g}}}}  \P\bigl[\Rn=R,X_{\TT_0}=x_0,X_{\TT_1}=x_0x_1\bigr] \nonumber\\
&& \cdot 
\biggl(  \sum_{x\in R\setminus\{o,x_1\}} \mathbb{P}_x[S_R=\infty]+ 
 \P_{\mathfrak{g}}\biggl[\substack{S_{\mathfrak{g}R}=\infty,\\ \forall n\geq 1: X_n\in C(\mathfrak{g})}\biggr] + \P_{\mathfrak{g}x_1}\biggl[\substack{S_{\mathfrak{g}R}=\infty,\\ \forall n\geq 1: X_n\notin C(\mathfrak{g}x_1)}\biggr] \biggr)\nonumber\\
&=& \sum_{\substack{R\in\mathcal{W},\\ x_1\in V_{2,\mathfrak{g}}}}  \P\bigl[\Rn=R,\ii=x_1\bigr] \nonumber\\
&& \cdot 
\biggl(  \sum_{x\in R\setminus\{o,x_1\}} \mathbb{P}_x[S_R=\infty]+ 
\P_{\mathfrak{g}}\biggl[\substack{S_{\mathfrak{g}R}=\infty,\\ \forall n\geq 1: X_n\in C(\mathfrak{g})}\biggr] + \P_{\mathfrak{g}x_1}\biggl[\substack{S_{\mathfrak{g}R}=\infty,\\ \forall n\geq 1: X_n\notin C(\mathfrak{g}x_1)}\biggr] \biggr)\nonumber\\
&=& \sum_{\substack{R\in\mathcal{W},\\ x_1\in V_{2,\mathfrak{g}}}}  \sum_{m\geq 1}\P\bigl[\Rn=R,\ii=x_1,\TT_1-\TT_0=m\bigr]
 \nonumber \\
&&\quad \cdot \biggl(  \sum_{x\in R\setminus\{o,x_1\}}\overline{U}(x,R|1)+ \overline{U}_0(\mathfrak{g},\mathfrak{g}R|1) + \overline{U}_1(\mathfrak{g}x_1,\mathfrak{g}R|1)  \biggr)\nonumber\\
%
%
%
%
%&=& \sum_{\substack{R\subset V \textrm{ finite},\\ x_0,x_1\in V}} \ g(x_0,x_1,R|z)\cdot \biggl[   \sum_{x\in R\setminus\{x_0,x_1\}} \overline{U}(x,R|z) + \overline{U}_0(x,R|z)+ \overline{U}_1(x,R|z)\biggr]\biggl|_{z=1}\nonumber \\
&=&   \mathcal{E}^{(I)}(1)+ \mathcal{E}^{(0)}(1)+ \mathcal{E}^{(1)}(1)= \mathcal{E}(1). \label{equ:E-formula}
\end{eqnarray}
\end{proof}
Set
\begin{eqnarray*}
\mathcal{E}^{(I)}_0(z)&:=& \sum_{\substack{R\in\mathcal{W},\\ x_1\in V_{2,\mathfrak{g}}}}\  \sum_{x\in R\setminus\{o,x_1\}} U(x,R|z) g(x_1,R|z) \quad \textrm{ and }\\ 
\mathcal{E}_i^\ast(z)&:=&\sum_{\substack{R\in\mathcal{W},\\ x_1\in V_{2,\mathfrak{g}}}}\  (|R|-2)^i  g(x_1,R|z) \quad \textrm{ for } i\in\{0,1\}.
\end{eqnarray*}
In the following proofs we will make use of the observation that
 $\TT_1-\TT_0=m\in\N$ implies $|\Rn|=|\widetilde{\calR}_1|\leq m+1 \leq 2m$.
We have:
\begin{lemma}\label{lem:E-ast}
The power series $\mathcal{E}_i^\ast(z)$, $i\in\{0,1\}$, have radii of convergence strictly bigger than $1$.
\end{lemma}
\begin{proof}
Let  $z\in(0,r_1)$. Then:
\begin{eqnarray*}
\mathcal{E}_1^\ast(z)&=&\sum_{\substack{R\in\mathcal{W},\\ x_1\in V_{2,\mathfrak{g}}}}\  (|R|-2) \cdot  g(x_1,R|z)  
\leq  \sum_{\substack{R\in\mathcal{W},\\ x_1\in V_{2,\mathfrak{g}}}} |R|\cdot  g(x_1,R|z)  \\
&= &  \sum_{\substack{R\in\mathcal{W},\\ x_1\in V_{2,\mathfrak{g}}}}\ \sum_{m\geq 1} |R| \cdot \P\Bigl[\substack{\Rn=R,\ii=x_1, \\ \TT_1-\TT_0=m}\Bigr]\cdot z^m
\leq    \sum_{\substack{R\in\mathcal{W},\\ x_1\in V_{2,\mathfrak{g}}}}\ \sum_{m\geq 1} 2m \cdot \P\Bigl[\substack{\Rn=R,\ii=x_1, \\ \TT_1-\TT_0=m}\Bigr]\cdot z^m\\
&=& 2\cdot  \sum_{m\geq 1} m \cdot \P\bigl[\TT_1-\TT_0=m\bigr]z^m 
= 2 z\cdot \mathbb{T}'(z)<\infty,
\end{eqnarray*}
where finiteness follows from existence of exponential moments of $\TT_1-\TT_0$ (see Proposition \ref{prop:TT-exp-moments}). Thus, $\mathcal{E}_1^\ast(z)$ has radius of convergence strictly bigger than $1$. The same holds for $\mathcal{E}_0^\ast(z)$ since $|\widetilde{\mathcal{R}}_{\mathrm{norm}}|\geq 2$ almost surely and  
$$
\mathcal{E}_0^\ast(z) = \sum_{\substack{R\in\mathcal{W},\\ x_1\in V_{2,\mathfrak{g}}}}\    g(x_1,R|z) \leq \sum_{\substack{R\in\mathcal{W},\\ x_1\in V_{2,\mathfrak{g}}}}\  |R| \cdot  g(x_1,R|z) <\infty.
$$
\end{proof}
Now we are able to prove:
\begin{proposition}\label{prop:E(I)}
$\mathcal{E}^{(I)}(z)$ has radius of convergence strictly bigger than $1$.
\end{proposition}
\begin{proof}
Let  $z\in (0,r_1)$.
By Proposition \ref{prop:U-convergence2}, there is some constant $M_0$ such that, for all finite $R\subset V$ and every $x\in R$, $U(x,R|z)\leq |R|\cdot M_0$. 
This yields:
\begin{eqnarray*}
\mathcal{E}^{(I)}_0(z)&=&\sum_{\substack{R\in\mathcal{W},\\ x_1\in V_{2,\mathfrak{g}}}}\  \sum_{x\in R\setminus\{o,x_1\}} U(x,R|z) g(x_1,R|z)  \\
&\leq& \sum_{\substack{R\in\mathcal{W},\\ x_1\in V_{2,\mathfrak{g}}}}\  \sum_{x\in R\setminus\{o,x_1\}} |R|\cdot M_0 \cdot  g(x_1,R|z)\\ 
&\leq & M_0\cdot  \sum_{\substack{R\in\mathcal{W},\\ x_1\in V_{2,\mathfrak{g}}}} |R|^2\cdot  g(x_1,R|z)\\
&\leq & M_0\cdot  \sum_{\substack{R\in\mathcal{W},\\ x_1\in V_{2,\mathfrak{g}}}}\ \sum_{m\geq 1} |R|^2 \cdot \P\Bigl[\substack{\Rn=R,\ii=x_1, \\ \TT_1-\TT_0=m}\Bigr]\cdot z^m\\
&\leq &  4M_0\cdot \sum_{\substack{R\in\mathcal{W},\\ x_1\in V_{2,\mathfrak{g}}}}\ \sum_{m\geq 1} m^2 \cdot \P\Bigl[\substack{\Rn=R,\ii=x_1, \\ \TT_1-\TT_0=m}\Bigr]\cdot z^m\\
&=& 4M_0\cdot  \sum_{m\geq 1} m^2 \cdot \P\bigl[\TT_1-\TT_0=m\bigr]\cdot z^m \\
&=& 4M_0z^2\cdot \mathbb{T}''(z) + 4M_0z\cdot \mathbb{T}'(z) <\infty,
%\frac{\partial^2}{\partial^2 z} \biggl[ \sum_{m\geq 0} \P\bigl[\TT_1-\TT_0=m\bigr]z^m\biggr]\Biggr|_{z=1}+z\cdot \frac{\partial}{\partial z} \biggl[ \sum_{m\geq 0} \P\bigl[\TT_1-\TT_0=m\bigr]z^m\biggr]\Biggr|_{z=1}.
\end{eqnarray*}
where finiteness follows once again from existence of exponential moments of $\TT_1-\TT_0$. Hence, $\mathcal{E}^{(I)}_0(z)$ has radius of convergence strictly bigger than $1$. 
Finally, we get the proposed claim with Lemma \ref{lem:E-ast} due the equation
$$
\mathcal{E}^{(I)}(z)=\mathcal{E}_1^\ast(z)-\mathcal{E}^{(I)}_0(z).
$$
%
%and since $\sum_{R\subset V \textrm{ finite}}\  \sum_{x\in R} U(x,R|z) g(R|z)$ has only non-negative coefficients,  the power series $\mathcal{E}(z)$ has also radius of convergence strictly bigger than $1$. This proves the claim.
\end{proof}
Furthermore:
\begin{proposition}\label{prop:E0-E1}
$\mathcal{E}^{(0)}(z)$ and $\mathcal{E}^{(1)}(z)$ have radii of convergence strictly bigger than $1$.
\end{proposition}
\begin{proof}
The proof works analogously to Lemma \ref{lem:E-ast}. By Proposition \ref{prop:U-convergence2}, we have $U_0(x,R|z)\leq U(x,R|z)\leq |R|\cdot M_0$.
Let  $z\in (0,r_1)$. Then:
\begin{eqnarray*}
\mathcal{E}^{(0)}_0(z)&:=&\sum_{\substack{R\in\mathcal{W},\\ x_1\in V_{2,\mathfrak{g}}}} \  U_0(\mathfrak{g},\mathfrak{g}R|z) \cdot   g(x_1,R|z)  
\leq  \sum_{\substack{R\in\mathcal{W},\\ x_1\in V_{2,\mathfrak{g}}}} |R|\cdot M_0\cdot  g(x_1,R|z)  \\
&= &  M_0\cdot \sum_{\substack{R\in\mathcal{W},\\ x_1\in V_{2,\mathfrak{g}}}}\ \sum_{m\geq 1} |R| \cdot \P\Bigl[\substack{\Rn=R,\ii=x_1, \\ \TT_1-\TT_0=m}\Bigr]\cdot z^m\\
&\leq &   M_0\cdot \sum_{\substack{R\in\mathcal{W},\\ x_1\in V_{2,\mathfrak{g}}}}\ \sum_{m\geq 1} 2m \cdot \P\Bigl[\substack{\Rn=R,\ii=x_1, \\ \TT_1-\TT_0=m}\Bigr]\cdot z^m\\
&=& 2M_0\cdot  \sum_{m\geq 1} m \cdot \P\bigl[\TT_1-\TT_0=m\bigr]z^m 
= 2M_0 z\cdot \mathbb{T}'(z)<\infty.
\end{eqnarray*}
Since $\mathcal{E}^{(0)}(z)= (1-\alpha) z  \mathcal{E}_0^\ast(z)-\mathcal{E}^{(0)}_0(z)$ and by Lemma \ref{lem:E-ast}, we have shown that $\mathcal{E}^{(0)}(z)$ has radius of convergence strictly bigger than $1$.
\par
Moreover, since
$$
\mathcal{E}^{(1)}_0(z):=\sum_{\substack{R\in\mathcal{W},\\ x_1\in V_{2,\mathfrak{g}}}} \  U_1(\mathfrak{g}x_1,\mathfrak{g}R|z)\cdot    g(x_1,R|z)  
\leq  M_0\cdot \sum_{\substack{R\in\mathcal{W},\\ x_1\in V_{2,\mathfrak{g}}}} |R|\cdot  g(x_1,R|z)<\infty,
$$
the same calculus as above and  Lemma \ref{lem:E-ast} show that \mbox{$\mathcal{E}^{(1)}(z)=\alpha z \mathcal{E}_0^\ast(z)-\mathcal{E}^{(1)}_0(z)$} has also radius of convergence strictly bigger than $1$.
\end{proof}

\begin{corollary}\label{cor:E-radius-convergence}
$\mathcal{E}(z)$ has radius of convergence strictly bigger than $1$.
\end{corollary}
\begin{proof}
This follows immediately from Propositions \ref{prop:E(I)}, \ref{prop:E0-E1} and by definition of $\mathcal{E}(z)$. 
\end{proof}
Now we have constructed a power series $\mathcal{E}(z)$ having radius of convergence strictly bigger than $1$ and satisfying $\mathcal{E}(1)=\mathbb{E}[\widetilde{\mathcal{C}}_1]$. It remains to show that the coefficients of $z^m$, $m\in\mathbb{N}$, in $\mathcal{E}(z)$ have the form as requested in Remark \ref{rem:procedure-analyticity}, that is, 
the coefficients of $z^m$ are sums of monomials of the form $a(n_1,\dots,n_d) \cdot p_1^{n_1}\cdot \ldots \cdot p_d^{n_d}$ with \mbox{$n_1,\dots,n_d,a(n_1,\dots,n_d)\in\mathbb{N}_0$} and $n_1+\ldots + n_d=m$. Once this is shown, we have proven analyticity of $\mathbb{E}[\widetilde{\mathcal{C}}_1]$ in $(p_1,\dots,p_d)$ according to Remark  \ref{rem:procedure-analyticity}.
%
%
%
%Since the last power series has radius of convergence strictly bigger than $1$, the expectation $\mathbb{E}[\widetilde{\mathcal{C}}_1]$ varies real-analytically in the random walk parameters $(p_1,\dots,p_d)$.
%

\begin{proposition}\label{prop:multivariate1}
For all $R\in\mathcal{W}$, $x\in V_{2,\mathfrak{g}}$ and $m\in\N$, the probability \mbox{$\P\bigl[\substack{\Rn=R,\ii=x,\\ \TT_1-\TT_0=m}\bigr]$} can be rewritten in the form
$$
\P\Bigl[\substack{\Rn=R,\ii=x,\\ \TT_1-\TT_0=m}\Bigr] = \sum_{\substack{n_1,\dots,n_d\in\mathbb{N}: \\ n_1+\ldots + n_d=m}} a_{R,x,m}(n_1,\dots,n_d)\cdot p_1^{n_1}\cdot \ldots \cdot p_d^{n_d}, \quad a_{R,x,m}(n_1,\dots,n_d)\in\N_0.
$$
%varies real-analytically in $(p_1,\dots,p_d)$ in a neighbourhood of any $\underline{p}_0\in\mathcal{P}$.
\end{proposition}
\begin{proof}
%Let be $R\subset V$ finite and $m\in \N$. 
%%
%%First, observe that we can rewrite $\P[\widetilde{\calR}_1=R]$ as a univariate power series in $z$ evaluated at $z=1$:
%%$$
%%\P[\widetilde{\calR}_1=R] = \sum_{m\geq 1} \P\bigl[\widetilde{\calR}_1=R,\TT_1-\TT_0=m\bigr]z^m \biggl|_{z=1}.
%%$$
%%The next goal is to show that the probabilities $\P\bigl[\widetilde{\calR}_1=R,\TT_1-\TT_0=m\bigr]$, where $m\in\mathbb{N}$ and $R\subset V$ finite, can be rewritten as multivariate polynomials in $p_1,\dots,p_d$ of degree $m$.  
%Denote by 
%$$
%V_{\mathfrak{g}}:=\bigl\lbrace v_1\dots v_k\in V \mid k\in\N, v_k=\mathfrak{g},v_1,\dots,v_{k-1}\neq \mathfrak{g}\bigr\rbrace,
%$$ 
%the set of words in $V$ which end up with the letter $\mathfrak{g}$ with no further occurrence of this letter except at the end; these are the possible values of $X_{\TT_0}$. 
For $R\in\mathcal{W}$, $x_0\in V_{\mathfrak{g}}$, $x\in V_{2,\mathfrak{g}}$ and $m\in\mathbb{N}$, write $\mathcal{Q}_{x_0,x,R,m}$ for the set of paths $(x_0,w_1,\dots,w_m=x_0x)\in V^{m+1}$ such that 
$$
\mathbb{P}\bigl[ X_{\TT_0}=x_0,X_{\TT_0+1}=w_1,\dots,X_{\TT_0+m}=w_m,\mathcal{R}_1=x_0R,\TT_1-\TT_0=m\bigr]>0.
$$
%If $x_1\in C(x_0)$ has the form $x_1=x_0v_1\dots v_k$ in the sense of (\ref{equ:free-product}), then we set $\bar x_1:=\mathfrak{g}v_1\dots v_k$.
By decomposing according to the values of $\TT_0$ and $X_{\TT_0}$ we obtain with Lemma \ref{lem:cone-probs} (by shifting paths inside $C(w), w\in V_{\mathfrak{g}}$ to paths inside $C(\mathfrak{g})$ by replacing the common prefix $w$ with $\mathfrak{g}$):
\begin{eqnarray*}
&&\P\bigl[\Rn=R,\ii=x,\TT_1-\TT_0=m\bigr] \\
&=&\sum_{x_0\in V_{\mathfrak{g}}} \sum_{k\geq 1} \P\bigl[\widetilde{\calR}_1=x_0R,\TT_0=k, \TT_1-\TT_0=m, X_{\TT_0}=x_0,X_{\TT_1}=x_0x\bigr]  \\
%&=& \sum_{k\geq 1} \P\bigl[\widetilde{\calR}_1=R,\TT_1-\TT_0=m,\TT_0=k,X_k=x_0,X_{k+m}=x_1\bigr]\\
&=& \sum_{x_0\in V_{\mathfrak{g}}} \sum_{k\geq 1}  \sum_{(x_0,w_1,\dots,w_m)\in Q_{x_0,x,R,m}} \mathbb{P}\left[ \substack{X_k=x_0, \forall j<k: X_j\notin C(x_0)\\ X_{k+1}=w_1,\dots,X_{k+m-1}=w_{m-1},X_{k+m}=x_0x,\\ \forall t\geq 1: X_{k+m+t}\in C(x_0x)}\right] \\
&=& \sum_{x_0\in V_{\mathfrak{g}}}  \sum_{k\geq 1}  \mathbb{P}\bigl[  X_k=x_0, \forall j<k: X_j\notin C(x_0)\bigr] \\
&& \quad \cdot  \sum_{(x_0,w_1,\dots,w_m)\in Q_{x_0,x,R,m}} \mathbb{P}_{x_0}\bigl[ X_{1}=w_1,\dots,X_{m-1}=w_{m-1},X_{m}=x_0x\bigr]\\
&&\quad \cdot \underbrace{\mathbb{P}_{x_0x} \bigl[ \forall t\geq 1: X_{t}\in C(x_0x)\bigr]}_{=(1-\xi_1) \textrm{ (since $\delta(x_0x)=\delta(\mathfrak{g})=1$)}} \\
&\stackrel{\textrm{L.\ref{lem:cone-probs}}}{=}& \underbrace{\sum_{x_0\in V_{\mathfrak{g}}} \sum_{k\geq 1}  \underbrace{ \mathbb{P}\bigl[  X_k=x_0, \forall j<k: X_j\notin C(x_0)\bigr] \cdot (1-\xi_1)}_{=\P[X_{\TT_0}=x_0,\TT_0=k]}}_{=\mathbb{P}[\TT_0<\infty]=1} \\
&&\quad \cdot \sum_{(\mathfrak{g},w_1,\dots,w_m)\in Q_{\mathfrak{g}, x,R,m}} \mathbb{P}_\mathfrak{g}\bigl[ X_{1}=w_1,\dots,X_{m-1}=w_{m-1},X_{m}=w_m\bigr]\\
&=& \sum_{(\mathfrak{g},w_1,\dots,w_m)\in Q_{\mathfrak{g}, x,R,m}} \mathbb{P}_\mathfrak{g}\bigl[ X_{1}=w_1,\dots,X_{m-1}=w_{m-1},X_{m}=w_m\bigr].
%&=& \sum_{k\geq 1} \sum_{x\in V: \mathbb{P}[X_{k}=x,\TT_0=k]>0} \sum_{x_1,\dots,x_{k-1}\in V\setminus C(x)}  \sum_{y\in V: \mathbb{P}[X_{k+m}=y,\TT_1=k+m]>0} \sum_{y_1,\dots,y_{m-1}\in C(x)\setminus C(y)}
% \mathbb{P}\left[ \begin{array}{c} X_1=x_1,\dots,X_{k-1}=x_{k-1},X_k=x,\\ X_{k+1}=y_1,\dots,X_{k+m-1}=y_{k-1},X_{k+m}=y,\\ \forall t\geq 1: X_{k+m+t}\in C(y)\end{array}\right] \\
% &=&  \sum_{k\geq 1} \sum_{x\in V: \mathbb{P}[X_{k}=x,\TT_0=k]>0} \sum_{x_1,\dots,x_{k-1}\in V\setminus C(x)}  \sum_{y\in V: \mathbb{P}[X_{k+m}=y,\TT_1=k+m]>0} \sum_{y_1,\dots,y_{m-1}\in C(x)\setminus C(y)}  \mathbb{P}\bigl[ X_1=x_1,\dots,X_{k-1}=x_{k-1},X_k=x\bigr] \cdot \mathbb{P}_x\bigl[ X_1=y_1,\dots,X_{m-1}=y_{m-1},X_m=y\bigr] \cdot \underbrace{\mathbb{P}_y \bigl[ \forall t\geq 1: X_t\in C(y)\bigr]}_{= (1-\xi_1)}\\
% &=& \sum_{k\geq 1} \sum_{x\in V: \mathbb{P}[X_{k}=x,\TT_0=k]>0} \sum_{x_1,\dots,x_{k-1}\in V\setminus C(x)}  \mathbb{P}\bigl[ X_1=x_1,\dots,X_{k-1}=x_{k-1},X_k=x\bigr] \cdot (1-\xi_1)\\
% && \quad \cdot 
\end{eqnarray*}
In the latter sum, each summand is obviously a product of factors $p_1,\dots,p_d$ of degree $m$. That is, we have shown that $\P\bigl[\Rn=R,\ii=x,\TT_1-\TT_0=m\bigr]$ can be rewritten as a sum of multivariate monomials in $p_1,\dots,p_d$ of degree $m$.
%, say
%$$
%\P\bigl[\widetilde{\calR}_1=R,\TT_1-\TT_0=m\bigr] = \sum_{\substack{n_1,\dots,n_d\in\mathbb{N}: \\ n_1+\ldots + n_d=m}} a_{R,m}(n_1,\dots,n_d) p_1^{n_1}\cdot \ldots \cdot p_d^{n_d}.
%$$
%This yields for $\delta>0$ small enough:
%\begin{eqnarray*}
%\infty& >&\sum_{m\geq 1} \P\bigl[\widetilde{\calR}_1=R,\TT_1-\TT_0=m\bigr] (1+\delta)^m \\
%&=& \sum_{m\geq 1} \sum_{\substack{n_1,\dots,n_d\in\mathbb{N}: \\ n_1+\ldots + n_d=m}} a_{R,m}(n_1,\dots,n_d) p_1^{n_1}\cdot \ldots \cdot p_d^{n_d}(1+\delta)^m \\
%&=& \sum_{m\geq 1} \sum_{\substack{n_1,\dots,n_d\in\mathbb{N}: \\ n_1+\ldots + n_d=m}} a_{R,m}(n_1,\dots,n_d) \bigl(p_1(1+\delta)\bigr)^{n_1}\cdot \ldots \cdot \bigl(p_d(1+\delta)\bigr)^{n_d}.
%\end{eqnarray*}
%Therefore, we have shown that $\P[\widetilde{\calR}_1=R]$ varies real-analytically in $(p_1,\dots,p_d)$.
\end{proof}
%
%Finally, we can show the remaining ingredients for the proof of analyticity of $\mathbb{E}[\widetilde{\mathcal{C}}_1]$:
\begin{lemma}\label{lem:multivariate1}
The coefficient of $z^m$, $m\in\N$, in $\mathcal{E}^{(I)}_0(z)$ can be rewritten as a sum of multivariate monomials in $(p_1,\dots,p_d)$ of degree $m$.
\end{lemma}
\begin{proof}
Let  $m\in\N$. Then the coefficient of $z^m$ in $\mathcal{E}^{(I)}_0(z)$ is given by
\begin{equation}\label{equ:coefficient}
\sum_{\substack{R\in\mathcal{W},\\ x_1\in V_{2,\mathfrak{g}}}}\ \sum_{x\in R\setminus\{o,x_1\}} \sum_{k=1}^{m-1} \P_x[S_R=k]\cdot \P\Bigl[\substack{\Rn=R, \ii=x_1, \\ \TT_1-\TT_0=m-k}\Bigr].
\end{equation}
Observe that, for all $k\in\N$, $\P_x[S_R=k]$ can obviously be written in the form
$$
\sum_{n_1,\dots,n_d\in\N: n_1+\dots +n_d=k} b(n_1,\dots,n_d)\cdot  p_1^{n_1} \cdot \ldots \cdot p_d^{n_d},
$$
where $b(n_1,\dots,n_d)\in\mathbb{N}_0$. By Proposition \ref{prop:multivariate1}, the same representation as a sum of multivariate monomials of degree $m-k$ holds for the probability  \mbox{$\P_x\Bigl[\substack{\Rn=R,\ii=x_1,\\ \TT_1-\TT_0=m-k}\Bigr]$} for all  $R\in\mathcal{W}$, $x_1\in V_{2,\mathfrak{g}}$ and $k\in\{1,\dots,m-1\}$. In view of (\ref{equ:coefficient}) the claim is proven.
\end{proof}
\begin{lemma}\label{lem:multivariate2}
The coefficient of $z^m$, $m\in\N$, in $\mathcal{E}_i^\ast(z)$, $i\in\{0,1\}$, can be rewritten as a sum of multivariate monomials in $(p_1,\dots,p_d)$ of degree $m$.
\end{lemma}
\begin{proof}
The coefficient of $z^m$, $m\in\mathbb{N}$, in $\mathcal{E}_i^\ast(z)$, $i\in\{0,1\}$, is given by
$$
\sum_{\substack{R\in\mathcal{W},\\ x_1\in V_{2,\mathfrak{g}}}} (|R|-2)^i \cdot \P_x\Bigl[\Rn=R, \ii=x_1, \\ \TT_1-\TT_0=m\Bigr].
$$
The claim follows now immediately from Proposition \ref{prop:multivariate1}. 
%and the fact that $\mathcal{E}_i^\ast(z)$ has radius of convergence strictly bigger than $1$, which was shown in Lemma \ref{lem:E-ast}.
\end{proof}

\begin{lemma}\label{lem:multivariate3}
The coefficient of $z^m$, $m\in\N$, in $\mathcal{E}^{(0)}(z)$ can be rewritten as a sum of multivariate monomials in $(p_1,\dots,p_d)$ of degree $m$.
\end{lemma}
\begin{proof}
For $z\in\mathbb{C}$, set 
$$
\mathcal{E}^{(0)}_0(z):=\sum_{\substack{R\in\mathcal{W},\\ x_1\in V_{2,\mathfrak{g}}}} \  U_0(\mathfrak{g},\mathfrak{g}R|z) \cdot   g(x_1,R|z).  
$$
Then $\mathcal{E}^{(0)}(z)=(1-\alpha)z \mathcal{E}_0^\ast(z)-\mathcal{E}^{(0)}_0(z)$. By Lemma \ref{lem:multivariate2}, it suffices to show that the coefficients of $\mathcal{E}^{(0)}_0(z)$ have the requested form. The coefficient of $z^m$, $m\in\N$, in $\mathcal{E}^{(0)}_0(z)$ is given by
\begin{equation}
\sum_{\substack{R\in\mathcal{W},\\ x_1\in V_{2,\mathfrak{g}}}}\ \sum_{k=1}^{m-1} \P_{\mathfrak{g}}\bigl[S_{\mathfrak{g}R}=k,\forall m\leq k: X_m\in C(\mathfrak{g})\bigr]\cdot \P\Bigl[\substack{\Rn=R,\ii=x_1,\\ \TT_1-\TT_0=m-k}\Bigr].
\end{equation}
Obviously, for all $k\in\N$, $\P_{\mathfrak{g}}\bigl[S_{\mathfrak{g}R}=k,\forall m\leq k: X_m\in C(\mathfrak{g})\bigr]$ can  be written in the form
$$
\sum_{n_1,\dots,n_d\in\N: n_1+\dots +n_d=k} b(n_1,\dots,n_d)\cdot  p_1^{n_1} \cdot \ldots \cdot p_d^{n_d}, \quad b(n_1,\dots,n_d)\in\N_0.
$$
The claim follows now directly with Proposition \ref{prop:multivariate1} and the fact that $1-\alpha$ is a sum of some values out of $p_1,\dots, p_d$. 
%and the fact that $\mathcal{E}^{(0)}_0(z)$ has radius of convergence strictly bigger than $1$.
\end{proof}

\begin{lemma}\label{lem:multivariate4}
The coefficient of $z^m$, $m\in\N$, in $\mathcal{E}^{(1)}(z)$ can be rewritten as a sum of multivariate monomials in $(p_1,\dots,p_d)$ of degree $m$.
\end{lemma}
\begin{proof}
For $z\in\mathbb{C}$, set  
$$
\mathcal{E}^{(1)}_0(z):=\sum_{\substack{R\in\mathcal{W},\\ x_1\in V_{2,\mathfrak{g}}}} \  U_1(\mathfrak{g}x_1,\mathfrak{g}R|z) \cdot   g(x_1,R|z).  
$$
Then $\mathcal{E}^{(1)}(z)=\alpha z \mathcal{E}_0^\ast(z)-\mathcal{E}^{(1)}_0(z)$. By Lemma \ref{lem:multivariate2}, it suffices to show  that the coefficients of $\mathcal{E}^{(1)}_0(z)$ have the requested form. The coefficient of $z^m$, $m\in\N$,  in $\mathcal{E}^{(1)}_0(z)$ is given by
\begin{equation}
\sum_{\substack{R\in\mathcal{W},\\ x_1\in V_{2,\mathfrak{g}}}}\ \sum_{k=1}^{m-1} \P_{\mathfrak{g}x_1}\bigl[S_{\mathfrak{g}R}=k,\forall m\in\{1,\dots,k-1\}: X_m\notin C(\mathfrak{g}x_1)\bigr]\cdot \P\Bigl[\substack{\Rn=R,\ii=x_1,\\ \TT_1-\TT_0=m-k}\Bigr].
\end{equation}
Obviously, for all $k\in\N$, $\P_{\mathfrak{g}x_1}\bigl[S_{\mathfrak{g}R}=k,\forall m\in\{1,\dots,k-1\}: X_k\notin C(\mathfrak{g}x_1)\bigr]$ can  be written in the form
$$
\sum_{n_1,\dots,n_d\in\N: n_1+\dots +n_d=k} b(n_1,\dots,n_d)\cdot  p_1^{n_1} \cdot \ldots \cdot p_d^{n_d}, \quad b(n_1,\dots,n_d)\in\N_0.
$$
The claim follows now directly with Proposition \ref{prop:multivariate1} and the fact that $\alpha$ is a sum of some values out of $p_1,\dots,p_d$.
%$\mathcal{E}^{(1)}_0(z)$ has radius of convergence strictly bigger than $1$.
\end{proof}
Finally, we have proven the remaining part that the coefficients of $\mathcal{E}(z)$ have the  form as requested in Remark \ref{rem:procedure-analyticity}:
\begin{proposition}\label{prop:multivariate}
The coefficient of $z^m$, $m\in\N$, in $\mathcal{E}(z)$ can be rewritten as a sum of multivariate monomials in $(p_1,\dots,p_d)$ of degree $m$.
\end{proposition}
\begin{proof}
The claim follows now immediately from Lemmas \ref{lem:multivariate1}, \ref{lem:multivariate2}, \ref{lem:multivariate3} and \ref{lem:multivariate4}.
\end{proof}

% Therefore, for $\delta>0$ small enough,
%\begin{eqnarray*}
%&&\sum_{m\geq 0}\sum_{R\subset V \textrm{finite}}\ \sum_{x\in R}\sum_{k=1}^m \P_x[S_R=k]\cdot \P_x\bigl[\widetilde{\calR}_1=R,\TT_1-\TT_0=m-k\bigr]\cdot (1+\delta)^m \\
%&=& \sum_{m\geq 0} \sum_{\substack{n_1,\dots,n_d\in\N: \\ n_1+\dots +n_d=m}} A(n_1,\dots,n_d) \bigl(p_1(1+\delta)\bigr)^{n_1} \cdot \ldots \cdot \bigl(p_d(1+\delta)\bigr)^{n_d}
%\end{eqnarray*}
%for suitable $A(n_1,\dots,n_d)\in\mathbb{N}_0$.
%This proves that the mapping
%$$
%\mathcal{P}\ni (p_1,\dots,p_d) \mapsto \mathbb{E}\bigl[\widetilde{\mathcal{C}}_1\bigr]
%$$
%varies real-analytically in $(p_1,\dots,p_d)$.
\begin{proof}[Proof of Theorem \ref{th:analyticity}]
In view of formulas (\ref{equ:c-formula}) and (\ref{equ:E-formula}) the claim follows now directly with Lemma \ref{lem:T1-T0-analytic}, Corollary \ref{cor:E-radius-convergence} and Proposition \ref{prop:multivariate} together with the explanations in Remark \ref{rem:procedure-analyticity}. 
%
%
%Moreover, in \cite{gilch:22} it is shown that $\mathbb{E}[\TT_1-\TT_0]$ varies also real-anayltically in $(p_1,\dots,p_d)$. In view of formula (\ref{equ:c-formula}), we have shown that the mapping
%$$
%\mathfrak{c}: \mathcal{P} \to [0,1], (p_1,\dots,p_d) \mapsto \mathfrak{c}=\mathfrak{c}(p_1,\dots,p_d)
%$$
%varies real-analytically. This proves Theorem \ref{th:analyticity}.
\end{proof}

%\begin{remark}
%For random walks on groups beyond free products of groups, it is unknown whether the asymptotic capacity of the range varies real-analytically. In particular, this question is open for the well-studied case of $\mathbb{Z}^d$.
%\end{remark}

\subsection{Open Problems}

Finally, we formulate some open questions in connection with the analyticity results in \mbox{Section \ref{sec:analyticity}:}
\begin{remark}\ 
\begin{itemize}
\item Does the asymptotic capacity vary real-analytically for nearest-neighbour (or finite-range) random walks on $\mathbb{Z}^d$ for $d\geq 2$?
\item For which classes of  finitely generated groups do finite-range random walks vary real-analytically in terms of probability measures of constant support?
\end{itemize}
\end{remark}

\section*{Acknowledgement}
The author is grateful to both anonymous referees for their suggestions and hints regarding content and exposition, and also for the proposed simplification in the  proof of Theorem 1.1.

%%%%%%%%%%%%%%%%%%%%%%%%%%%%%%%%%%%%%%%%%%%%%%%%%%%%%%%%%%%%%%%%%%%
%%                                                               %%
%% Use the two commands below for producing your bibliography    %%
%% with bibtex, then comment again the commands and include the  %%
%% content of the .bbl file in this file below the commands.     %%
%%                                                               %%
%%%%%%%%%%%%%%%%%%%%%%%%%%%%%%%%%%%%%%%%%%%%%%%%%%%%%%%%%%%%%%%%%%%

%\bibliographystyle{amsplain}
%\bibliography{literatur}

% add below the content of your .bbl file produced by bibtex.

\end{document}